\documentclass[10pt, reqno]{amsart}
\usepackage[pxpazo, numberbasedon={subsection}]{zzzams}
\usepackage{bbm} 
\usepackage{dynkin-diagrams}

\usepackage{float}
\usepackage[nomarkers, notables, nolists]{endfloat} 

\DeclareMathOperator{\vac}{|0\rangle}
\DeclareMathOperator{\Bh}{\mathbbm{h}}

\DeclareMathOperator{\Sat}{Sat}
\DeclareMathOperator{\Irr}{Irr}
\DeclareMathOperator{\RTmin}{RTmin}
\DeclareMathOperator{\RT}{RT}
\DeclareMathOperator{\KL}{KL}
\DeclareMathOperator{\charp}{char}
\newcommand{\ubar}[1]{\underline{#1}}
\DeclareMathOperator{\cl}{\mathbf{cl}}
\DeclareMathOperator{\AV}{AV}
\DeclareMathOperator{\Stab}{Stab}
\DeclareMathOperator{\Gal}{Gal}
\DeclareMathOperator{\val}{val}

\DeclareMathOperator{\Rep}{Rep}

\DeclareMathOperator{\Mat}{Mat}

\newcommand{\xeq}[1]{\overset{\mathrm{#1}}{=\joinrel=}}
\newcommand{\twi}[1]{\textcolor{white}{#1}}
 
\newcommand{\laurent}[1]{(\!({\ensuremath{#1}})\!)}
\newcommand{\series}[1]{[\![{\ensuremath{#1}}]\!]}

\newsavebox{\sembox}

\title[Cyclotomic levels and associated varieties]{Cyclotomic level maps and associated varieties of simple affine vertex algebras}
\author{Peng Shan$^{1,2}$}
\author{Wenbin Yan$^1$} 
\author{Qixian Zhao$^1$}

\address{\scriptsize{$^1$} Yau Mathematical Sciences Center, Tsinghua University, Beijing, 100084, China}
\address{\scriptsize{$^2$} Department of Mathematical Sciences, Tsinghua University, Beijing, 100084, China}

\begin{document}

\begin{abstract}
	In this paper, we introduce and study two cyclotomic level maps defined respectively on the set of nilpotent orbits $\ubar \cN$ in a complex semi-simple Lie algebra $\fg$ and the set of conjugacy classes $\ubar W$ in its Weyl group, with values in positive integers. We show that these maps are compatible under Lusztig's map $\ubar W \to \ubar \cN$, which is also the minimal reduction type map as shown by Yun. We also discuss their relationship with two-sided cells in affine Weyl groups.
	
	We use these maps to formulate a conjecture on the associated varieties of simple affine vertex algebras attached to $\fg$ at non-admissible integer levels, and provide some evidence for this conjecture.
\end{abstract}

\maketitle

\tableofcontents

\section{Introduction}

Let $\fg$ be a finite dimensional semi-simple complex Lie algebra. For any complex number $k$, the simple affine vertex algebra $L_k= L_k(\fg)$ is defined as the simple quotient of the universal affine vertex algebra $V^k(\fg)$ at level $k$. They form an important family of vertex algebras which has a close relationship with representations of affine Lie algebras, conformal field theory, and the $4d$/VOA correspondences arising from physics \cite{BLLPRvR}. The goal of this paper and a few future ones is to propose a Langlands style framework for representation theory of $L_k$ at non-admissible levels.

The first goal of this paper concerns the associated variety of $L_k$, which is an important invariant capturing many key structures of $L_k$. Recall that the associated variety $X_{L_k}$ of $L_k$ is defined as the reduced spectrum of its Zhu's $C_2$-algebra. Since $L_k$ a quotient of the universal affine vertex algebra $V^k$, $X_{L_k}$ is a closed conic $G$-invariant Poisson subscheme in $\fg^\ast$. Throughout the paper, we identify $\fg^*$ with $\fg$ using the Killing form, so that $X_{L_k} \subset \fg$. Let $\check \Bh$ be the dual Coxeter number of $\fg$. Write $m = k+\check\Bh$. If $m \not\in \mathbb{Q}_{\geqslant 0}$, then $L_k$ is the universal vertex algebra $V^k(\fg)$ itself, so $X_{L_k}=\fg$. When $m=0$, i.e. when $k$ is the critical level, $X_{L_k}$ equals to nilpotent cone $\cN$ in $\fg$. If $k$ is admissible, then $X_{L_k}$ is the closure of a nilpotent orbit explicitly determined in \cite{Arakawa:C2}. When $k$ is non-admissible, the structure of $X_{L_k}$ is quite mysterious. Some particular cases have been computed using different methods \cite{Arakawa-Moreau:Omin, Arakawa-Moreau:sheets, Arakawa-Moreau:irred, AFK, Jiang-Song, ADFLM}. There were also some predictions on $X_{L_k}$ from physics \cite{Xie-Yan:4dN2}. But there was not a uniform description of what the variety $X_{L_k}$ should be in general. 

In this paper, we propose a conjecture which gives a uniform description for $X_{L_k}$ for simply-laced $\fg$ and for all integer level $k$ above the critical level, that is for $m\in\BZ_{\ge 1}$. The key ingredient for its formulation is a map from the set $\ubar \cN$ of adjoint nilpotent orbits in the Lie algebra $\fg$ to positive integers $\BZ_{\ge 1}$, called the cyclotomic level map
\begin{equation*}
	\cl_n: \ubar \cN \aro \BZ_{\ge 1}.
\end{equation*}
Its definition is given in terms of the highest weight for a well-chosen $\fsl_2$-triple associated to the orbits, see Definition \ref{def:cln}, and is inspired by \cite[Proposition 9.2]{KL:Fl} and \cite[\S 7.3]{stable-grading}. Alternatively, $\cl_n$ can be defined using the largest size of the Jordan blocks of nilpotent elements, see Lemma \ref{lem:cln=order}. This map has the following remarkable properties which, in terms of the latter definition of $\cl_n$, are due to \cite{VIGRE}.

\begin{theorem}
	Let $m$ be any positive integer. Then
	\begin{equation*}
		\cl_n\inv([1,m]) = \overline{\BO(m)}.
	\end{equation*}
	In other words,
	\begin{equation*}
		\overline{\BO(m)} = \bigcup_{\cl_n(\BO) \le m} \BO.
	\end{equation*}
\end{theorem}
This appears as Theorem \ref{thm:O(m)}. It gives us a way to attach a nilpotent orbit to each positive integer and allows us to formulate the following conjecture.

\begin{conjecture}[Conjecture \ref{conj:AV}]
	Suppose $\fg$ is of simply-laced type. Write
	\begin{equation*}
		\check \cl_n: \ubar {\check \cN} \aro \BZ_{\ge 1}
	\end{equation*}
	for the cyclotomic level map defined on nilpotent orbits of $\check \fg$. Let $m$ be an integer with $m \ge 1$, and let $k = m - \check \Bh$. Let $\check \BO(m)$ be the orbit attached to $m$ as described above, and write $\check \BO(m) = \Sat_{\check L}^{\check G} \BO_{\check L}$ where $\BO_{\check L}$ is distinguished in $\check \fl$. Then
	\begin{equation*}
		X_{L_k} = \overline{\cS(\fl, \bd \BO_{\check L})}.
	\end{equation*}	
\end{conjecture}
Here $\Sat_{\check L}^{\check G}$ denotes the saturation of orbits, sending a nilpotent orbit $\BO_{\check L}$ in the Levi subalgebra $\check \fl$ to the orbit $\Ad(\check G) \cdot \BO_{\check L}$ in $\check \fg$. The map $\bd$ is the Barbasch-Vogan-Lusztig-Spaltenstein duality on nilpotent orbits of $\check \fl$, see \ref{cls:BV}. The notation $\overline{\cS(\fl, \bd \BO_{\check L})}$ denotes the \textit{sheet} attached to the pair $(\fl, \bd \BO_{\check L})$, which is defined to be the closure of the image of $G \times_P (\bd \BO_{\check L} \times \fz(\fl) \times \fu)$ under the moment map
\begin{equation*}
	G \times_P \fp \aro \fg, \quad (g,x) \mapsto \Ad(g) \cdot x.
\end{equation*}
Here $P = LU$ is any parabolic in $G$ containing $L$ as its Levi, and $\fp = \fl \oplus \fu$ is its Lie algebra.

We list $\check \BO(m)$ in all types, providing an explicit description of the sheets $\overline{\cS(\fl, \bd \BO_{\check L})}$ appearing in the conjecture. Based on these computations, we show that this conjecture conforms with all the examples in which $X_{L_k}$ has been computed. A list of them is given in Table \ref{tbl:known-AV}.

This conjecture also gives a way to detect when the vertex algebra $L_k$ is quasi-lisse. Recall that a vertex algebra is called \emph{quasi-lisse} if its associated variety has only finitely many symplectic leaves \cite{Arakawa-Kawasetsu}. This family of vertex algebras is remarkable, as they have only finitely many simple ordinary modules, whose normalised characters enjoys modular invariance properties \cite{Arakawa-Kawasetsu}. In the case of $L_k$, it is quasi-lisse if and only if $X_{L_k}$ is contained in the nilpotent cone. The variety $\overline{\cS(\fl, \bd \BO_{\check L})}$ is contained in the nilpotent cone if and only if the center $\fz(\fl)$ is trivial. This happens if and only if $L=G$. So our conjecture implies the following explicit criterion on the quasi-lisseness of $L_k$.

\begin{namedtheorem}[Corollary of Conjecture,]
	Suppose $\fg$ is of simply-laced type and $k$ is an integer with $k > - \check \Bh$. Then $L_k$ is quasi-lisse if and only if $\check \BO(m)$ is a distinguished orbit.
\end{namedtheorem}

As a byproduct, we also obtain a conjectural description for the associated varieties of affine $W$-algebras obtained as quantized Drinfeld-Sokolov reductions of $L_k$ and a criterion for their quasi-lisseness, see Remark \ref{rmk:W-algs}.

\medskip

The second goal of this paper is to prove results which provide links between $L_k$, affine Springer fibers, and Kazhdan-Lusztig cells. In the finite dimensional setting, associated varieties of primitive ideals of simple highest weight modules have close relationships with cells in the finite Weyl group and Springer fibers, see for example \cite{Lusztig:char,Barbasch-Vogan:unipotent} (some of this is mentioned in \S \ref{subsec:KL-cells}). This motivated us to try to find a similar picture in the affine situation. Further elaboration on this picture will appear in a future work; in this paper, we prove two compatibility results, which can be viewed as conceptual evidence for the conjecture above.

Recall that Lusztig \cite{Lusztig:Phi} defined a map $\Phi$ from the set $\ubar W$ of conjugacy classes in the Weyl group $W$ to $\ubar \cN$. Yun \cite{Yun:RTmin-arXiv} gave another construction of this map using the so called minimal reduction type map $\RTmin$ for regular semi-simple elements in the affine Lie algebra $\fg_{aff}$, see \S \ref{subsec:reduction}, and the latter plays an important role in understanding affine Springer representations, see \cite{Jakob-Yun,Chua,Chua:RTmin} as examples. Our next result relates the map $\cl_n$ with $\RTmin$ through an explicitly map
\begin{equation*}
	\cl_W: \ubar W \aro \BZ_{\ge 1}
\end{equation*}
defined using the maximal order of eigenvalues of $W \acts \fh$, see Definition \ref{def:clW}. 

\begin{theorem}[Theorem \ref{thm:cl-triangle}]
	The following diagram commutes
	\begin{equation*}
		\begin{tikzcd}
			\ubar W \ar[dr, "\RTmin"', two heads] \ar[rr, "\cl_W"] 
			&& \BZ_{\ge 1}\\
			& \ubar \cN \ar[ur, "\cl_n"', end anchor=south west]
		\end{tikzcd}.
	\end{equation*}
\end{theorem}

For the last result, recall that as a $\fg_{aff}$-module, $L_k$ is the unique simple highest weight module with highest weight $k\Lambda_0$. Let $\rho$ denote the half sum of positive roots in $\fg$, and let $\hat \rho = \check \Bh \Lambda_0 + \rho$. Then $k \Lambda_0 + \hat \rho = m \Lambda_0 + \rho$. Let $\xi_m$ be the unique dominant affine weight in the orbit of $m\Lambda_0 + \rho$ for the action of the affine Weyl group $W_{aff}$ on the dual of the affine Cartan subalgebra. Let $W_m$ be the stabilizer of $\xi_m$ in $W_{aff}$, and let $w_m$ be the longest element in $W_m$. Lusztig \cite{Lusztig:aff-cells-4} constructed a bijection between the set of two-sided cells in $W_{aff}$ and the set $\ubar{\check \cN}$ of nilpotent orbits in $\check \fg$, see \ref{cls:L-bij}. Our third theorem states that the image of $\check \BO(m)$ under this bijection can be constructed using $\xi_m$ and $w_m$. 

\begin{theorem}[Theorem \ref{thm:O(m)-n-cells}]
	Suppose $\fg$ is of simply-laced type. Let $m$ be an integer with $m \ge 1$. Then Lusztig's bijection sends the two-sided cell $\ubar \bc(w_m)$ of $w_m$ to the orbit $\check \BO(m)$.
\end{theorem}

We remark that when $m \in \im \check \cl_n$, the assignment $\check \BO(m) \mapsto \xi_m \mapsto \ubar \bc(w_m)$ is similar in spirit to the bijection in classical representation theory between special nilpotent orbits in $\check \fg$ and two-sided cells in the finite Weyl group $W$, see \ref{cls:aff-duality}. This analogy will be explored further in the sequel.

The framework we envision is partially motivated by ideas from conformal field theory. From the point of view of the $4d$/VOA duality in physics \cite{BLLPRvR}, $L_k$ (for some $k$) should arise as the vertex algebra associated with some four dimensional $\cN=2$ superconformal field theory, whose Higgs branch is expected to be the associated variety $X_{L_k}$ and the Coulomb branch is expected to be related to certain moduli space $\cM_k$ of Higgs bundles of on $\BP^1$ with an irregular singularity at $\infty$ and a regular singularity at zero. The $4d$ mirror symmetry proposal in \cite{Shan-Xie-Yan:phys} predicts several correspondences between representations of $L_k$ and the geometry of torus fixed manifolds in certain Hitchin moduli space $\cM_k$. In \cite{Shan-Xie-Yan}, it was proved that for the boundary admissible level $k$, there is a natural bijection between simple $L_k$-modules and fixed points for a $\BC^\times$-action on an affine Springer fiber (which is homeomorphic to the central fiber of $\cM_k$), and that this bijection is compatible with modularity properties. One of the goals of the current paper is to extend this correspondence to non-admissible levels by first providing a description of $X_{L_k}$ using data coming from an appropriate affine Springer fiber, and by building foundations for constructing a map (with several good properties) from the top homology of the said fiber to the Grothendieck group of $L_k$-modules.

\smallskip

\subsection{Structure of the paper}

In \S \ref{subsec:aff-alg}-\ref{subsec:AV} we fix notations and recall definitions for the affine Lie algebra, affine vertex algebra, and associated varieties. In \S \ref{subsec:cln} we define the cyclotomic level map $\cl_n$, list its values in all cases, and prove Theorem \ref{thm:O(m)}. The main conjecture is contained in \S \ref{subsec:conj-AV} along with a discussion of known cases. The relation of cyclotomic levels with minimal reduction types is studied in \S \ref{sec:RTmin}. We recall minimal reduction types map in \S \ref{subsec:reduction}. In \S \ref{subsec:rs-elem} we give a review on the root valuation strata discovered by Goresky-Kottwitz-MacPherson. They will be used in the proof of Theorem \ref{thm:cl-triangle} in \S \ref{subsec:cl-RTmin}. In \S \ref{sec:cells} the interaction between $\cl_n$ and two-sided cells is explored. The last section \S \ref{sec:co} contains combinatorics for proofs in classical types.

\subsection*{Acknowledgements} We would like to thank Tomoyuki Arakawa, Gurbir Dhillon, Zhiwei Yun for helpful discussions. The third author would like to thank the \textsf{atlas} team for creating the \textsf{atlas} software which helped tremendously in our computations. He would like to especially thank Jeffrey Adams for answering many questions regarding the \textsf{atlas} software, and for generously providing codes for producing closure diagrams for nilpotent orbits.

PS is supported by NSFC Grant No.~12225108 and by the New Cornerstone Science Foundation through the Xplorer Prize.

WY is supported by national key research
and development program of China (No.~2022ZD0117000).

\section{Associated varieties of simple affine vertex algebras}\label{sec:AV}

In this section, we first define a map 
\begin{equation*}
	\cl_n: \ubar \cN \aro \BZ_{\ge 1},
\end{equation*}
which we call the \emph{cyclotomic level} map, from nilpotent orbits to positive integers, compute its values and describe one of its main properties. This map allows us to formulate our main conjecture \ref{conj:AV} on associated varieties of the simple affine vertex algebras $L_k(\fg)$, see \S \ref{subsec:conj-AV}. The two subsections \S \ref{subsec:aff-alg}-\ref{subsec:AV} contain a review of the players involved in the conjecture and set notations for the rest of the paper.

\subsection{Cyclotomic levels of nilpotent orbits}\label{subsec:cln}

Let us first fix Lie theoretic notation to be used throughout the paper.

\begin{clause}[Lie theoretic notations]
	Write $\fg$ for a complex simple Lie algebra, $\fh \subset \fg$ for a Cartan subalgebra, $\fb \subset \fg$ for a Borel containing $\fh$, $W$ for the Weyl group, and $\ubar W$ for the set of conjugacy classes in $W$. Let $G$ be the connected simply-connected algebraic group with Lie algebra $\fg$. The root system of $(\fg,\fh)$ is denoted by $R = R(\fg,\fh)$, and the set of positive roots determined by $\fb$ is written as $R^+$. We write $\rho$ for the half sum of positive roots. Let $\check \fg$ be the Langlands dual Lie algebra containing a Cartan $\check \fh$ which is identified with $\fh^*$. 
	
	We fix an invariant symmetric bilinear form $(-,-)$ on $\fg$ so that the induced form on $\fh^*$ satisfies $(\theta , \theta) = 2$, where $\theta \in R^+$ is the highest root. We write $\Bh$ and $\check \Bh$ for the Coxeter number and the dual Coxeter number of $\fg$. 
	
	
	We write $\cN$ and $\check \cN$ for the nilpotent cone of $\fg$ and $\check \fg$, and write $\ubar \cN$, $\ubar{\check \cN}$ for the set of adjoint orbits in $\cN$ and $\check \cN$, respectively. The partial order $\BO \preccurlyeq \BO'$ on $\ubar \cN$ means $\BO \subseteq \overline{\BO'}$. The Barbasch-Vogan-Lusztig-Spaltenstein duality map of nilpotent orbits is denoted by $\bd: \ubar{\check \cN} \to \ubar \cN$, its definition will be recalled in \ref{cls:BV}. Given a Levi $L \subset G$ and an orbit $\BO_L \in \ubar \cN_L$, we write $\Sat_L^G \BO_L$ for the orbit $\Ad (G) \cdot \BO_L$ ($\Sat$ stands for ``saturation''), and write $\Ind_L^G \BO_L$ for the Lusztig-Spaltenstein induction, which is the unique open orbit in the image of the moment map $G \times_P (\overline{\BO_L} \times \fu) \to \fg$. Here $P = LU$ is any parabolic containing $L$ as its Levi, and $\fp = \fl \oplus \fu$ is its Lie algebra. Given any orbit $\BO \in \ubar \cN$, there is a unique conjugacy classes of Levi $L$ so that $\BO= \Sat_L^G \BO_L$ where $\BO_L$ is \textit{distinguished} in $\fl$, i.e. $\BO_L$ is not the saturation of a strictly smaller Levi. In this case $L$ will be called a \textit{Bala-Carter Levi} of $\BO$.
\end{clause}

\begin{definition}\label{def:cln}
	The \textbf{cyclotomic level} map
	\begin{equation*}
		\cl_n: \cN \aro \BZ_{\ge 1}	
	\end{equation*}
	is defined as follows. For any $e \in \cN$, let $\fl \subset \fg$ be a smallest Levi containing $e$, so that $e$ is distinguished in $\fl$ (i.e. $\fl$ is a \textit{Bala-Carter Levi} of $e$). Let $\{h,e,f\}$ be an $\fsl_2$-triple in $\fl$, and let $2a$ be the largest $\ad h$-weight in $\fl$. It is known that $2a$ is an even non-negative integer, see, for example, \cite[Proposition 5.1]{Barbasch-Vogan:unipotent}.	We then set
	\begin{equation*}
		\cl_n(e) := a+1.
	\end{equation*}
	Clearly the map $\cl_n$ descends to a map on orbits
	\begin{equation*}
		\cl_n: \ubar \cN \aro \BZ_{\ge 1}.
	\end{equation*}
	We write $\cl_{n,G}$ for $\cl_n$ when it is necessary to specify the group.
\end{definition}

The name for the map comes from its alternative definition, see Definition \ref{def:clW} and Theorem \ref{thm:cl-triangle}. This definition of $\cl_n$ has been considered before in various settings, for example \cite[Proposition 9.2]{KL:Fl} and \cite[\S 7.3]{stable-grading}. We would like to present another equivalent definition. For this, we need the following lemma.

\begin{lemma}\label{lem:max-h-hts}
	Let $\fl \subset \fg$ be a Levi, let $\{h,e,f\} \subset \fl$ be an $\fsl_2$-triple where $h \in \fh$ and $e$ is distinguished in $\fl$. Then
	\begin{equation*}
		\max_{\alpha \in R(\fl,\fh)} \langle \alpha, h \rangle \le 
		\max_{\beta \in R(\fg,\fh)} \langle \beta, h \rangle \le 
		\max_{\alpha \in R(\fl,\fh)} \langle \alpha, h \rangle +1.
	\end{equation*}
	In other words, if the highest $h$-weight on $\fl$ is $2a$, then the highest $h$-weight on $\fg$ is either $2a$ or $2a+1$.
\end{lemma}

The proof is a type-by-type verification. We postpone the proof to \S \ref{subapp:partition}

\begin{lemma}\label{lem:cln=order}
	For $e \in \cN$, we have
	\begin{equation*}
		\cl_n(e) = \min\{m \mid (\ad e)^{2m} = 0 \}.
	\end{equation*}
\end{lemma}

\begin{proof}
	Suppose $\cl_n(e) = a+1$. By the definition of $\cl_n$ and the preceding lemma, the largest $\ad h$-weight on $\fg$ is either $2a$ or $2a+1$. Let $\fs \subset \fg$ be an $\fsl_2$-subalgebra spanned by an $\fsl_2$-triple containing $e$. In the first case, the largest irreducible $\ad \fs$-submodule in $\fg$ has highest weight $2a$ and is $(2a+1)$-dimensional. So $(\ad e)^{2a+1} = 0$ and $(\ad e)^{2a} \neq 0$ on this submodule, and hence also on $\fg$. Therefore
	\begin{equation}\label{eqn:cln=order}
		\min\{m \mid (\ad e)^{2m} = 0 \} = a+1.
	\end{equation}
	In the second case, the largest irreducible $\fs$-submodule is $(2a+2)$-dimensional, and hence $(\ad e)^{2a+2} = 0$, $(\ad e)^{2a+1} \neq 0$. We again get (\ref{eqn:cln=order}).
\end{proof}

We have the following remarkable properties of $\cl_n$.

\begin{theorem}\label{thm:O(m)}
	Let $m_1$, $m_2$ be in the image of $\cl_n$, and let $m$ be any positive integer.
	\begin{enumerate}
		\item There is a unique maximal orbit $\BO(m)$ in $\cl_n\inv([1,m])$.
		
		\item If $\BO_1 \preccurlyeq \BO_2$, then $\cl_n(\BO_1) \le \cl_n(\BO_2)$. 
		
		\item If $m_1 < m_2$, then $\BO(m_1) \prec \BO(m_2)$.
		
		\item We have
		\begin{equation*}
			\overline{\BO(m)} = \bigcup_{\cl_n(\BO) \le m} \BO.
		\end{equation*}
	\end{enumerate}
\end{theorem}

\begin{proof}
	Part (2) is clear from the alternative description of $\cl_n$ given in Lemma \ref{lem:cln=order}. Part (1) and (3) both follow easily from (4). Part (4) is proven in \cite{VIGRE} by type-based calculations (see also \cite[Theorem 5.8.1]{Arakawa:C2}.
\end{proof}


\begin{remark}
	In classical types, the partition for $\BO(m)$ is roughly the ``most rectangular'' partition with as many parts $\approx m$ as possible.
\end{remark}

We now list the values of $\cl(\BO)$ and $\BO(m)$ in classical types. The values $\cl_n(\BO)$ are computed based on the description of $\max_{\alpha \in R(\fl,\fh)} \langle \alpha, h\rangle$ in terms of the partitions of $\BO$, which is contained in the proof of Lemma \ref{lem:max-h-hts} in \S \ref{subapp:partition}. The partition $q(m)$ corresponding to $\BO(m)$ can be found in \cite[\S3.3 Theorem]{VIGRE}.

\begin{notation}[Partitions]\label{notn:partition}
	Partitions are written as $p = (p_1, p_2, \ldots, p_r)$, where $p_1 \ge p_2 \ge \cdots \ge p_r \ge 1$. Each $p_i$ is said to be a \textit{part} of $p$. The number of parts is written $\# p = r$, and the sum of parts is $|p| = p_1 + \cdots + p_r$. If the multiplicity of a part $x$ of $p$ is $m$, i.e. if $x$ is repeated $m$ times, we will abbreviate $(\ldots, \underbrace{x,\ldots,x}_{m \text{ times}}, \ldots)$ by $(\ldots, x^m, \ldots)$. We write $\BO_p$ for the nilpotent orbit corresponding to partition $p$.
	
	For partitions $p = (p_1, \ldots, p_r)$ and $q = (q_1,\ldots, q_s)$, let $p \cup q$ denote the union of parts, i.e. up to reordering, $p \cup q = (p_1,\ldots, p_r,q_1,\ldots, q_s)$. Let $p+q$ denote the addition of parts, i.e. $p+q = (p_1+q_1, p_2+q_2,\ldots)$.
	
	A partition $p$ is said to be \textit{symplectic} if each odd part appears even number of times. A partition $p$ is said to be \textit{orthogonal} if each even part appears even number of times. 
\end{notation}

\begin{namedtheorem}[Values in type A,]\label{values:A}
	Let $\fg = \fsl_n$. Then 
	\begin{equation*}
		\cl_n(\BO_p) = p_1.
	\end{equation*}
	The image of $\cl_n$ is equal to $[1,n]$. The partition for $\BO(m)$ is equal to
	\begin{equation*}
		q(m) = (m^b, v),
	\end{equation*}
	where $b$ is the largest integer so that $bm \le n$, and $v = n-bm$ is the unique integer so that $q(m)$ is a partition of $n$.
\end{namedtheorem}

\begin{namedtheorem}[Values in type B,]\label{values:B}
	Let $\fg = \fso(2n+1)$. Then
	\begin{align*}
		\cl_n(\BO_p) 
		&= \max \{ (\text{1st repeated part of }p), (\text{1st non-repeated part of } p)-1 \}\\
		&= \begin{cases}
			p_1 & \text{if } p_1 = p_2,\\
			p_1-1 & \text{if } p_1 > p_2.
		\end{cases}
	\end{align*}
	The image of $\cl$ is equal to $[1,n] \cup (2\BZ \cap [n+1,2n])$. The partition $q(m)$ for $\BO(m)$ is described as follows. If $m$ is even, let $k \ge 0$ be the largest integer so that $(m+1)+2km \le 2n+1$; if $m$ is odd, let $b \ge 0$ be the largest integer so that $bm \le 2n+1$ (then $b \ge 2$). Then 
	\begin{equation*}
		q(m) =
		\begin{cases}
			(m+1,m^{2k}) & m \text{ even}, (2k+1)m = 2n\\
			(m+1,m^{2k}, v,1) & m \text{ even}, (2k+1)m < 2n \le (2k+2)m\\
			(m+1,m^{2k},m-1,v) & m \text{ even}, (2k+2)m < 2n\\
			(m^b) & m \text{ odd}, b \ge 3 \text{ odd}, bm = 2n+1 \\
			(m^b, v, 1) & m \text{ odd}, b \ge 3 \text{ odd}, bm < 2n+1\\
			(m^b, v) & m \text{ odd}, b \ge 2 \text{ even}.
		\end{cases}
	\end{equation*}
	Here $v$ is the unique number so that the above tuples are partitions of $2n+1$. 
\end{namedtheorem}

\begin{namedtheorem}[Values in type C,]\label{values:C}
	Let $\fg = \fsp(2n)$. Then
	\begin{equation*}
		\cl_n(\BO_p) = p_1.
	\end{equation*}
	The image of $\cl$ is equal to $[1,n] \cup (2\BZ \cap [n+1,2n])$. The partition $q(m)$ for $\BO(m)$ is described as follows. If $m$ is odd, let $k \ge 0$ be the largest integer so that $2km \le 2n$; if $m$ is even, let $b \ge 0$ be the largest integer so that $bm \le 2n$. Then 
	\begin{equation*}
		q(m) =
		\begin{cases}
			(m^{2k}) & m \text{ odd}, 2km = 2n\\
			(m^{2k}, v) & m \text{ odd}, 2n+1-m \le 2km < 2n\\
			(m^{2k}, m-1, v) & m \text{ odd}, 2n+1-m > 2km < 2n\\
			(m^b) & m \text{ even}, bm=2n\\
			(m^b, v) & m \text{ even}, bm < 2n.
		\end{cases}
	\end{equation*}
	Here $v$ is the unique number so that the above tuples are partitions of $2n$. 
\end{namedtheorem}

\begin{namedtheorem}[Values in type D,]\label{values:D}
	Let $\fg = \fso(2n)$. Then 
	\begin{align*}
		\cl_n(\BO_p) 
		&= \max \{ (\text{1st repeated part of }p), (\text{1st non-repeated part of } p)-1 \}\\
		&= \begin{cases}
			p_1 & \text{if } p_1 = p_2,\\
			p_1-1 & \text{if } p_1 > p_2.
		\end{cases}
	\end{align*}
	The image of $\cl$ is equal to $[1,n] \cup (2\BZ \cap [n+1,2n-2])$. The partition $q(m)$ for $\BO(m)$ is described as follows. If $m$ is even, let $k \ge 0$ be the largest integer so that $(m+1)+2km \le 2n$; if $m$ is odd, let $b \ge 0$ be the largest integer so that $bm \le 2n$ (then $b \ge 2$). Then 
	\begin{equation*}
		q(m) =
		\begin{cases}
			(m+1,m^{2k}, v) & m \text{ even}, (2k+2)m \ge 2n\\
			(m+1,m^{2k},m-1,v,1) & m \text{ even}, (2k+2)m < 2n\\
			(m^b) & m \text{ odd}, b \ge 2 \text{ even}, bm = 2n \\
			(m^b, v, 1) & m \text{ odd}, b \ge 2 \text{ even}, bm < 2n\\
			(m^b, v) & m \text{ odd}, b \ge 3 \text{ odd}.
		\end{cases}
	\end{equation*}
	Here $v$ is the unique number so that the above tuples are partitions of $2n$. Note that $q(m)$ can never be very even, so $q(m)$ uniquely determines an orbit $\BO(m)$.
\end{namedtheorem}

The values in exceptional types are computed using the \textsf{atlas} software.

\begin{namedtheorem}[Values in exceptional types,]\label{values:exc}
	The values of $\cl_n(\BO)$ and $\BO(m)$ in exceptional types are contained in Figures \ref{values:E6}-\ref{values:FG}. We explain here how to read them.
	
	Regarding the nodes:
	\begin{itemize}
		\item Each node denotes a nilpotent orbit $\BO$.
		\item The first label is the Bala-Carter name of the orbit $\BO$.
		\item The second label is $\cl_n(\BO)$. If $\BO = \BO(m)$, then the second label is put in bracket $[m]$.
		\item Dark blue filling means $\BO$ is even.
		\item Light blue filling means $\BO$ is special but not even.
		\item White filling means $\BO$ is not special.
		\item Red border means $\bd \BO$ is even.
	\end{itemize}
	For example, in E6 (Figure \ref{values:E6}), the node
	\begin{equation*}
		\scalebox{0.6}
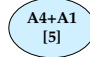
{
			\begin{tikzpicture}[every text node part/.style={align=center}]
				\definecolor{fillcolor}{rgb}{0.82,0.94,1.0};
				\node [draw=black,fill=fillcolor,ellipse] {\normalfont{\textbf{A4+A1}} \\ \normalfont{\textbf{[5]}}};
			\end{tikzpicture}
		}
	\end{equation*}
	denotes the orbit $\BO$ with Bala-Carter label A4+A1 which is special, non-even, with $\bd \BO$ non-even, and $\cl_n(\BO) = 5$, $\BO = \BO(5)$.
	
	Regarding the edges:
	\begin{itemize}
		\item Edges denote the closure relations between the orbits.
		\item Green edge means the orbits connected by the edge are in the same special piece.
	\end{itemize}
\end{namedtheorem}

\begin{figure}
	\caption{Values of $\cl_n$ in E6}\label{values:E6}	
	\vspace{\baselineskip}
	\scalebox{0.65}{
		\begin{tikzpicture}[>=latex',line join=bevel, every text node part/.style={align=center, font=\bfseries}, scale=0.7, xscale=1.2]
			\node (1) at (191.3bp,19.94bp) [draw=red,fill=blue,ellipse] {\twi 0 \\ \twi{[1]}};
			\definecolor{fillcolor}{rgb}{0.82,0.94,1.0};
			\node (2) at (191.3bp,95.82bp) [draw=red,fill=fillcolor,ellipse] {A1 \\ 2};
			\definecolor{fillcolor}{rgb}{0.82,0.94,1.0};
			\node (3) at (191.3bp,171.7bp) [draw=red,fill=fillcolor,ellipse] {2A1 \\ 2};
			\node (4) at (191.3bp,247.58bp) [draw=red,fill=white,ellipse] {3A1 \\ {[2]}};
			\node (5) at (191.3bp,323.46bp) [draw=red,fill=blue,ellipse] {\twi{A2} \\ \twi 3};
			\definecolor{fillcolor}{rgb}{0.82,0.94,1.0};
			\node (6) at (191.3bp,391.34bp) [draw=black,fill=fillcolor,ellipse] {A2+A1 \\ 3 };
			\node (7) at (284.3bp,459.22bp) [draw=red,fill=blue,ellipse] {\twi{2A2} \\ \twi 3};
			\definecolor{fillcolor}{rgb}{0.82,0.94,1.0};
			\node (8) at (99.296bp,459.22bp) [draw=black,fill=fillcolor,ellipse] {A2+2A1 \\ 3 };
			\definecolor{fillcolor}{rgb}{0.82,0.94,1.0};
			\node (9) at (67.296bp,535.1bp) [draw=red,fill=fillcolor,ellipse] {A3 \\ 4};
			\node (10) at (284.3bp,535.1bp) [draw=red,fill=white,ellipse] {2A2+A1  \\ {[3]}};
			\node (11) at (175.3bp,610.98bp) [draw=red,fill=white,ellipse] {A3+A1 \\ 4};
			\node (12) at (175.3bp,686.86bp) [draw=red,fill=blue,ellipse] {\twi{D4(a1)} \\ \twi{[4]}};
			\node (13) at (99.296bp,762.74bp) [draw=black,fill=blue,ellipse] {\twi{A4} \\ \twi 5};
			\node (14) at (252.3bp,762.74bp) [draw=red,fill=blue,ellipse] {\twi{D4} \\ \twi 6};
			\definecolor{fillcolor}{rgb}{0.82,0.94,1.0};
			\node (15) at (99.296bp,830.62bp) [draw=black,fill=fillcolor,ellipse] {A4+A1 \\ {[5]}};
			\node (16) at (76.296bp,898.5bp) [draw=red,fill=white,ellipse] {A5 \\ 6};
			\definecolor{fillcolor}{rgb}{0.82,0.94,1.0};
			\node (17) at (252.3bp,898.5bp) [draw=black,fill=fillcolor,ellipse] {D5(a1) \\ 6};
			\node (18) at (164.3bp,974.38bp) [draw=red,fill=blue,ellipse] {\twi{E6(a3)} \\ \twi{[6]}};
			\node (19) at (164.3bp,1042.3bp) [draw=black,fill=blue,ellipse] {\twi{D5} \\ \twi{[8]}};
			\node (20) at (164.3bp,1102.1bp) [draw=black,fill=blue,ellipse] {\twi{E6(a1)} \\ \twi{[9]}};
			\node (21) at (164.3bp,1170.0bp) [draw=red,fill=blue,ellipse] {\twi{E6} \\ \twi{[12]}};
			\draw [black,] (2) ..controls (191.3bp,64.549bp) and (191.3bp,51.059bp)  .. (1);
			\draw [black,] (3) ..controls (191.3bp,140.43bp) and (191.3bp,126.94bp)  .. (2);
			\draw [black,] (4) ..controls (191.3bp,216.31bp) and (191.3bp,202.82bp)  .. (3);
			\draw [green,] (5) ..controls (191.3bp,292.19bp) and (191.3bp,278.7bp)  .. (4);
			\draw [black,] (6) ..controls (191.3bp,369.47bp) and (191.3bp,355.23bp)  .. (5);
			\draw [black,] (7) ..controls (242.32bp,428.48bp) and (221.23bp,413.54bp)  .. (6);
			\draw [black,] (8) ..controls (131.82bp,434.93bp) and (158.67bp,415.7bp)  .. (6);
			\draw [black,] (9) ..controls (81.536bp,501.22bp) and (89.457bp,482.93bp)  .. (8);
			\draw [black,] (10) ..controls (284.3bp,503.83bp) and (284.3bp,490.34bp)  .. (7);
			\draw [black,] (10) ..controls (203.74bp,501.93bp) and (155.89bp,482.82bp)  .. (8);
			\draw [black,] (11) ..controls (131.17bp,579.79bp) and (109.64bp,565.06bp)  .. (9);
			\draw [green,] (11) ..controls (219.2bp,580.22bp) and (239.96bp,566.15bp)  .. (10);
			\draw [green,] (12) ..controls (175.3bp,655.59bp) and (175.3bp,642.1bp)  .. (11);
			\draw [black,] (13) ..controls (122.4bp,739.28bp) and (141.41bp,720.79bp)  .. (12);
			\draw [black,] (14) ..controls (221.37bp,732.07bp) and (206.84bp,718.13bp)  .. (12);
			\draw [black,] (15) ..controls (99.296bp,806.25bp) and (99.296bp,786.87bp)  .. (13);
			\draw [black,] (16) ..controls (86.981bp,866.89bp) and (91.988bp,852.55bp)  .. (15);
			\draw [black,] (17) ..controls (252.3bp,863.1bp) and (252.3bp,810.95bp)  .. (14);
			\draw [black,] (17) ..controls (198.68bp,874.41bp) and (152.96bp,854.72bp)  .. (15);
			\draw [green,] (18) ..controls (128.65bp,943.45bp) and (111.58bp,929.11bp)  .. (16);
			\draw [black,] (18) ..controls (203.0bp,940.88bp) and (224.96bp,922.44bp)  .. (17);
			\draw [black,] (19) ..controls (164.3bp,1020.4bp) and (164.3bp,1006.1bp)  .. (18);
			\draw [black,] (20) ..controls (164.3bp,1079.6bp) and (164.3bp,1064.8bp)  .. (19);
			\draw [black,] (21) ..controls (164.3bp,1138.1bp) and (164.3bp,1123.8bp)  .. (20);
		\end{tikzpicture}
	}
\end{figure}

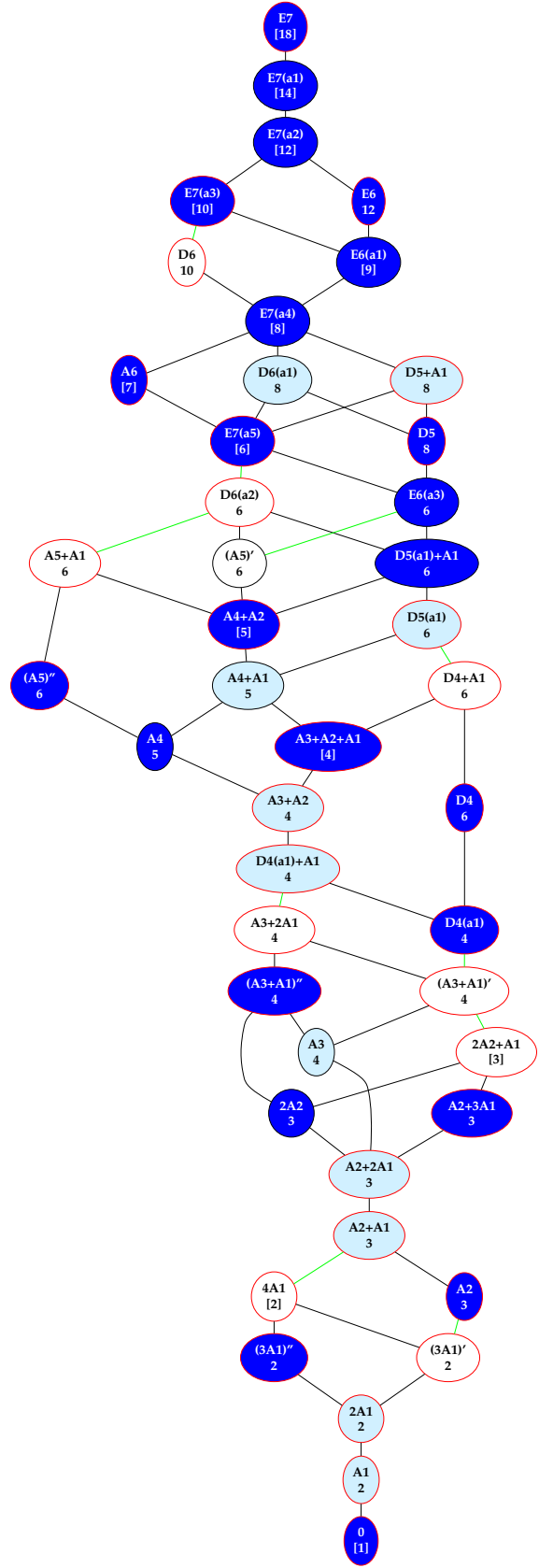
\begin{figure}
	\caption{Values of $\cl_n$ in E7}\label{values:E7}	
	\vspace{\baselineskip}
	\scalebox{0.53}{
		\begin{tikzpicture}[>=latex',line join=bevel, every text node part/.style={align=center, font=\bfseries}, scale=0.21, xscale=0.3]
			\node (1) at (4648.5bp,92.146bp) [draw=red,fill=blue,ellipse] {\twi 0 \\ \twi{[1]}};
			\definecolor{fillcolor}{rgb}{0.82,0.94,1.0};
			\node (2) at (4648.5bp,312.44bp) [draw=red,fill=fillcolor,ellipse] {A1 \\ 2};
			\definecolor{fillcolor}{rgb}{0.82,0.94,1.0};
			\node (3) at (4648.5bp,532.73bp) [draw=red,fill=fillcolor,ellipse] {2A1 \\ 2};
			\node (4) at (3602.5bp,753.02bp) [draw=red,fill=blue,ellipse] {\twi{(3A1)''} \\ \twi 2};
			\node (5) at (5695.5bp,753.02bp) [draw=red,fill=white,ellipse] {(3A1)'  \\ 2};
			\node (6) at (5890.5bp,973.31bp) [draw=red,fill=blue,ellipse] {\twi{A2} \\ \twi 3};
			\node (7) at (3602.5bp,973.31bp) [draw=red,fill=white,ellipse] {4A1  \\ {[2]}};
			\definecolor{fillcolor}{rgb}{0.82,0.94,1.0};
			\node (8) at (4746.5bp,1193.6bp) [draw=red,fill=fillcolor,ellipse] {A2+A1  \\ 3};
			\definecolor{fillcolor}{rgb}{0.82,0.94,1.0};
			\node (9) at (4746.5bp,1413.9bp) [draw=red,fill=fillcolor,ellipse] {A2+2A1  \\ 3};
			\definecolor{fillcolor}{rgb}{0.82,0.94,1.0};
			\node (10) at (4110.5bp,1854.5bp) [draw=black,fill=fillcolor,ellipse] {A3 \\ 4};
			\node (11) at (3810.5bp,1634.2bp) [draw=black,fill=blue,ellipse] {\twi{2A2} \\ \twi 3};
			\node (12) at (5980.5bp,1634.2bp) [draw=red,fill=blue,ellipse] {\twi{A2+3A1} \\ \twi 3};
			\node (13) at (3607.5bp,2074.8bp) [draw=red,fill=blue,ellipse] {\twi{(A3+A1)''} \\ \twi 4};
			\node (14) at (6275.5bp,1854.5bp) [draw=red,fill=white,ellipse] {2A2+A1  \\ {[3]}};
			\node (15) at (5890.5bp,2074.8bp) [draw=red,fill=white,ellipse] {(A3+A1)'  \\ 4};
			\node (16) at (5892.5bp,2295.1bp) [draw=red,fill=blue,ellipse] {\twi{D4(a1)} \\ \twi 4};
			\node (17) at (3607.5bp,2295.1bp) [draw=red,fill=white,ellipse] {A3+2A1 \\ 4};
			\node (18) at (5892.5bp,2735.6bp) [draw=red,fill=blue,ellipse] {\twi{D4} \\ \twi 6};
			\definecolor{fillcolor}{rgb}{0.82,0.94,1.0};
			\node (19) at (3769.5bp,2515.4bp) [draw=red,fill=fillcolor,ellipse] {D4(a1)+A1 \\ 4};
			\definecolor{fillcolor}{rgb}{0.82,0.94,1.0};
			\node (20) at (3769.5bp,2735.6bp) [draw=red,fill=fillcolor,ellipse] {A3+A2 \\ 4};
			\node (21) at (2170.5bp,2955.9bp) [draw=black,fill=blue,ellipse] {\twi{A4} \\ \twi 5};
			\node (22) at (4252.5bp,2955.9bp) [draw=red,fill=blue,ellipse] {\twi{A3+A2+A1} \\ \twi{[4]}};
			\node (23) at (783.52bp,3176.2bp) [draw=red,fill=blue,ellipse] {\twi{(A5)''} \\ \twi 6};
			\node (24) at (5892.5bp,3176.2bp) [draw=red,fill=white,ellipse] {D4+A1 \\ 6};
			\definecolor{fillcolor}{rgb}{0.82,0.94,1.0};
			\node (25) at (3289.5bp,3176.2bp) [draw=black,fill=fillcolor,ellipse] {A4+A1 \\ 5};
			\definecolor{fillcolor}{rgb}{0.82,0.94,1.0};
			\node (26) at (5438.5bp,3396.5bp) [draw=red,fill=fillcolor,ellipse] {D5(a1) \\ 6};
			\node (27) at (3238.5bp,3396.5bp) [draw=red,fill=blue,ellipse] {\twi{A4+A2} \\ \twi{[5]}};
			\node (28) at (3187.5bp,3616.8bp) [draw=black,fill=white,ellipse] {(A5)' \\ 6};
			\node (29) at (1087.5bp,3616.8bp) [draw=red,fill=white,ellipse] {A5+A1 \\ 6};
			\node (30) at (5438.5bp,3616.8bp) [draw=black,fill=blue,ellipse] {\twi{D5(a1)+A1} \\ \twi 6};
			\node (31) at (3187.5bp,3837.1bp) [draw=red,fill=white,ellipse] {D6(a2) \\ 6};
			\node (32) at (5435.5bp,3837.1bp) [draw=black,fill=blue,ellipse] {\twi{E6(a3)}\\ \twi 6};
			\node (33) at (5435.5bp,4057.4bp) [draw=red,fill=blue,ellipse] {\twi{D5} \\ \twi 8};
			\node (34) at (3223.5bp,4057.4bp) [draw=red,fill=blue,ellipse] {\twi{E7(a5)}  \\ \twi{[6]}};
			\node (35) at (1858.5bp,4277.7bp) [draw=red,fill=blue,ellipse] {\twi{A6} \\ \twi{[7]}};
			\definecolor{fillcolor}{rgb}{0.82,0.94,1.0};
			\node (36) at (5435.5bp,4277.7bp) [draw=red,fill=fillcolor,ellipse] {D5+A1 \\ 8};
			\definecolor{fillcolor}{rgb}{0.82,0.94,1.0};
			\node (37) at (3645.5bp,4277.7bp) [draw=black,fill=fillcolor,ellipse] {D6(a1) \\ 8};
			\node (38) at (3645.5bp,4490.0bp) [draw=black,fill=blue,ellipse] {\twi{E7(a4)}  \\ \twi{[8]}};
			\node (39) at (2553.5bp,4702.3bp) [draw=red,fill=white,ellipse] {D6 \\ 10};
			\node (40) at (4738.5bp,4702.3bp) [draw=black,fill=blue,ellipse] {\twi{E6(a1)}  \\ \twi{[9]} };
			\node (41) at (4738.5bp,4922.6bp) [draw=red,fill=blue,ellipse] {\twi{E6} \\ \twi{12}};
			\node (42) at (2741.5bp,4922.6bp) [draw=red,fill=blue,ellipse] {\twi{E7(a3)}  \\ \twi{[10]}};
			\node (43) at (3739.5bp,5134.8bp) [draw=black,fill=blue,ellipse] {\twi{E7(a2)}  \\ \twi{[12]}};
			\node (44) at (3739.5bp,5339.1bp) [draw=black,fill=blue,ellipse] {\twi{E7(a1)}  \\ \twi{[14]}};
			\node (45) at (3739.5bp,5551.4bp) [draw=red,fill=blue,ellipse] {\twi{E7}  \\ \twi{[18]}};
			\draw [black,] (2) ..controls (4648.5bp,208.14bp) and (4648.5bp,196.2bp)  .. (1);
			\draw [black,] (3) ..controls (4648.5bp,428.43bp) and (4648.5bp,416.49bp)  .. (2);
			\draw [black,] (4) ..controls (4073.2bp,653.78bp) and (4158.7bp,635.96bp)  .. (3);
			\draw [black,] (5) ..controls (5208.4bp,650.46bp) and (5131.3bp,634.37bp)  .. (3);
			\draw [green,] (6) ..controls (5798.4bp,869.2bp) and (5787.6bp,857.12bp)  .. (5);
			\draw [black,] (7) ..controls (3602.5bp,869.02bp) and (3602.5bp,857.08bp)  .. (4);
			\draw [black,] (7) ..controls (4518.2bp,876.81bp) and (4762.3bp,851.36bp)  .. (5);
			\draw [black,] (8) ..controls (5282.8bp,1090.3bp) and (5367.8bp,1074.1bp)  .. (6);
			\draw [green,] (8) ..controls (4210.6bp,1090.3bp) and (4125.9bp,1074.2bp)  .. (7);
			\draw [black,] (9) ..controls (4746.5bp,1309.6bp) and (4746.5bp,1297.7bp)  .. (8);
			\draw [black,] (10) ..controls (4664.3bp,1769.7bp) and (4706.9bp,1750.1bp)  .. (4727.5bp,1726.3bp) .. controls (4779.0bp,1667.0bp) and (4778.2bp,1575.0bp)  .. (9);
			\draw [black,] (11) ..controls (4217.1bp,1538.4bp) and (4295.2bp,1520.2bp)  .. (9);
			\draw [black,] (12) ..controls (5432.3bp,1536.2bp) and (5327.9bp,1517.7bp)  .. (9);
			\draw [black,] (13) ..controls (3849.6bp,1968.7bp) and (3887.0bp,1952.5bp)  .. (10);
			\draw [black,] (13) ..controls (3265.8bp,1976.0bp) and (3252.6bp,1962.2bp)  .. (3241.5bp,1946.6bp) .. controls (3194.0bp,1879.9bp) and (3190.3bp,1826.3bp)  .. (3241.5bp,1762.3bp) .. controls (3259.1bp,1740.4bp) and (3295.8bp,1722.0bp)  .. (11);
			\draw [black,] (14) ..controls (5144.3bp,1753.3bp) and (4772.5bp,1720.4bp)  .. (11);
			\draw [black,] (14) ..controls (6136.1bp,1750.3bp) and (6119.2bp,1737.8bp)  .. (12);
			\draw [black,] (15) ..controls (5045.6bp,1970.1bp) and (4831.9bp,1943.9bp)  .. (10);
			\draw [green,] (15) ..controls (6072.1bp,1970.8bp) and (6094.0bp,1958.4bp)  .. (14);
			\draw [green,] (16) ..controls (5891.6bp,2190.8bp) and (5891.5bp,2178.8bp)  .. (15);
			\draw [black,] (17) ..controls (3607.5bp,2190.8bp) and (3607.5bp,2178.8bp)  .. (13);
			\draw [black,] (17) ..controls (4562.9bp,2202.7bp) and (4855.1bp,2174.8bp)  .. (15);
			\draw [black,] (18) ..controls (5892.5bp,2568.2bp) and (5892.5bp,2462.4bp)  .. (16);
			\draw [black,] (19) ..controls (4712.6bp,2417.4bp) and (4951.7bp,2392.8bp)  .. (16);
			\draw [green,] (19) ..controls (3693.0bp,2411.2bp) and (3684.0bp,2399.2bp)  .. (17);
			\draw [black,] (20) ..controls (3769.5bp,2631.3bp) and (3769.5bp,2619.4bp)  .. (19);
			\draw [black,] (21) ..controls (2843.6bp,2863.1bp) and (3023.3bp,2838.5bp)  .. (20);
			\draw [black,] (22) ..controls (4025.5bp,2852.3bp) and (3996.9bp,2839.4bp)  .. (20);
			\draw [black,] (23) ..controls (1397.8bp,3078.5bp) and (1561.5bp,3052.8bp)  .. (21);
			\draw [black,] (24) ..controls (5892.5bp,3008.8bp) and (5892.5bp,2903.0bp)  .. (18);
			\draw [black,] (24) ..controls (5187.3bp,3081.4bp) and (5006.4bp,3057.3bp)  .. (22);
			\draw [black,] (25) ..controls (2771.1bp,3074.1bp) and (2661.2bp,3052.7bp)  .. (21);
			\draw [black,] (25) ..controls (3722.5bp,3077.1bp) and (3799.0bp,3059.7bp)  .. (22);
			\draw [green,] (26) ..controls (5652.0bp,3292.8bp) and (5679.7bp,3279.5bp)  .. (24);
			\draw [black,] (26) ..controls (4550.0bp,3305.3bp) and (4277.0bp,3277.5bp)  .. (25);
			\draw [black,] (27) ..controls (3263.3bp,3289.5bp) and (3266.7bp,3274.8bp)  .. (25);
			\draw [black,] (28) ..controls (3210.3bp,3518.1bp) and (3213.8bp,3503.3bp)  .. (27);
			\draw [black,] (29) ..controls (972.13bp,3449.3bp) and (898.8bp,3343.5bp)  .. (23);
			\draw [black,] (29) ..controls (2003.0bp,3522.9bp) and (2284.3bp,3494.4bp)  .. (27);
			\draw [black,] (30) ..controls (5438.5bp,3518.1bp) and (5438.5bp,3503.3bp)  .. (26);
			\draw [black,] (30) ..controls (4504.0bp,3523.1bp) and (4211.5bp,3494.1bp)  .. (27);
			\draw [black,] (31) ..controls (3187.5bp,3730.1bp) and (3187.5bp,3715.4bp)  .. (28);
			\draw [green,] (31) ..controls (2282.0bp,3742.0bp) and (1997.1bp,3712.4bp)  .. (29);
			\draw [black,] (31) ..controls (4147.2bp,3743.0bp) and (4470.3bp,3711.7bp)  .. (30);
			\draw [green,] (32) ..controls (4451.2bp,3740.5bp) and (4134.6bp,3709.8bp)  .. (28);
			\draw [black,] (32) ..controls (5436.9bp,3735.8bp) and (5437.1bp,3718.1bp)  .. (30);
			\draw [black,] (33) ..controls (5435.5bp,3950.4bp) and (5435.5bp,3935.6bp)  .. (32);
			\draw [green,] (34) ..controls (3206.5bp,3953.1bp) and (3204.5bp,3941.2bp)  .. (31);
			\draw [black,] (34) ..controls (4214.5bp,3958.6bp) and (4490.8bp,3931.3bp)  .. (32);
			\draw [black,] (35) ..controls (2460.0bp,4180.5bp) and (2585.4bp,4160.4bp)  .. (34);
			\draw [black,] (36) ..controls (5435.5bp,4173.4bp) and (5435.5bp,4161.4bp)  .. (33);
			\draw [black,] (36) ..controls (4527.9bp,4187.1bp) and (4231.2bp,4157.8bp)  .. (34);
			\draw [black,] (37) ..controls (4389.6bp,4185.9bp) and (4631.5bp,4156.4bp)  .. (33);
			\draw [black,] (37) ..controls (3458.1bp,4179.7bp) and (3428.0bp,4164.2bp)  .. (34);
			\draw [black,] (38) ..controls (2835.8bp,4393.7bp) and (2633.1bp,4369.8bp)  .. (35);
			\draw [black,] (38) ..controls (4456.1bp,4393.7bp) and (4659.1bp,4369.9bp)  .. (36);
			\draw [black,] (38) ..controls (3645.5bp,4391.2bp) and (3645.5bp,4376.3bp)  .. (37);
			\draw [black,] (39) ..controls (3046.1bp,4606.4bp) and (3142.4bp,4587.9bp)  .. (38);
			\draw [black,] (40) ..controls (4236.4bp,4604.7bp) and (4144.8bp,4587.0bp)  .. (38);
			\draw [black,] (41) ..controls (4738.5bp,4815.6bp) and (4738.5bp,4800.8bp)  .. (40);
			\draw [green,] (42) ..controls (2652.6bp,4818.3bp) and (2642.1bp,4806.1bp)  .. (39);
			\draw [black,] (42) ..controls (3623.1bp,4825.2bp) and (3855.8bp,4799.8bp)  .. (40);
			\draw [black,] (43) ..controls (4191.0bp,5038.8bp) and (4273.0bp,5021.5bp)  .. (41);
			\draw [black,] (43) ..controls (3294.6bp,5040.1bp) and (3219.4bp,5024.2bp)  .. (42);
			\draw [black,] (44) ..controls (3739.5bp,5243.0bp) and (3739.5bp,5231.1bp)  .. (43);
			\draw [black,] (45) ..controls (3739.5bp,5447.1bp) and (3739.5bp,5435.1bp)  .. (44);
		\end{tikzpicture}
	}
\end{figure}

\begin{figure}
	\caption{Values of $\cl_n$ in E8, part I}\label{values:E8a}
	\vspace{\baselineskip}
	\scalebox{0.7}{
		\begin{tikzpicture}[>=latex',line join=bevel, every text node part/.style={align=center, font=\bfseries}, scale=0.7]
			\node (42) at (270.77bp,19.94bp) [draw=red,fill=blue,ellipse] {\twi{E8(a7)}  \\ \twi{[6]}};
			\node (43) at (384.77bp,95.82bp) [draw=red,fill=blue,ellipse] {\twi{A6} \\ \twi 7};
			\definecolor{fillcolor}{rgb}{0.82,0.94,1.0};
			\node (44) at (157.77bp,95.82bp) [draw=red,fill=fillcolor,ellipse] {D6(a1) \\ 8};
			\definecolor{fillcolor}{rgb}{0.82,0.94,1.0};
			\node (45) at (399.77bp,163.7bp) [draw=black,fill=fillcolor,ellipse] {A6+A1 \\ {[7]} };
			\definecolor{fillcolor}{rgb}{0.82,0.94,1.0};
			\node (46) at (157.77bp,163.7bp) [draw=black,fill=fillcolor,ellipse] {E7(a4) \\ 8};
			\node (47) at (157.77bp,307.46bp) [draw=black,fill=blue,ellipse] {\twi{E6(a1)} \\ \twi 9};
			\node (48) at (397.77bp,231.58bp) [draw=red,fill=blue,ellipse] {\twi{D5+A2}  \\ \twi 8};
			\node (49) at (602.77bp,307.46bp) [draw=red,fill=white,ellipse] {D6 \\ 10};
			\node (50) at (76.767bp,459.22bp) [draw=red,fill=blue,ellipse] {\twi{E6} \\ \twi{12}};
			\definecolor{fillcolor}{rgb}{0.82,0.94,1.0};
			\node (51) at (395.77bp,307.46bp) [draw=black,fill=fillcolor,ellipse] {D7(a2)  \\ 8};
			\node (52) at (208.77bp,383.34bp) [draw=red,fill=white,ellipse] {A7 \\ {[8]}};
			\definecolor{fillcolor}{rgb}{0.82,0.94,1.0};
			\node (53) at (452.77bp,383.34bp) [draw=black,fill=fillcolor,ellipse] {E6(a1)+A1  \\ 9};
			\definecolor{fillcolor}{rgb}{0.82,0.94,1.0};
			\node (54) at (546.77bp,459.22bp) [draw=red,fill=fillcolor,ellipse] {E7(a3) \\ 10};
			\node (55) at (299.77bp,459.22bp) [draw=red,fill=blue,ellipse] {\twi{E8(b6)}  \\ \twi{[9]} };
			\node (56) at (420.77bp,535.1bp) [draw=black,fill=blue,ellipse] {\twi{D7(a1)}  \\ \twi{10}};
			\node (57) at (187.77bp,535.1bp) [draw=red,fill=white,ellipse] {E6+A1 \\ 12};
			\node (58) at (187.77bp,610.98bp) [draw=red,fill=white,ellipse] {E7(a2) \\ 12};
			\node (59) at (427.77bp,610.98bp) [draw=black,fill=blue,ellipse] {\twi{E8(a6)}  \\ \twi{[10]}};
			\node (60) at (427.77bp,686.86bp) [draw=red,fill=white,ellipse] {D7 \\ 12};
			\node (61) at (214.77bp,686.86bp) [draw=red,fill=blue,ellipse] {\twi{E8(b5)}  \\ \twi{12}};
			\definecolor{fillcolor}{rgb}{0.82,0.94,1.0};
			\node (62) at (214.77bp,762.74bp) [draw=black,fill=fillcolor,ellipse] {E7(a1) \\ 14};
			\node (63) at (418.77bp,762.74bp) [draw=red,fill=blue,ellipse] {\twi{E8(a5)}  \\ \twi{[12]}};
			\node (64) at (316.77bp,830.62bp) [draw=black,fill=blue,ellipse] {\twi{E8(b4)}  \\ \twi{[14]}};
			\node (65) at (215.77bp,898.5bp) [draw=red,fill=white,ellipse] {E7 \\ 18};
			\node (66) at (417.77bp,898.5bp) [draw=black,fill=blue,ellipse] {\twi{E8(a4)}  \\ \twi{[15]}};
			\node (67) at (316.77bp,974.38bp) [draw=red,fill=blue,ellipse] {\twi{E8(a3)}  \\ \twi{[18]}};
			\node (68) at (316.77bp,1042.3bp) [draw=black,fill=blue,ellipse] {\twi{E8(a2)}  \\ \twi{[20]}};
			\node (69) at (316.77bp,1102.1bp) [draw=black,fill=blue,ellipse] {\twi{E8(a1)}  \\ \twi{[24]}};
			\node (70) at (316.77bp,1170.0bp) [draw=red,fill=blue,ellipse] {\twi{E8}  \\ \twi{[30]}};
			\draw [black,] (43) ..controls (339.02bp,65.169bp) and (317.02bp,50.918bp)  .. (42);
			\draw [black,] (44) ..controls (203.41bp,64.977bp) and (225.13bp,50.775bp)  .. (42);
			\draw [black,] (45) ..controls (395.02bp,141.83bp) and (391.77bp,127.59bp)  .. (43);
			\draw [black,] (46) ..controls (231.55bp,141.28bp) and (288.46bp,124.77bp)  .. (43);
			\draw [black,] (46) ..controls (157.77bp,141.83bp) and (157.77bp,127.59bp)  .. (44);
			\draw [black,] (47) ..controls (157.77bp,268.84bp) and (157.77bp,202.22bp)  .. (46);
			\draw [black,] (48) ..controls (398.71bp,199.63bp) and (399.14bp,185.41bp)  .. (45);
			\draw [black,] (48) ..controls (294.34bp,202.19bp) and (234.74bp,185.83bp)  .. (46);
			\draw [black,] (49) ..controls (525.61bp,278.65bp) and (480.13bp,262.26bp)  .. (48);
			\draw [black,] (50) ..controls (78.204bp,418.69bp) and (82.212bp,386.98bp)  .. (95.767bp,363.4bp) .. controls (106.38bp,344.95bp) and (125.44bp,329.47bp)  .. (47);
			\draw [black,] (51) ..controls (396.38bp,283.77bp) and (396.87bp,265.75bp)  .. (48);
			\draw [black,] (52) ..controls (288.72bp,350.75bp) and (337.7bp,331.4bp)  .. (51);
			\draw [black,] (53) ..controls (354.1bp,357.63bp) and (254.87bp,332.78bp)  .. (47);
			\draw [black,] (53) ..controls (433.4bp,357.23bp) and (415.26bp,333.72bp)  .. (51);
			\draw [green,] (54) ..controls (569.71bp,429.31bp) and (578.37bp,416.23bp)  .. (583.77bp,403.28bp) .. controls (594.17bp,378.33bp) and (599.01bp,347.26bp)  .. (49);
			\draw [black,] (54) ..controls (505.42bp,425.73bp) and (481.96bp,407.28bp)  .. (53);
			\draw [green,] (55) ..controls (262.64bp,428.08bp) and (245.43bp,414.1bp)  .. (52);
			\draw [black,] (55) ..controls (366.84bp,425.83bp) and (405.48bp,407.17bp)  .. (53);
			\draw [black,] (56) ..controls (460.01bp,511.09bp) and (491.94bp,492.37bp)  .. (54);
			\draw [black,] (56) ..controls (383.35bp,511.25bp) and (353.26bp,492.88bp)  .. (55);
			\draw [black,] (57) ..controls (142.67bp,504.08bp) and (120.94bp,489.62bp)  .. (50);
			\draw [black,] (57) ..controls (232.88bp,504.34bp) and (254.21bp,490.27bp)  .. (55);
			\draw [black,] (58) ..controls (285.0bp,579.15bp) and (348.5bp,559.01bp)  .. (56);
			\draw [green,] (58) ..controls (187.77bp,579.71bp) and (187.77bp,566.22bp)  .. (57);
			\draw [black,] (59) ..controls (425.39bp,584.87bp) and (423.16bp,561.36bp)  .. (56);
			\draw [black,] (60) ..controls (427.77bp,652.59bp) and (427.77bp,634.46bp)  .. (59);
			\draw [green,] (61) ..controls (203.81bp,655.89bp) and (198.8bp,642.17bp)  .. (58);
			\draw [black,] (61) ..controls (305.41bp,654.42bp) and (361.63bp,634.92bp)  .. (59);
			\draw [black,] (62) ..controls (214.77bp,739.05bp) and (214.77bp,721.03bp)  .. (61);
			\draw [green,] (63) ..controls (422.45bp,731.47bp) and (424.1bp,717.98bp)  .. (60);
			\draw [black,] (63) ..controls (338.85bp,732.79bp) and (295.66bp,717.15bp)  .. (61);
			\draw [black,] (64) ..controls (280.59bp,806.25bp) and (250.58bp,786.87bp)  .. (62);
			\draw [black,] (64) ..controls (349.42bp,808.53bp) and (372.13bp,793.86bp)  .. (63);
			\draw [black,] (65) ..controls (261.35bp,867.76bp) and (284.26bp,852.82bp)  .. (64);
			\draw [black,] (66) ..controls (382.06bp,874.21bp) and (352.59bp,854.98bp)  .. (64);
			\draw [green,] (67) ..controls (275.74bp,943.36bp) and (255.96bp,928.9bp)  .. (65);
			\draw [black,] (67) ..controls (361.19bp,940.88bp) and (386.4bp,922.44bp)  .. (66);
			\draw [black,] (68) ..controls (316.77bp,1020.4bp) and (316.77bp,1006.1bp)  .. (67);
			\draw [black,] (69) ..controls (316.77bp,1079.6bp) and (316.77bp,1064.8bp)  .. (68);
			\draw [black,] (70) ..controls (316.77bp,1138.1bp) and (316.77bp,1123.8bp)  .. (69);
		\end{tikzpicture}
	}
\end{figure}
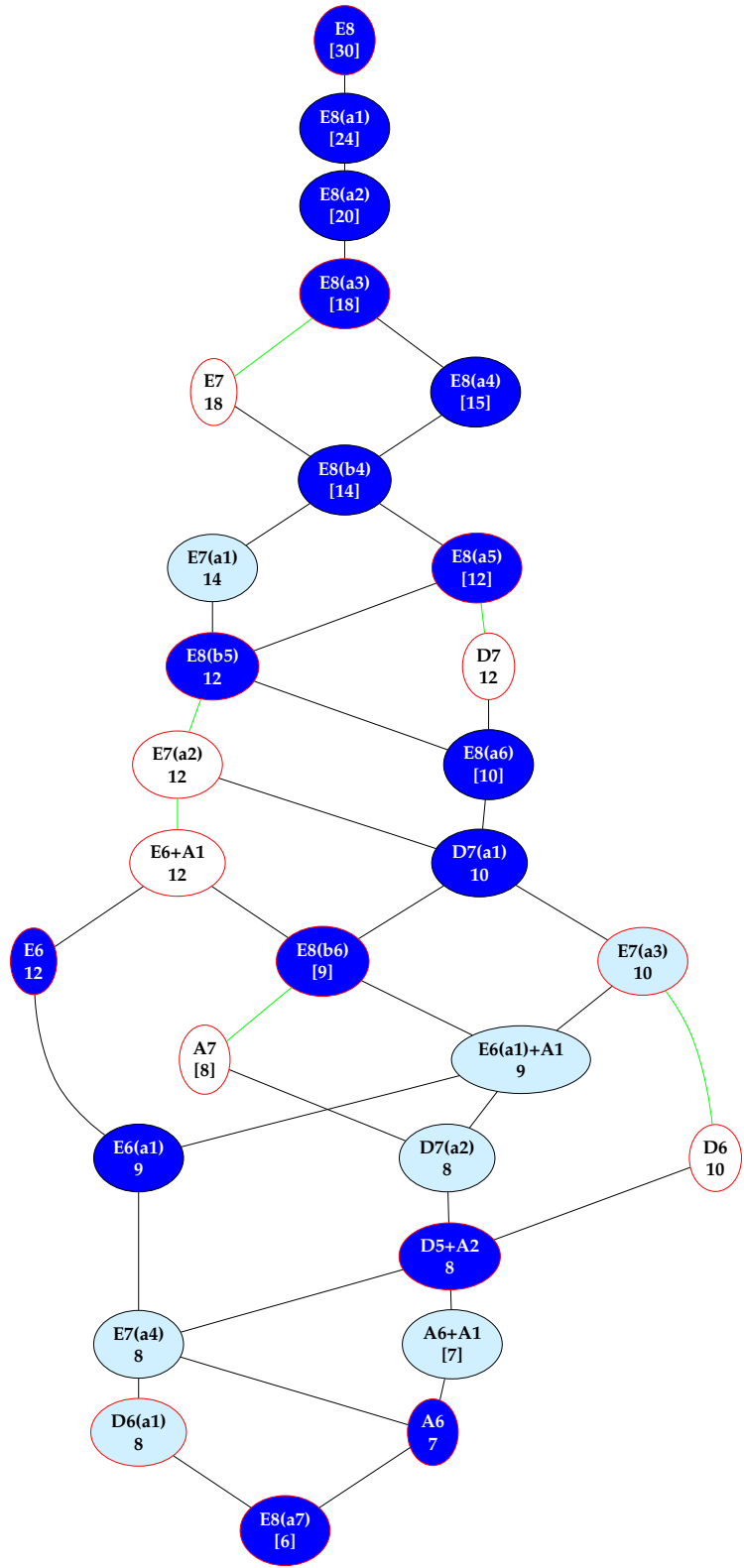

\begin{figure}
	\caption{Values of $\cl_n$ in E8, part II}\label{values:E8b}
	\vspace{\baselineskip}
	\scalebox{0.55}{
		\begin{tikzpicture}[>=latex',line join=bevel,  every text node part/.style={align=center, font=\bfseries}, scale=0.6]
			\node (1) at (232.64bp,19.94bp) [draw=red,fill=blue,ellipse] {\twi 0  \\  \twi{[1]}};
			\definecolor{fillcolor}{rgb}{0.82,0.94,1.0};
			\node (2) at (232.64bp,95.82bp) [draw=red,fill=fillcolor,ellipse] {A1  \\  2};
			\definecolor{fillcolor}{rgb}{0.82,0.94,1.0};
			\node (3) at (232.64bp,171.7bp) [draw=red,fill=fillcolor,ellipse] {2A1  \\ 2};
			\node (4) at (232.64bp,247.58bp) [draw=red,fill=white,ellipse] {3A1  \\  2};
			\node (5) at (123.64bp,323.46bp) [draw=red,fill=blue,ellipse] {\twi{A2} \\ \twi 3};
			\node (6) at (342.64bp,323.46bp) [draw=red,fill=white,ellipse] {4A1  \\  {[2]}};
			\definecolor{fillcolor}{rgb}{0.82,0.94,1.0};
			\node (7) at (232.64bp,399.34bp) [draw=red,fill=fillcolor,ellipse] {A2+A1  \\  3};
			\definecolor{fillcolor}{rgb}{0.82,0.94,1.0};
			\node (8) at (232.64bp,475.22bp) [draw=red,fill=fillcolor,ellipse] {A2+2A1  \\  3};
			\definecolor{fillcolor}{rgb}{0.82,0.94,1.0};
			\node (9) at (126.64bp,626.98bp) [draw=black,fill=fillcolor,ellipse] {A3 \\  4};
			\node (10) at (320.64bp,551.1bp) [draw=red,fill=white,ellipse] {A2+3A1  \\  3};
			\node (11) at (339.64bp,626.98bp) [draw=red,fill=blue,ellipse] {\twi{2A2} \\ \twi 3};
			\node (12) at (339.64bp,702.86bp) [draw=red,fill=white,ellipse] {2A2+A1  \\  3};
			\node (13) at (126.64bp,778.74bp) [draw=red,fill=white,ellipse] {A3+A1  \\ 4};
			\node (14) at (118.64bp,854.62bp) [draw=red,fill=blue,ellipse] {\twi{D4(a1)} \\ \twi 4};
			\node (15) at (118.64bp,1006.4bp) [draw=red,fill=blue,ellipse] {\twi{D4} \\ \twi 6};
			\node (16) at (364.64bp,778.74bp) [draw=red,fill=white,ellipse] {2A2+2A1  \\  {[3]}};
			\node (17) at (364.64bp,854.62bp) [draw=red,fill=white,ellipse] {A3+2A1  \\ 4};
			\definecolor{fillcolor}{rgb}{0.82,0.94,1.0};
			\node (18) at (349.64bp,930.5bp) [draw=red,fill=fillcolor,ellipse] {D4(a1)+A1  \\ 4};
			\definecolor{fillcolor}{rgb}{0.82,0.94,1.0};
			\node (19) at (349.64bp,1006.4bp) [draw=red,fill=fillcolor,ellipse] {A3+A2 \\  4};
			\node (20) at (250.64bp,1082.3bp) [draw=black,fill=blue,ellipse] {\twi{A4} \\ \twi 5};
			\node (21) at (487.64bp,1082.3bp) [draw=red,fill=white,ellipse] {A3+A2+A1  \\ 4};
			\node (22) at (118.64bp,1158.1bp) [draw=red,fill=white,ellipse] {D4+A1 \\ 6};
			\node (23) at (487.64bp,1158.1bp) [draw=red,fill=blue,ellipse] {\twi{D4(a1)+A2} \\ \twi 4};
			\definecolor{fillcolor}{rgb}{0.82,0.94,1.0};
			\node (24) at (262.64bp,1226.0bp) [draw=black,fill=fillcolor,ellipse] {A4+A1  \\ 5};
			\node (25) at (494.64bp,1226.0bp) [draw=black,fill=white,ellipse] {2A3  \\  {[4]}};
			\definecolor{fillcolor}{rgb}{0.82,0.94,1.0};
			\node (26) at (175.64bp,1293.9bp) [draw=red,fill=fillcolor,ellipse] {D5(a1) \\ 6};
			\definecolor{fillcolor}{rgb}{0.82,0.94,1.0};
			\node (27) at (427.64bp,1293.9bp) [draw=black,fill=fillcolor,ellipse] {A4+2A1  \\ 5};
			\node (28) at (420.64bp,1369.8bp) [draw=red,fill=blue,ellipse] {\twi{A4+A2} \\ \twi{5}};
			\node (29) at (358.64bp,1505.5bp) [draw=black,fill=white,ellipse] {A5 \\ 6};
			\definecolor{fillcolor}{rgb}{0.82,0.94,1.0};
			\node (30) at (255.64bp,1437.7bp) [draw=black,fill=fillcolor,ellipse] {D5(a1)+A1 \\ 6};
			\definecolor{fillcolor}{rgb}{0.82,0.94,1.0};
			\node (31) at (567.64bp,1437.7bp) [draw=black,fill=fillcolor,ellipse] {A4+A2+A1 \\ 5};
			\node (32) at (791.64bp,1505.5bp) [draw=red,fill=blue,ellipse] {\twi{D4+A2} \\ \twi 6};
			\node (33) at (279.64bp,1581.4bp) [draw=black,fill=blue,ellipse] {\twi{E6(a3)} \\ \twi 6};
			\node (34) at (279.64bp,1733.2bp) [draw=red,fill=blue,ellipse] {\twi{D5} \\ \twi 8};
			\node (35) at (567.64bp,1505.5bp) [draw=red,fill=white,ellipse] {A4+A3  \\  {[5]}};
			\node (36) at (513.64bp,1581.4bp) [draw=red,fill=white,ellipse] {A5+A1  \\ 6};
			\node (37) at (756.64bp,1581.4bp) [draw=red,fill=white,ellipse] {D5(a1)+A2  \\ 6};
			\node (38) at (756.64bp,1657.3bp) [draw=red,fill=white,ellipse] {D6(a2) \\ 6};
			\node (39) at (507.64bp,1657.3bp) [draw=red,fill=white,ellipse] {E6(a3)+A1 \\ 6};
			\node (40) at (514.64bp,1733.2bp) [draw=red,fill=white,ellipse] {E7(a5) \\ 6};
			\node (41) at (279.64bp,1809.1bp) [draw=red,fill=white,ellipse] {D5+A1 \\ 8};
			\node (42) at (521.64bp,1809.1bp) [draw=red,fill=blue,ellipse] {\twi{E8(a7)}  \\ \twi{[6]}};
			\definecolor{fillcolor}{rgb}{0.82,0.94,1.0};
			\node (44) at (400.64bp,1884.9bp) [draw=red,fill=fillcolor,ellipse] {D6(a1) \\ 8};
			\draw [black,] (2) ..controls (232.64bp,64.549bp) and (232.64bp,51.059bp)  .. (1);
			\draw [black,] (3) ..controls (232.64bp,140.43bp) and (232.64bp,126.94bp)  .. (2);
			\draw [black,] (4) ..controls (232.64bp,216.31bp) and (232.64bp,202.82bp)  .. (3);
			\draw [green,] (5) ..controls (167.8bp,292.53bp) and (188.95bp,278.2bp)  .. (4);
			\draw [black,] (6) ..controls (298.37bp,292.72bp) and (276.95bp,278.34bp)  .. (4);
			\draw [black,] (7) ..controls (188.33bp,368.3bp) and (167.45bp,354.15bp)  .. (5);
			\draw [green,] (7) ..controls (277.36bp,368.3bp) and (298.43bp,354.15bp)  .. (6);
			\draw [black,] (8) ..controls (232.64bp,443.95bp) and (232.64bp,430.46bp)  .. (7);
			\draw [black,] (9) ..controls (152.62bp,589.28bp) and (196.63bp,527.1bp)  .. (8);
			\draw [black,] (10) ..controls (285.2bp,520.34bp) and (268.44bp,506.27bp)  .. (8);
			\draw [green,] (11) ..controls (331.86bp,595.71bp) and (328.39bp,582.22bp)  .. (10);
			\draw [black,] (12) ..controls (339.64bp,671.59bp) and (339.64bp,658.1bp)  .. (11);
			\draw [black,] (13) ..controls (126.64bp,726.71bp) and (126.64bp,664.38bp)  .. (9);
			\draw [green,] (13) ..controls (210.33bp,748.71bp) and (255.83bp,732.93bp)  .. (12);
			\draw [green,] (14) ..controls (121.92bp,823.35bp) and (123.38bp,809.86bp)  .. (13);
			\draw [black,] (15) ..controls (118.64bp,957.35bp) and (118.64bp,903.53bp)  .. (14);
			\draw [black,] (16) ..controls (354.5bp,747.77bp) and (349.86bp,734.05bp)  .. (12);
			\draw [black,] (17) ..controls (271.55bp,824.72bp) and (218.97bp,808.4bp)  .. (13);
			\draw [green,] (17) ..controls (364.64bp,823.35bp) and (364.64bp,809.86bp)  .. (16);
			\draw [black,] (18) ..controls (257.64bp,900.07bp) and (209.17bp,884.57bp)  .. (14);
			\draw [green,] (18) ..controls (355.79bp,899.23bp) and (358.53bp,885.74bp)  .. (17);
			\draw [black,] (19) ..controls (349.64bp,975.11bp) and (349.64bp,961.62bp)  .. (18);
			\draw [black,] (20) ..controls (281.26bp,1058.4bp) and (305.88bp,1040.0bp)  .. (19);
			\draw [black,] (21) ..controls (432.1bp,1051.5bp) and (405.23bp,1037.1bp)  .. (19);
			\draw [black,] (22) ..controls (118.64bp,1109.1bp) and (118.64bp,1055.3bp)  .. (15);
			\draw [black,] (22) ..controls (253.77bp,1130.1bp) and (347.44bp,1111.3bp)  .. (21);
			\draw [green,] (23) ..controls (487.64bp,1126.9bp) and (487.64bp,1113.4bp)  .. (21);
			\draw [black,] (24) ..controls (259.46bp,1187.4bp) and (253.82bp,1120.8bp)  .. (20);
			\draw [black,] (24) ..controls (334.87bp,1203.9bp) and (388.03bp,1188.3bp)  .. (23);
			\draw [black,] (25) ..controls (492.43bp,1204.1bp) and (490.91bp,1189.9bp)  .. (23);
			\draw [green,] (26) ..controls (144.35bp,1264.4bp) and (133.68bp,1251.6bp)  .. (127.64bp,1238.0bp) .. controls (119.24bp,1218.9bp) and (117.62bp,1194.9bp)  .. (22);
			\draw [black,] (26) ..controls (215.61bp,1262.6bp) and (234.71bp,1248.2bp)  .. (24);
			\draw [black,] (27) ..controls (369.57bp,1269.7bp) and (320.84bp,1250.3bp)  .. (24);
			\draw [green,] (27) ..controls (451.4bp,1269.5bp) and (471.12bp,1250.1bp)  .. (25);
			\draw [black,] (28) ..controls (423.8bp,1335.5bp) and (425.51bp,1317.4bp)  .. (27);
			\draw [black,] (29) ..controls (374.59bp,1470.1bp) and (398.76bp,1418.0bp)  .. (28);
			\draw [black,] (30) ..controls (235.74bp,1401.4bp) and (203.31bp,1343.9bp)  .. (26);
			\draw [black,] (30) ..controls (308.58bp,1415.5bp) and (346.53bp,1400.4bp)  .. (28);
			\draw [black,] (31) ..controls (520.64bp,1415.6bp) and (487.17bp,1400.6bp)  .. (28);
			\draw [black,] (32) ..controls (707.11bp,1489.3bp) and (694.53bp,1487.3bp)  .. (682.64bp,1485.6bp) .. controls (563.76bp,1468.7bp) and (425.52bp,1454.5bp)  .. (30);
			\draw [black,] (32) ..controls (695.66bp,1476.3bp) and (640.36bp,1460.0bp)  .. (31);
			\draw [green,] (33) ..controls (306.58bp,1555.2bp) and (331.92bp,1531.5bp)  .. (29);
			\draw [black,] (33) ..controls (272.56bp,1558.2bp) and (267.52bp,1540.8bp)  .. (264.64bp,1525.5bp) .. controls (259.63bp,1498.7bp) and (257.22bp,1466.6bp)  .. (30);
			\draw [black,] (34) ..controls (279.64bp,1681.1bp) and (279.64bp,1618.8bp)  .. (33);
			\draw [black,] (35) ..controls (567.64bp,1473.6bp) and (567.64bp,1459.4bp)  .. (31);
			\draw [black,] (36) ..controls (446.15bp,1548.2bp) and (406.06bp,1529.1bp)  .. (29);
			\draw [green,] (36) ..controls (535.55bp,1550.4bp) and (545.57bp,1536.7bp)  .. (35);
			\draw [black,] (37) ..controls (770.84bp,1550.4bp) and (777.34bp,1536.7bp)  .. (32);
			\draw [green,] (37) ..controls (681.34bp,1551.0bp) and (642.11bp,1535.6bp)  .. (35);
			\draw [green,] (38) ..controls (662.09bp,1627.5bp) and (607.97bp,1611.1bp)  .. (36);
			\draw [green,] (38) ..controls (756.64bp,1626.0bp) and (756.64bp,1612.5bp)  .. (37);
			\draw [black,] (39) ..controls (410.49bp,1624.8bp) and (350.07bp,1605.2bp)  .. (33);
			\draw [green,] (39) ..controls (510.1bp,1626.0bp) and (511.2bp,1612.5bp)  .. (36);
			\draw [green,] (39) ..controls (605.58bp,1627.2bp) and (659.58bp,1611.2bp)  .. (37);
			\draw [green,] (40) ..controls (608.67bp,1703.5bp) and (662.33bp,1687.1bp)  .. (38);
			\draw [green,] (40) ..controls (511.78bp,1701.9bp) and (510.5bp,1688.4bp)  .. (39);
			\draw [black,] (41) ..controls (279.64bp,1777.8bp) and (279.64bp,1764.3bp)  .. (34);
			\draw [black,] (41) ..controls (371.12bp,1779.3bp) and (422.9bp,1763.0bp)  .. (40);
			\draw [green,] (42) ..controls (518.78bp,1777.8bp) and (517.5bp,1764.3bp)  .. (40);
			\draw [green,] (44) ..controls (351.63bp,1854.0bp) and (328.15bp,1839.7bp)  .. (41);
			\draw [black,] (44) ..controls (449.52bp,1854.1bp) and (472.78bp,1839.9bp)  .. (42);
		\end{tikzpicture}
	}
\end{figure}
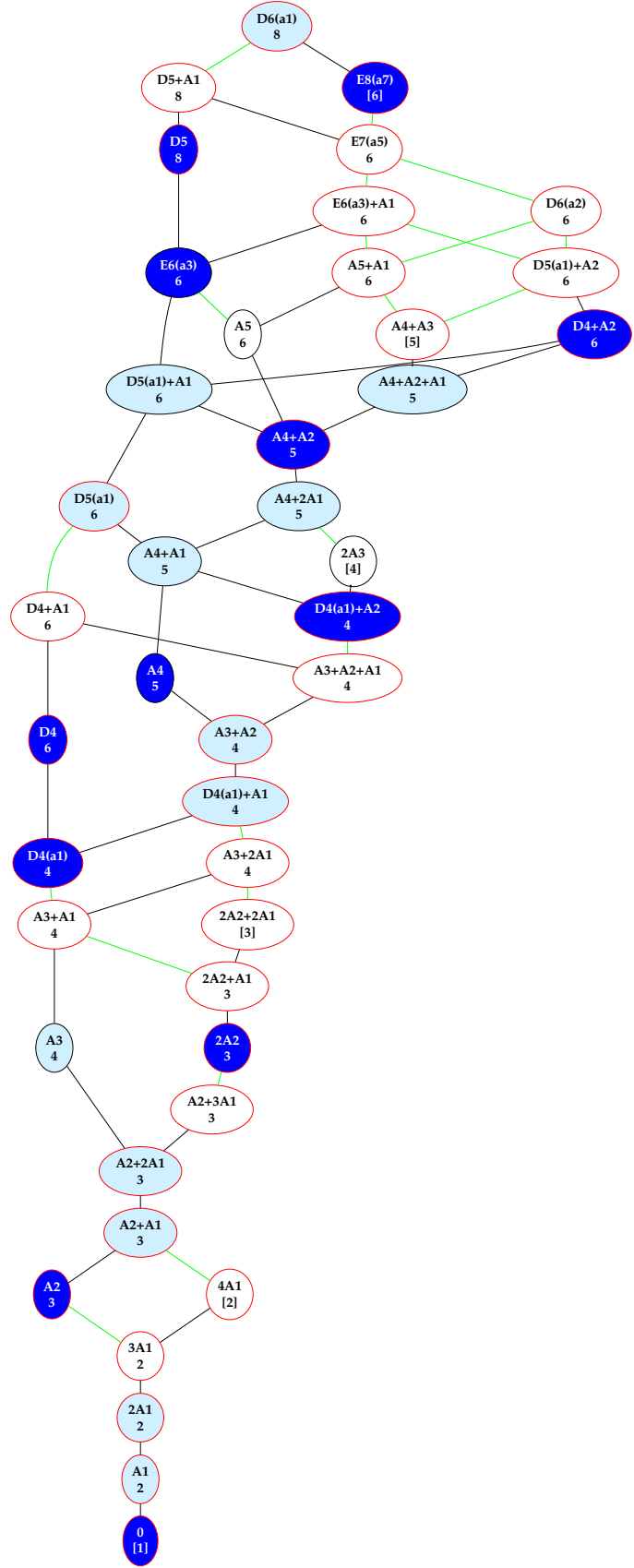

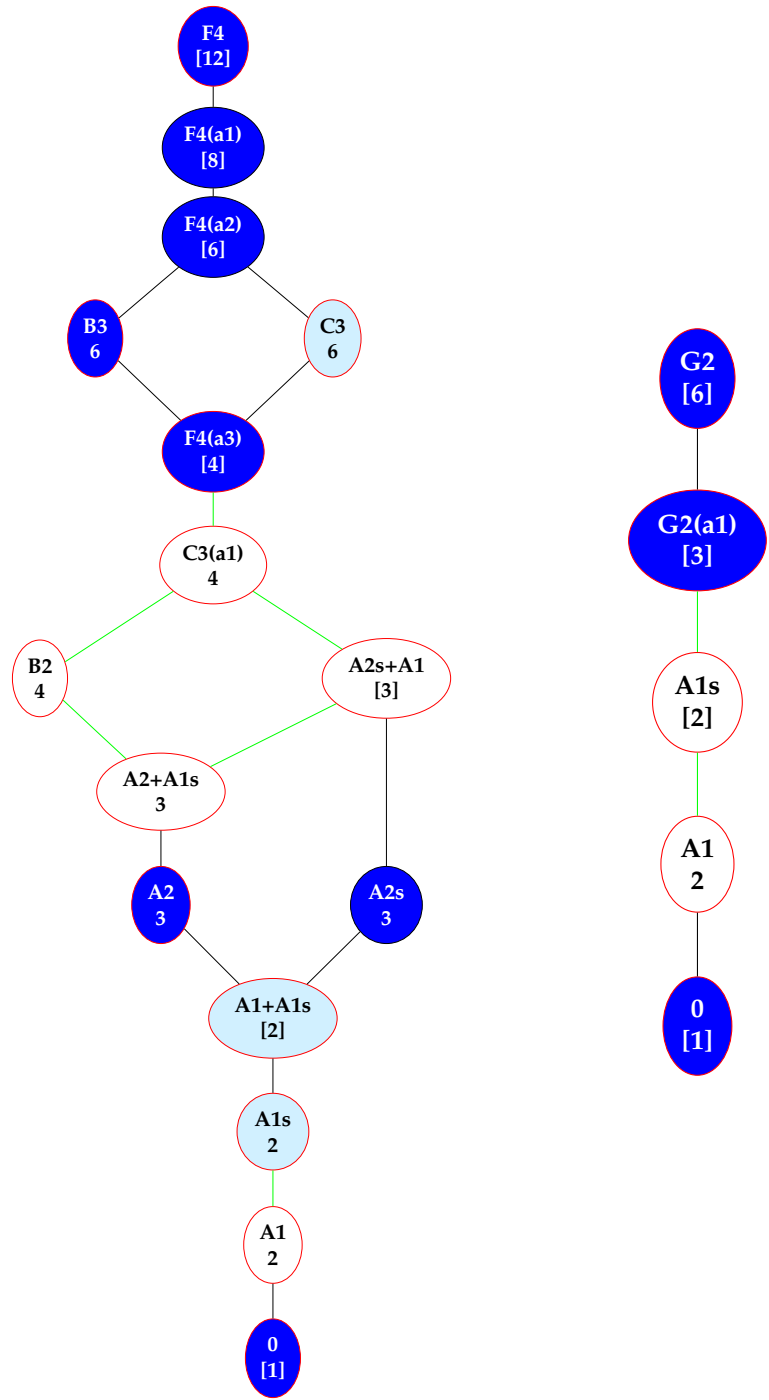
\begin{figure}
	\caption{Values of $\cl_n$ in F4 and G2}\label{values:FG}
	\vspace{\baselineskip}
	\begin{minipage}{0.45\linewidth}
		\scalebox{0.8}{
			\begin{tikzpicture}[>=latex',line join=bevel, every text node part/.style={align=center, font=\bfseries}, scale=0.7]
				\node (1) at (250.79bp,19.94bp) [draw=red,fill=blue,ellipse] {\twi 0 \\ \twi{[1]}};
				\node (2) at (250.79bp,95.82bp) [draw=red,fill=white,ellipse] {A1 \\ 2};
				\definecolor{fillcolor}{rgb}{0.82,0.94,1.0};
				\node (3) at (250.79bp,171.7bp) [draw=red,fill=fillcolor,ellipse] {A1s\\2};
				\definecolor{fillcolor}{rgb}{0.82,0.94,1.0};
				\node (4) at (250.79bp,247.58bp) [draw=red,fill=fillcolor,ellipse] {A1+A1s\\{[2]}};
				\node (5) at (175.79bp,323.46bp) [draw=red,fill=blue,ellipse] {\twi{A2}\\ \twi 3};
				\node (6) at (326.79bp,323.46bp) [draw=black,fill=blue,ellipse] {\twi{A2s}\\ \twi 3};
				\node (7) at (175.79bp,399.34bp) [draw=red,fill=white,ellipse] {A2+A1s\\3};
				\node (8) at (94.79bp,475.22bp) [draw=red,fill=white,ellipse] {B2\\4};
				\node (9) at (326.79bp,475.22bp) [draw=red,fill=white,ellipse] {A2s+A1\\{[3]}};
				\node (10) at (210.79bp,551.1bp) [draw=red,fill=white,ellipse] {C3(a1)\\4};
				\node (11) at (210.79bp,626.98bp) [draw=red,fill=blue,ellipse] {\twi{F4(a3)}\\ \twi{[4]}};
				\node (12) at (131.79bp,702.86bp) [draw=red,fill=blue,ellipse] {\twi{B3}\\ \twi 6};
				\definecolor{fillcolor}{rgb}{0.82,0.94,1.0};
				\node (13) at (290.79bp,702.86bp) [draw=red,fill=fillcolor,ellipse] {C3\\6};
				\node (14) at (210.79bp,770.74bp) [draw=black,fill=blue,ellipse] {\twi{F4(a2)}\\ \twi{[6]}};
				\node (15) at (210.79bp,830.62bp) [draw=black,fill=blue,ellipse] {\twi{F4(a1)}\\ \twi{[8]}};
				\node (16) at (210.79bp,898.5bp) [draw=red,fill=blue,ellipse] {\twi{F4} \\ \twi{[12]}};
				\draw [black,] (2) ..controls (250.79bp,64.549bp) and (250.79bp,51.059bp)  .. (1);
				\draw [green] (3) ..controls (250.79bp,140.43bp) and (250.79bp,126.94bp)  .. (2);
				\draw [black,] (4) ..controls (250.79bp,216.31bp) and (250.79bp,202.82bp)  .. (3);
				\draw [black,] (5) ..controls (205.91bp,292.79bp) and (220.06bp,278.85bp)  .. (4);
				\draw [black,] (6) ..controls (303.37bp,299.69bp) and (284.65bp,281.49bp)  .. (4);
				\draw [black,] (7) ..controls (175.79bp,368.07bp) and (175.79bp,354.58bp)  .. (5);
				\draw [green] (8) ..controls (127.41bp,444.46bp) and (142.84bp,430.39bp)  .. (7);
				\draw [black,] (9) ..controls (326.79bp,423.19bp) and (326.79bp,360.86bp)  .. (6);
				\draw [green] (9) ..controls (266.23bp,444.59bp) and (236.46bp,430.02bp)  .. (7);
				\draw [green] (10) ..controls (163.8bp,520.17bp) and (141.29bp,505.84bp)  .. (8);
				\draw [green] (10) ..controls (257.65bp,520.26bp) and (279.94bp,506.05bp)  .. (9);
				\draw [green] (11) ..controls (210.79bp,595.71bp) and (210.79bp,582.22bp)  .. (10);
				\draw [black,] (12) ..controls (163.2bp,672.48bp) and (178.35bp,658.32bp)  .. (11);
				\draw [black,] (13) ..controls (258.98bp,672.48bp) and (243.64bp,658.32bp)  .. (11);
				\draw [black,] (14) ..controls (185.41bp,748.58bp) and (167.65bp,733.76bp)  .. (12);
				\draw [black,] (14) ..controls (236.49bp,748.58bp) and (254.48bp,733.76bp)  .. (13);
				\draw [black,] (15) ..controls (210.79bp,808.03bp) and (210.79bp,793.29bp)  .. (14);
				\draw [black,] (16) ..controls (210.79bp,866.55bp) and (210.79bp,852.33bp)  .. (15);
			\end{tikzpicture}
		}
	\end{minipage}
	\begin{minipage}{0.3\linewidth}
		\qquad
		\begin{tikzpicture}[>=latex',line join=bevel, every text node part/.style={align=center, font=\bfseries}, scale=0.8]
			\node (1) at (46.113bp,19.94bp) [draw=red,fill=blue,ellipse] {\twi 0 \\ \twi{[1]}};
			\node (2) at (46.113bp,95.82bp) [draw=red,fill=white,ellipse] {A1 \\ 2 };
			\node (3) at (46.113bp,171.7bp) [draw=red,fill=white,ellipse] {A1s \\ {[2]}};
			\node (4) at (46.113bp,247.58bp) [draw=red,fill=blue,ellipse] {\twi{ G2(a1)} \\ \twi{[3]}};
			\node (5) at (46.113bp,323.46bp) [draw=red,fill=blue,ellipse] {\twi{ G2} \\ \twi {[6]}};
			\draw [black,] (2) ..controls (46.113bp,64.549bp) and (46.113bp,51.059bp)  .. (1);
			\draw [green,] (3) ..controls (46.113bp,140.43bp) and (46.113bp,126.94bp)  .. (2);
			\draw [green,] (4) ..controls (46.113bp,216.31bp) and (46.113bp,202.82bp)  .. (3);
			\draw [black,] (5) ..controls (46.113bp,292.19bp) and (46.113bp,278.7bp)  .. (4);
		\end{tikzpicture}
	\end{minipage}
\end{figure}


\subsection{Affine Lie algebras and affine vertex algebras}\label{subsec:aff-alg}

\begin{clause}[Notations on affine Lie algebras]\label{cls:aff-alg-notations}	
	Let $F = \BC \laurent{t}$ be the field of Laurent series and let $\cO = \BC \series{t}$ be the ring of power series. We have the valuation map $\val: \cO \to \BZ_{\ge 0} \cup \{\infty\}$ sending a nonzero element $f \in \cO$ to the lowest exponent of $t$ appearing in $f$ and $0 \in \cO$ to $\infty$. This can be extended to a valuation on $F$. Fix the algebraic closure of $F$ to be
	\begin{equation*}
		\bar F = \bigcup_{k \ge 0} \BC \laurent{t^{1/k}}.
	\end{equation*}
	There is a unique automorphism $\sigma \in \Gal(\bar F/F)$ sending $t^{1/k}$ to $e^{2\pi i/k} t^{1/k}$.
	The valuation map can also be extended to $\bar F$ by allowing values in rational numbers and set $\val(t^{1/k})=1/k$.
	
	We write $G(F)$, $G(\cO)$, $I$, and $I^+$ for the loop group, the arc group, the Iwahori subgroup corresponding to $B$, and the unipotent radical of $I$. Their Lie algebras are denoted by $\fg(F)$, $\fg(\cO)$, $\fI$, and $\fI^+$, respectively. The \textit{affine Lie algebra} is a central extension
	\begin{equation*}
		0 \aro \BC K \aro \fg_{aff} \aro \fg(F) \aro 0
	\end{equation*}
	and the \textit{affine Kac-Moody algebra} is a further semi-direct product
	\begin{equation*}
		\fg_{KM} = \fg_{aff} \rtimes \BC d
	\end{equation*}
	where the bracket is given by
	\begin{equation*}
		[at^m, bt^n] = [a,b] t^{m+n} + (a,b) m \delta_{m,-n} K,\quad
		[K, \fg_{KM}] = 0,\quad
		[d, at^n] = a nt^n.
	\end{equation*}
	We fix Cartan subalgebras
	\begin{align*}
		\fh_{aff} &= \fh \oplus \BC K,\\
		\fh_{KM} &= \fh \oplus \BC K \oplus \BC d.
	\end{align*}
	We write the dual of $\fh_{KM}$ as
	\begin{equation*}
		\fh_{KM}^* = \fh^* \oplus \BC \Lambda_0 \oplus \BC \delta
	\end{equation*}
	where $\Lambda_0$ and $\delta$ are dual elements to $K$ and $d$, respectively. An element $\Lambda \in \fh_{aff}^*$ is said to have \textit{level $k \in \BC$} if $\Lambda \in k\Lambda_0 + \fh^*$.
	
	Let $W_{aff}$ be the affine Weyl group of $G$. Since $G$ is assumed to be simply-connected, $W_{aff}$ is a Coxeter group. It acts naturally on $\fh_{aff}^*$.
	
	Let $\check G$ be the Langlands dual group of $G$ with Lie algebra $\check \fg$.  
\end{clause}


\begin{clause}[Affine vertex algebras]
	For the definition of vertex algebras and its properties, we refer readers to \cite{Kac:VA, Frenkel-Ben-Zvi, Arakawa:W-alg-notes}. Given a vertex algebra $V$, we write $\vac \in V$ for the vacuum vector, and for $A \in V$ write $Y(A,z) = \sum_{n \in \BZ} A_{(n)} z^{-n-1} \in \End(V) \series{z,z\inv}$ for the corresponding field.
	
	Fix a number $k \in \BC$. Let
	\begin{equation*}
		V^k = V^k(\fg) = \cU(\fg_{aff}) \dotimes_{\cU(\fg(\cO)) \oplus \BC K} \BC_k,
	\end{equation*}
	where $\fg(\cO)$ acts trivially on $\BC_k$ and $K$ acts on $\BC_k$ by $k$. This is a highest weight module of $\fg_{aff}$ with highest weight $k \Lambda_0$ of level $k$. It is spanned by the elements $\vac = 1 \otimes 1$ and $a_1 t^{-n_1} \cdots a_m t^{-n_m} \vac$ for $n_i \ge 1$,where $a_1,\ldots, a_m$ are elements of $\fg$. 
	It has the structure of a vertex algebra (see, for example, \cite{Frenkel-Zhu} or \cite[\textsection 2.4]{Frenkel-Ben-Zvi}).
	The unique irreducible quotient of $V^k$ as a $\fg_{aff}$-module inherits a structure of a vertex algebra from the one on $V^k$. We denote the resulting vertex algebra by 
	\begin{equation*}
		L_k = L_k(\fg).
	\end{equation*}
\end{clause}

\subsection{Associated varieties}\label{subsec:AV}

The associated variety of a vertex algebra is defined in a similar spirit to the associated varieties of associated algebras. It may be helpful to review the latter in order to draw the similarities between the two notions. Suppose we have an associative Noetherian $\BC$-algebra $A$ with generators $\{a_i\}_{i \in I}$. Then $A$ is equipped with an increasing filtration $F_0 A \subseteq F_1 A \subseteq \cdots$, where
\begin{equation*}
	F_p A = \operatorname{span}_\BC \{ \text{a product of $\le p$ generators} \}.
\end{equation*}
Suppose moreover that $F_\bullet A$ is \textit{almost commutative}, i.e. that $\gr^F A$ is a commutative ring. Let $M$ be a finitely generated $A$-module. Then for any choice of a good filtration $F_\bullet M$ on $M$, we may define the \textit{associative variety of $M$} to be $\Spec \gr^F M$ (this is independent of the choice of good filtration, see \cite[Appendix D]{HTT}).

\sloppy Returning to the vertex algebra world: suppose $V$ is a vertex algebra equipped with a set $\{A_i\}_{i \in I}$ of \textit{strong generators}. This means that $V$ is spanned by the vacuum vector $\vac$ and elements of the form $A_{i_1,(-n_1-1)} \cdots A_{i_r,(-n_r-1)} \vac$ with $r \ge 0$, $i_j \in I$, $n_j \ge 0$. The \textbf{Li filtration} on $V$ is a decreasing filtration $V = F^0 V \supseteq F^1 V \supseteq \cdots$, where
\begin{equation*}
	F^p V = \operatorname{span}_\BC \big\{ A_{i_1,(-n_1-1)} \cdots A_{i_r,(-n_r-1)} \vac \mid n_j \ge 0, \sum_j n_j \ge p \big\}.
\end{equation*}

\begin{proposition}[{\cite{Li:VA, Li:abVA}}]~
	\begin{enumerate}
		\item $F^\bullet V$ does not depend on the choice of strong generators.
		
		\item $\gr_F V$ is a commutative Poisson vertex algebra.
		
		\item $R_V := \gr^0_F V = V/ F^1 V$ is a commutative Poisson associative algebra with multiplication $a \cdot b = a_{(-1)} b$ and Poisson bracket $\{a,b\} = a_{(0)}b$.
		
		\item $R_V = \gr^0_F V$ generates $\gr_F V$ as a differential algebra.
	\end{enumerate}
\end{proposition}

The ring $R_V$ was first studied by Zhu \cite{Zhu:modular} and is called \textit{Zhu's $C_2$-algebra}.

\begin{definition}
	The \textbf{associated scheme} (resp. \textbf{associated variety}) of $V$ is defined to be 
	\begin{equation*}
		\tilde X_V := \Spec R_V,\quad
		(\text{resp. } X_V := (\Spec R_V)_{red}).
	\end{equation*}
	Here the subscript $(-)_{red}$ denotes taking the reduced scheme. In particular, $X_V$ is a Poisson variety. The vertex algebra $V$ is said to be \textbf{quasi-lisse} if $X_V$ has finitely many symplectic leaves \cite{Arakawa-Kawasetsu}.
\end{definition}

In the case of $V = V^k(\fg)$, there is an isomorphism of Poisson algebras
\begin{equation*}
	\BC[\fg^*] \bijects R_{V^k},\quad
	a \cdots b \mapsto
	(at^{-1}) \cdots (b t^{-1}) \vac \quad \text{where } a,\ldots,b \in \fg = \fg^{**} \subset \BC[\fg^*].
\end{equation*}
Hence
\begin{equation*}
	X_{V^k} = \big( \Spec R_{V^k} \big)_{red} = \Spec \BC[\fg^*] = \fg^* \cong \fg.
\end{equation*}
Finally, the surjection $V^k \surj L_k$ induces a surjection $R_{V^k} \surj R_{L_k}$, and hence an inclusion of Poisson varieties
\begin{equation*}
	X_{L_k} \injects \fg^* \cong \fg.
\end{equation*}
Since $V^k$ is naturally $\BZ_{\ge 0}$-graded and the quotient $V^k \surj L_k$ preserves gradings, $X_{L_k}$ is conical with respect to the $\BC^\times$ dilation on $\fg$. Hence $X_{L_k}$ is contained in the nilpotent cone $\cN \subset \fg$ if and only if $X_{L_k}$ has finitely many symplectic leaves, i.e. if and only if $L_k$ is quasi-lisse.

\subsection{Conjecture on associated varieties}\label{subsec:conj-AV}

To state the conjecture, we need one more notation. Given a parabolic $\fp = \fl \oplus \fu \subset \fg$ and a nilpotent orbit $\BO_L \subset \fl$, the \textbf{sheet attached to the pair $(\fl,\BO_L)$}\footnotemark, denoted by $\overline{\cS(\fl, \BO_L)}$, is the closure of the image of the map
\begin{equation*}
	G \times_P \big( \BO_L \times \fz(\fl) \times \fu \big) \subset G \times_P \fp \aro \fg.
\end{equation*}
Here $\fz(\fl)$ is the center of $(\fl)$, and the second map is the moment map, sending $(g, \xi) \in G \times_P \fp$ to $\Ad(g) \xi \in \fg$.

\footnotetext{Our (closed) sheet $\overline{\cS(\fl, \BO_L)}$ is equal to the closure of a Jordan class (or a ``decomposition class'' in \cite{Borho-Kraft:deform}). If $\BO_L$ is rigid in $L$, then $\overline{\cS(\fl, \BO_L)}$ is the closure of sheet in Borho-Kraft's sense.}

We write $\check \cl_n: \ubar{\check \cN} \to \BZ$ and $\check \BO(m)$ for the objects defined on the dual side.

\begin{conjecture}\label{conj:AV}
	Suppose $\fg$ is of simply-laced type. Let $m \in \BZ_{\ge 1}$, write $\check \BO(m) = \Sat_{\check L}^{\check G} \BO_{\check L}$ where $\BO_{\check L}$ is distinguished in $\check \fl$. Let $k = m - \check \Bh$. Then
	\begin{equation*}
		X_{L_k(\fg)} = \overline{\cS(\fl, \bd \BO_{\check L})}.
	\end{equation*}
	In particular, if $\check \BO(m)$ is distinguished (so that $\check L = \check G$ and $\BO_{\check L} = \check \BO(m)$), then
	\begin{equation*}
		X_{L_k(\fg)} = \overline{\bd \check \BO(m)}
	\end{equation*}
	and $L_k(\fg)$ is quasi-lisse.
\end{conjecture}

\begin{remark}
	In the associated variety of $L_k(\fg)$ at admissible rational levels, the orbits $\BO(m)$ also appear, but on the $\fg$ side rather than the $\check \fg$ side \cite{Arakawa:C2}. The two kinds of appearances of $\BO(m)$ will be unified in the future paper \cite{SYZ:rat}.
\end{remark}

\begin{remark}[Affine W-algebras]\label{rmk:W-algs}
	Let $f$ be a nilpotent element of $\fg$. Let $H_f^0(L_k(\fg))$ be the vertex algebra associated with $f$ by applying the generalized quantized Drinfeld-Sokolov reduction to $L_k(\fg)$ \cite{Feigin-Frenkel, Kac-Roan-Wakimoto}. When nonzero, it admits a quotient onto the simple affine W-algebra $W_k(\fg,f)$, and it is conjectured to be isomorphic to $W_k(\fg,f)$. Then associated variety of $H_f^0(L_k(\fg))$ is \cite{Arakawa:C2}
	\begin{equation}
		X_{H_f^0(L_k(\fg))} = X_{L_k(\fg)} \cap S_f,
	\end{equation}
	where $S_f$ is the Slodowy slice at $f$. Hence Conjecture \ref{conj:AV} also provides a description of the associated variety of $H_f^0(L_k(\fg))$ with $k > - \check \Bh$, it is then not difficult to determine which $H_f^0(L_k(\fg))$ is quasi-lisse. In particular, if $X_{L_k(\fg)}=\overline{\BO_f}$ where $\BO_f$ is the $G$-orbit of $f$, then $H_f^0(L_k(\fg))$ is \emph{lisse} \cite{Arakawa:C2}, and hence so is $W_k(\fg,f)$. 
\end{remark}

\begin{example}
	In type $D_4$, the associated varieties predicted by the conjecture are summarized in the following table:
	\begin{table}[H]
		\caption{Associated varieties of $L_k(D_4)$ predicted by Conjecture \ref{conj:AV}}\label{tbl:D4-prediction}
		\begin{tabular}{ccccccc}
			\hline $m$ & $k$ & $\check \BO(m)$ & $\check L$ & $\BO_{\check L}$ & conjectural $X_{L_k}$ & confirmed by
			\\ \hline
			$6$ & $0$ & $(7,1)$ & $D_4$ & $(7,1)$ & $\{0\}$ & \cite[Theorem 5.7.1]{Arakawa:C2}\\
			$5$ & $-1$ & $(5,3)$ & $D_4$ & $(5,3)$ & $\overline{\BO_{min}}$ & \cite[Theorem 1.1]{Arakawa-Moreau:Omin}\\
			$4$ & $-2$ & $(5,3)$ & $D_4$ & $(5,3)$ & $\overline{\BO_{min}}$ & \cite[Theorem 1.1]{Arakawa-Moreau:Omin}\\
			$3$ & $-3$ & $(3,3,1,1)$ & $A_2$ & $(3)$ & $\overline{\cS(A_2, \{0\})}$\\
			$2$ & $-4$ & $(3,2,2,1)$ & $3A_1$ & $(2)\times (2) \times (2)$ & $\overline{\cS(3A_1, \{0\})}$\\
			$1$ & $-5$ & $(1^8)$ & $T$ & $\{0\}$ & $\fg$ & \cite[Theorem 0.2.1]{Gorelik-Kac}
		\end{tabular}
	\end{table}
\end{example}

We now list all the known cases to the conjecture, followed by explanations in each case.

\begin{table}[H]
	\caption{Known cases of Conjecture \ref{conj:AV}}\label{tbl:known-AV} 
	\begin{tabular}{ccccc}
		\hline $\fg$ & $k$ & $m$ & $X_{L_k(\fg)}$ & reference \\ \hline
		$\fg$ & $-h^\vee+1$ & $1$ & $\fg$ & \cite{Gorelik-Kac} \\ 
		$A_n$ & $-1$ & $n$ & $\overline{\cS(A_{n-1},\{0\})}$ &  \cite{Jiang-Song, Arakawa-Moreau:sheets, AFK}\\ 
		$A_{2n-1}$ & $-n$ & $n$ & $\overline{\cS(2A_{n-1},\{0\})}$ & \cite{AFK} \\ 
		$D_n$ & $-1,-2$ & $2n-3,2n-4$ & $\overline{\BO_{min}}$ & \cite{Arakawa-Moreau:Omin}  \\ 
		$D_n$, $n$ odd & $-n+2$ & $n$ & $\overline{\cS(A_{n-1},\{0\})}$ & \cite{Arakawa-Moreau:sheets}, \cite{AFK} \\ 
		$D_n$, $n$ even & $-n+2$ & $n$ & $\overline{\BO_{(2^{n-2},1^4)}}$ & \cite{Arakawa-Moreau:irred, AFK} \\ 
		$E_6$ & $-1,-2,-3$ & $12,10,9$ & $\overline{\BO_{min}}$ & \cite{Arakawa-Moreau:Omin}  \\ 
		$E_6$ & $-4$ & $8$ & $\overline{\cS(D_5,\{0\})}$ & \cite{AFK} \\ 
		$E_7$ & $-1,\ldots,-4$ & $17,\ldots,14$ & $\overline{\BO_{min}}$ & \cite{Arakawa-Moreau:Omin}  \\ 
		$E_7$ & $-6$ & $12$ & $\overline{\BO_{2A_1}}$ & \cite{AFK}\\ 
		$E_8$ & $-1,\ldots,-6$ & $29,\ldots,24$ & $\overline{\BO_{min}}$ &\cite{Arakawa-Moreau:Omin} 
	\end{tabular}	
\end{table}

\begin{clause}[Regular case]
	Let $m \ge \check \Bh$. Then $\check \BO(m) = \check \BO(\check \Bh) = \check \BO_{reg}$ is the regular orbit, and the level $k = m - \check \Bh \ge 0$ is \textit{integrable}, meaning that $L(k \Lambda_0)$ is an integrable representation of $\fg_{aff}$ in the sense of Kac-Wakimoto \cite{Kac-Wakimoto:modular-reps, Kac-Wakimoto:rat-W}. According to \cite[Theorem 5.7.1]{Arakawa:C2}, we have
	\begin{equation*}
		X_{L_k(\fg)} = \{0\} = \overline{\bd \check \BO_{reg}}
	\end{equation*}
	which verifies Conjecture \ref{conj:AV}. 
\end{clause}

\begin{clause}[Subregular case]\label{cls:sreg-AV}
	Let $m$ be so that $\check \cl_n(\check \BO_{sreg}) \le m < \check \Bh$ where $\check \BO_{sreg}$ denotes the subregular orbit. Then $\check \BO(m) = \check \BO_{sreg}$. Indeed, if this is not the case, then by Theorem \ref{thm:O(m)} we would have $\check \BO(m) \succ \check \BO_{sreg}$. But the only orbit strictly larger than the subregular orbit is the regular orbit $\check \BO_{reg}$, and in all types $\check\cl_n(\check \BO_{reg})$ is strictly greater than any other $\cl_n(\check \BO)$ (this can be checked in each type, or one can use the argument at the end of proof of Theorem \ref{thm:cl-triangle} below), so $\check\cl_n(\check \BO_{reg}) > \check\cl_n(\check \BO_{sreg})$, a contradiction. 
	
	\textbf{Type $A$.}
	If $\fg$ has type $A_n$, then $m = n$, $k = -1$, and the subregular orbit is not distinguished. In fact, $\check \BO_{sreg} = \Sat_{\check L}^{\check G} \check \BO_{\check L, reg}$ where $\check L$ is any Levi subgroup of type $A_{n-1}$ and $\BO_{\check L, reg}$ is the regular orbit in $\check L$. By \cite[Theorem 1.2]{Jiang-Song} and \cite[Theorem 1.1]{Arakawa-Moreau:sheets}, or by \cite[Theorem 3.24, Table 6]{AFK}, we have
	\begin{equation*}
		X_{L_k(\fg)} = \overline{\cS(\fl, \{0\}_L)} = \overline{\cS(\fl, \bd \BO_{\check L, reg})}
	\end{equation*}
	which verifies Conjecture \ref{conj:AV}.
	
	\textbf{Type $D$ and $E$.}
	Suppose $\fg$ has type $D$ or $E$. Let $m_0 = \check\cl_n(\check \BO_{sreg})$. Then we are looking at the following cases
	\begin{center}
		\begin{tabular}{c|cc}
			$\fg$ & $m_0$  \\ \hline
			$D_n$ & $2n-4$ \\
			$E_6$ & $9$ \\
			$E_7$ & $14$\\
			$E_8$ & $24$ 
		\end{tabular}
	\end{center}
	In these types the subregular orbit is distinguished. Arakawa-Moreau \cite[Theorem 1.1]{Arakawa-Moreau:Omin} showed that
	\begin{equation*}
		X_{L_k(\fg)} = \overline{\BO_{min}} = \overline{\bd \check \BO_{sreg}} = \overline{\bd \check \BO(m)}
	\end{equation*}
	for any $k$ so that $m_0 \le m := k+\check \Bh \le \check \Bh-1$. Here $\BO_{min}$ denotes the minimal orbit of $\fg$. This verifies Conjecture \ref{conj:AV} for those levels.
\end{clause}

\begin{clause}[The case $m=1$]
	Let $m = 1$, $k = 1 - \check \Bh$. Then $\check \BO(1) = \{0\}$ is the zero orbit. Indeed, from the definition of $\cl_n$ in \ref{def:cln}, we have $\cl_n(\check \BO) = 1$ if and only if the highest $h$-weight in the Bala-Carter Levi $\check \fl$ of $\check \BO$ is $0$. This happens if and only if $\check \fl$ is a Cartan subalgebra in $\check \fg$, in which case $\check \BO_L = \{0\}$ and $\check \BO = \{0\}$.
	
	By \cite[Theorem 0.2.1]{Gorelik-Kac}, $V^k(\fg)$ is irreducible. Therefore $V^k(\fg) = L_k(\fg)$, and 
	\begin{equation*}
		X_{L_k(\fg)} = X_{V^k(\fg)} = \fg.
	\end{equation*}
	This verifies Conjecture \ref{conj:AV}. In fact, we have $\check \BO(1) = \{0\} = \Sat_{\check H}^{\check G} \{0\}_{\check H}$ where $\check H$ is any Cartan subgroup, and $\{0\}_{\check H}$ is distinguished in $\check \fh$. The sheet in Conjecture \ref{conj:AV} is 
	\begin{equation*}
		\overline{\cS(\fl, \bd \{0\}_{\check H})}
		= \overline{\cS(\fh, \{0\}_H)}
		= \overline{\im\big( G \times_B \fb \to \fg \big)} = \fg
	\end{equation*}
	which equals $X_{L_k}$.
\end{clause}

\begin{clause}[Evidence from Arakawa-Futorny-Kri\v zka]
	In \cite[Theorem 3.24, Table 6]{AFK} the associated varieties of several families of simple affine vertex algebras were computed. The simply-laced cases in \textit{loc. cit.} not covered by previous discussions are reproduced in the following table.
	
	\begin{table}[H]
		\caption{\cite{AFK} versus Conjecture \ref{conj:AV}}\label{tbl:AFK}
		{\small
			\begin{tabular}{ccccccc}
				\hline
				$\fg$ &  $k$ & $m$ & $X_{L_k(\fg)}$ & $\check \BO(m)$ & $\check \fl$ & $\BO_{\check L}$ \\ \hline
				$A_{2n-1}$  
				& $-n$ & $n$ & $\overline{\cS(\fl, \{0\}_L)}$ 
				& $\BO_{(n,n)}$ 
				& $2A_{n-1}$
				& $\BO_{(n)} \times \BO_{(n)}$\\ [5ex]
				$D_n$, $n$ odd 
				& $-n+2$ & $n$ & $\overline{\cS(\fl, \{0\}_L)}$ 
				& $\BO_{(n,n)}$
				& $A_{n-1}$ 
				& $\BO_{(n)}$\\ [5ex]
				$D_n$, $n$ even 
				& $-n+2$ & $n$ & $\overline{\BO_{(2^{n-2},1^4)}}$ 
				& $\BO_{(n+1,n-1)}$
				& $D_n$ 
				& $\BO_{(n+1,n-1)}$\\ [5ex]
				$E_6$
				& $-4$ & $8$ & $\overline{\cS( \fl, \{0\}_L)}$ 
				& $\BO_{D_5}$ 
				& $D_5$ & $\BO_{(9,1)}$\\ [5ex]
				$E_7$
				& $-6$ & $12$ & $\overline{\BO_{2 A_1}}$ 
				& $\BO_{E_7(a2)}$ 
				& $E_7$	
				& $\BO_{E_7(a2)}$
			\end{tabular}
		}
	\end{table}
	
	The first four columns contain relevant information taken from \cite[Table 6]{AFK}; they are the type of $\fg$, the number $k$, the number $m = k + \check \Bh$, and the associated variety $X_{L_k(\fg)}$ (the subregular cases in \textit{loc. cit.} are not shown here since they are already covered in \ref{cls:sreg-AV}). The last three columns contain information on the orbit $\check \BO(m)$, the Bala-Carter Levi $\check \fl$ of $\check \BO(m)$, and the orbit $\BO_{\check L}$ in $\check \fl$ so that $\check \BO(m) = \Sat_{\check L}^{\check G} \BO_{\check L}$. It can be easily checked that the associated varieties $X_{L_k(\fg)}$ are indeed equal to $\overline{\cS(\fl, \bd \BO_{\check L})}$. This verifies Conjecture \ref{conj:AV} in these cases.
	
	We note that the $D_n$ cases in the table have been computed earlier in \cite[Theorem 1.2]{Arakawa-Moreau:sheets} and \cite[Theorem 6.1]{Arakawa-Moreau:irred}.
\end{clause}

\section{Relation to reduction types}\label{sec:RTmin}

In this section, we define another cyclotomic level map $\cl_W$. We show that the maps $\cl_W$ and $\cl_n$ are related by the minimal reduction types map $\RTmin$, namely the following triangle commutes
\begin{equation*}
	\begin{tikzcd}
		\ubar W \ar[dr, "\RTmin"', two heads] \ar[rr, "\cl_W"] 
		&& \BZ_{1 \le \bullet \le \Bh}\\
		& \ubar \cN \ar[ur, "\cl_n"', end anchor=south west]
	\end{tikzcd},
\end{equation*}
see Theorem \ref{thm:cl-triangle}. 
This property will allow us to determine an affine Springer fiber and relate its top homology to modules over $L_k(\fg)$ in a sequential paper \cite{SYZ:spr}.

We will first review root valuation strata in $\fg(F)$, the minimal reduction types map $\RTmin$, the Kazhdan-Lusztig map $\KL$, and related constructions in \S \ref{subsec:rs-elem}-\ref{subsec:reduction}. Then we proceed to the proof of Theorem \ref{thm:cl-triangle}.

\subsection{Regular semisimple elements in the loop Lie algebra}\label{subsec:rs-elem}

Recall the \textit{discriminant map}
\begin{equation*}
	\Delta: \fg(\bar F) \aro \bar F,\quad
	a \mapsto \mathrm{tr}\left(\Exterior^{\dim \fg - \rank \fg}(\ad (a): \fg(\bar F) \to \fg(\bar F))\right).
\end{equation*}
This is a $G(\bar F)$-invariant function whose restriction to the Cartan is
\begin{equation*}
	\Delta: \fh(\bar F) \aro \bar F,\quad
	h \mapsto \prod_{\alpha \in R} \langle \alpha, h\rangle.
\end{equation*}

\begin{definition}
	An element $\gamma \in \fg(F)$ is said to be \textbf{regular semisimple} if it is regular semisimple as an element in $\fg(\bar F)$. Equivalently, if $\Delta(\gamma)\neq 0$.
Such a $\gamma$ can always be conjugated to an element $\gamma_{diag} \in \fh(\bar F)$ by an element of $G(\bar F)$. 
	
	An element $\gamma$ is said to be \textbf{topologically nilpotent} if $\ad(\gamma)^r \to 0$ in $\End \fg(F)$ as $r \to \infty$ in the $t$-adic topology. Equivalently, $\gamma$ can be $G(F)$-conjugated to an element in $\cN + t \fg(\cO) \subset \fg(\cO)$.		
	The set of regular semisimple (resp. topologically nilpotent) elements in $\fg(F)$ will be denoted by $\fg(F)_{rs}$ (resp. $\fg(F)_{tn}$). We write $\fg(F)_{rstn}=\fg(F)_{rs}\cap \fg(F)_{tn}$.
\end{definition}

\smallskip

\begin{clause}[Cartan subalgebras in $\fg(F)$]\label{cls:tori-in-g(F)}
	For any regular semisimple $\gamma$ in $\fg(F)$, its centralizer $Z_{\fg(F)}(\gamma)$ is a Cartan subalgebra in $\fg(F)$. By \cite[\textsection 1 Lemma 2]{KL:Fl}, the set of conjugacy classes of Cartan subalgebras in $\fg(F)$ are parameterized by the set $\ubar W$ of conjugacy classes in $W$. If the centralizer $Z_{\fg(F)}(\gamma)$ of $\gamma$ is a Cartan corresponding to the class $[w] \in \ubar W$, we say $\gamma$ is \textbf{of type $[w]$}. We write $\fg(F)_{[w]}$ for the subset of regular semisimple elements of type $[w]$.
	
	For an element $w \in W$, the Cartan subalgebra corresponding to the class $[w]$ can be conjugated over $\bar F$ to 
	\begin{equation*}
		\fh_w(F) := 
		\big\{ \gamma_{diag} \in \fh\laurent{t^{1/n}} \mid w \sigma\inv(\gamma_{diag}) = \gamma_{diag} \big\} \subset \fh(\bar F)
	\end{equation*}
	\cite[\textsection 4]{GKM:rval}, where $n$ is the order of $w$, and $\sigma \in \Gal(\bar F/F)$ is the automorphism defined in \ref{cls:aff-alg-notations}. Moreover, for any $\gamma \in \fg(F)_{[w],tn}$ the entries of its diagonalization $\gamma_{diag} \in \fh(\bar F)$ over $\bar F$ involve only strictly positive powers of $t^{1/n}$. Therefore $\gamma_{diag}$ belongs to the $\cO$-form $\fh_w(\cO)$ of $\fh_w(F)$ for some $w \in [w]$, where
	\begin{equation*}
		\fh_w(\cO) := \fh_w(F) \cap \fh \series{t^{1/n}}.
	\end{equation*}
\end{clause}

\smallskip

\begin{clause}[Topological notions for subsets of arc spaces, {\cite[\textsection5]{GKM:rval}}]\label{cls:fp-top}
	For a scheme $X$ of finite type over $\BC$ and $n\ge 0$, let $L^+_n X$ be its truncated arc space representing the functor $A\mapsto X(A\series{t}/t^{n+1})$ for any commutative $\BC$-algerbra $A$. The arc space $L^+X$ is the limit $\varprojlim_n L^+_nX$ with natural projections $\pi_n: L^+X\to L^+_n X$.
	
	A subset $Z\subset L^+X(\BC)=X(\BC\series{t})$ is called fp constructible (where fp stands for ``finite presentation'') if there exists $n \ge 0$ and a Zariski constructible subset $Z_n \subset L^+_nX(\BC)=X(\cO/t^{n+1})$ such that $Z=\pi_n^{-1}(Z_n)$. Similarly, there is the notion of fp open, fp closed, fp locally closed subsets of $L^+X$. 
	
	If $Z \subset L^+X(\BC)$ is fp constructible, say $Z=\pi_n^{-1}(Z_n)$ for a constructible $Z_n \subset L^+_nX(\BC)$, then the closure $\overline{Z}$ of $Z$ inside $L^+X(\BC)$ is defined to be the fp closed subset $\pi_n^{-1}(\overline{Z_n})$, where $\overline{Z_n}$ is the Zariski closure of $Z_n$ in  $L^+_nX(\BC)$. This definition is independent of the choices of $n$ and $Z_n$.
\end{clause}

\smallskip

\begin{clause}[Root valuation strata, {\cite{GKM:rval}}]\label{cls:rt-val-str}
      The set $\fg(F)_{rstn}$ can be partitioned into strata which is in some sense analogous to the partition of the nilpotent cone $\cN$ into nilpotent orbits.	Given a regular semisimple topologically nilpotent element $\gamma \in \fg(F)$ we can define a pair of combinatorial invariants $(w,r)$ as follows. Choose a conjugate $\gamma_{diag}$ of $\gamma$ in the diagonal Cartan $\fh(\bar F)$. Then $w$ is the unique element in $W$ so that $w \sigma\inv(\gamma_{diag}) = \gamma_{diag}$ where $\sigma \in \Gal(\bar F/F)$ is the element defined in \ref{cls:aff-alg-notations} sending $t^{1/k}$ to $e^{2\pi i/k} t^{1/k}$. The second invariant is $r: R \to \BQ$ defined by $r(\alpha) = \val( \alpha(\gamma_{diag}))$, the \textit{$\alpha$-valuation} of $\gamma_{diag}$. Since $\gamma$ is topologically nilpotent, $r$ in fact lands in $\BQ_{>0}$. Since different choices of $\gamma_{diag}$ differ by an element of $W$, the pair $(w,r)$ is well-defined up to $W$-conjugacy. We say the map $r$ is the \textbf{root valuation map attached to $\gamma$}.
	
	Let us write $\fg(F)_{(w,r)}$ for the set of regular semisimple topologically nilpotent elements in $\fg(F)$ whose associated invariant is equal to the $W$-conjugacy class of $(w,r)$. These are called the \textbf{root valuation stratum} of $(w,r)$. Note that the conjugacy class $[w]$ is the type of $\gamma$ in the sense of Section \ref{cls:tori-in-g(F)}. In particular, an element in $\fg(F)_{(w,r)}$ is $G(F)$ conjugate to an element in $\fh_w(\cO)$. Let $\fh_w(\cO)_r= \fh_w(\cO) \cap  \Ad(G(\bar F)) \cdot \fg(F)_{(w,r)}$. Whenever it is nonempty, we say that $(w,r)$ is a GKM pair.

Let $\fh_w(\cO)_{tn}=\fh_w(\cO)\cap \Ad(G(\bar F)) \cdot\fg(F)_{tn}$ and $\fh_w(\cO)_{rstn}=\fh_w(\cO)\cap \Ad(G(\bar F)) \cdot\fg(F)_{rstn}$.

\begin{proposition}\label{prop:open-stratum}~
	\begin{enumerate}
		\item For a GKM pair $(w,r)$, the subset $\fh_w(\cO)_r\subset\fh_w(\cO)$ is fp constructible. Its fp closure is given by
		\[\overline{\fh_w(\cO)_r}=\{\gamma\in \fh_w(\cO)\mid \val(\alpha(\gamma))\geqslant r(\alpha),\quad\forall\ \alpha\in R\}.\]
		\item There is a unique open root valuation stratum $\fh_w(\cO)_{r_{min}}$ in $\fh_w(\cO)_{tn}$ such that for any GKM pair $(w,r)$, we have $r_{min}(\alpha)\leqslant r(\alpha)$ for all root $\alpha$. The set $\fh_w(\cO)_{r_{min}}$ is fp irreducible.
	\end{enumerate}
\end{proposition}
\begin{proof}
Part (1) is \cite[Proposition~6.0.1]{GKM:rval}.
For part (2), for each GKM pair $(w,r)$, the set $\fh_w(\cO)_r$ is fp irreducible by \cite[Theorem~8.2.2]{GKM:rval}. Since $\fh_w(\cO)_{tn}$ is also fp irreducible
by \cite[3.4.4]{Bouthier-Kazhdan-Varshavsky}, it has a unique open dense stratum, denoted $\fh_w(\cO)_{r_{min}}$. Then it follows from part (1) that the function $r_{min}$ has the desired property.
\end{proof}

In \cite[Definition~1.6]{Yun:RTmin-arXiv}, an element $\gamma\in \fh_w(\cO)_{rstn}$ is called \emph{shallow} if $\Delta(\gamma)$ attains the minimal valuation among all elements in $\fh_w(\cO)_{rstn}$.
By part (2) of the previous proposition, elements in $\fh_w(\cO)_{r_{min}}$ are precisely the shallow elements.

\end{clause}

\subsection{Reduction to nilpotent cone and the Kazhdan-Lusztig map}\label{subsec:reduction}

In \ref{cls:rt-val-str} we have associated to each element $\gamma \in \fg(F)_{rstn}$ an invariant $(w,r)$. The type $[w]$ of $\gamma$, is a courser invariant. There is an even courser invariant, namely the minimal reduction type of $\gamma$.

Recall that if $\gamma \in \fg(F)$ is regular semisimple topologically nilpotent, then any conjugate $\gamma'$ of $\gamma$ inside $\fg(\cO)$ lies in $\cN+ t \fg(\cO)$. Therefore the $G$-orbit of $\gamma' (\modulo\; t) \in \cN$ is a nilpotent orbit.

\begin{definition}
	For $\gamma \in \fg(F)_{rstn}$, its \textit{reduction type} is the set
	\begin{equation*}
		\RT(\gamma) = \big\{ \BO \in \cN \mid \thereis \gamma' \in \fg(\cO) \cap \Ad G(F) \cdot \gamma \text{ such that } \gamma' (\modulo\; t) \in \BO \big\}.
	\end{equation*}
	The \textit{minimal reduction type} $\RTmin(\gamma)$ of $\gamma$ is the set of minimal elements in $\RT(\gamma)$.
\end{definition}

\begin{theorem}[{\cite[Theorem 1.11]{Yun:RTmin-arXiv}}]\label{thm:RTmin}
	For each $[w] \in \ubar W$, the intersection
	\begin{equation*}
		\bigcap_{\gamma \in \fg(F)_{(w,r_{min})}} \RTmin(\gamma)
	\end{equation*}
	is a singleton $\{\BO\}$. Write $\RTmin([w]) = \BO$. This defines a map
	\begin{equation*}
		\RTmin: \ubar W \aro \ubar \cN.
	\end{equation*}
\end{theorem}

\begin{clause}[Kazhdan-Lusztig map, {\cite[\textsection 9]{KL:Fl}}]\label{cls:KL-map}
	To state more properties of $\RTmin$, we need to recall the \textit{Kazhdan-Lusztig map}
	\begin{equation*}
		\KL: \ubar \cN \aro \ubar W
	\end{equation*}
	which is defined as follows: for $\BO\in \ubar \cN$, take any $e \in \BO$. Then there is a fp open dense subset $U \subset e + t \fg(\cO)$ and a unique class $[w] \in \ubar W$ such that every element $\gamma \in U$ is regular semisimple topologically nilpotent of the same type $[w]$. Then $\KL(\BO) := [w]$.
	
	We will also need the following fact from \cite[Proof of Theorem 6.1]{Yun:RTmin-arXiv}
	
	\begin{proposition}\label{prop:yun6.1}
	There is an open dense subset of $\BO+ t \fg(\cO)$ consisting of shallow elements of type $[w]=\KL(\BO)$, in other words, elements in $\fg(F)_{(w,r_{min})}$.
	\end{proposition}
	
\end{clause}

	\medskip

\begin{clause}[Lusztig's maps $\Phi$ and $\Psi$, {\cite{Lusztig:Phi, Lusztig:Phi2}}]
	Lusztig constructed in \cite{Lusztig:Phi, Lusztig:Phi2} two maps
	\begin{align*}
		&\Phi: \ubar W \aro \ubar \cN,\\
		&\Psi: \ubar \cN \aro \ubar W.
	\end{align*}	
	For the definition of $\Phi$, take $[w] \in \ubar W$, let $w \in [w]$ be an element with minimal length. There is a unique unipotent conjugacy class $\cC \subset G$ that is minimal among all unipotent conjugacy classes intersecting the Bruhat cell $BwB$. Let $\BO \in \ubar \cN$ be the nilpotent orbit corresponding to $\cC$. Then $\Phi([w]) := \BO$. Recall a conjugacy class $[w] \in \ubar W$ is said to be \textbf{elliptic} if any element $w \in [w]$ is not contained in any proper Levi subgroup of $W$. The map $\Phi$ restricts to an injection on the set $\ubar W_{ell}$ of elliptic conjugacy classes in $W$, and the image of $\ubar W_{ell}$ contains all distinguished nilpotent orbits. 
	
	The map $\Psi$ is defined by sending $\BO \in \ubar \cN$ to the unique \textit{most elliptic} class in $\Phi\inv(\BO)$, i.e. compared to other classes $[w'] \in \Phi\inv(\BO)$, $\dim \fh^w < \dim \fh^{w'}$ for any $w \in [w]$ and any $w' \in [w']$.
\end{clause}

\begin{theorem}~
	\begin{enumerate}
		\item \cite{Lusztig:Phi, Lusztig:Phi2} $\Phi$ is surjective and $\Psi$ is a section of $\Phi$.
		
		\item \cite{Yun:RTmin-arXiv} The pair of maps $(\RTmin, \KL)$ is equal to $(\Phi,\Psi)$.
	\end{enumerate}
\end{theorem}

\subsection{Cyclotomic levels and reduction types}\label{subsec:cl-RTmin}

For each integer $m$, we denote by $\zeta_m$ a primitive $m$-th roots of unity in $\BC$, and let $\Phi_m\in\BQ[x]$ be the cyclotomic polynomial which is the minimal polynomial of $\zeta_m$ over $\BQ$.

\begin{definition}\label{def:clW}
	Define the \textbf{cyclotomic level} map by
	\begin{equation*}
		\cl_W : W \aro \BZ_{\ge 1},\quad
		w \mapsto \max\{ \text{orders of eigenvalues of } w \acts \fh\}
	\end{equation*}
	where $\fh$ is the reflection representation. Alternatively, for any $w \in W$, let $\charp(w)$ be the characteristic polynomial of $w \acts\fh$. Since the reflection representation is defined over $\BQ$, $\charp(w)$ is a product of cyclotomic polynomials. Let $m$ be the largest integer such that $\Phi_m$ divides $\charp(w)$. Then we set 
	\begin{equation*}
		\cl_W(w) := m.
	\end{equation*}
	It clearly descends to a map on conjugacy classes
	\begin{equation*}
		\cl_W: \ubar W \aro \BZ_{\ge 1}.
	\end{equation*}
\end{definition}

First we record two immediate facts about $\cl_W$.

\begin{lemma}\label{lem:clW-in-Levi}~
	\begin{enumerate}
		\item If $w \in W$ is contained in a Levi subgroup $W_L$, then the values $\cl_W(w)$ and $\cl_{W_L}(w)$ computed in $W$ and in $W_L$ are the same.
		
		\item Suppose $L \subset G$ is a Levi, $L_1, \ldots,L_s$ are the simple factors of $L$, $w_i \in W_{L_i}$ and $w = w_1 \cdots w_s$. Then 
		\begin{equation*}
			\cl_W(w) = \max\{\cl_W(w_1),\ldots,\cl_W(w_s)\}.
		\end{equation*}
	\end{enumerate}	
\end{lemma}

\begin{proof}~
	\begin{enumerate}
		\item Write $V$, $V_L$ for the reflection representation of $W$ and $W_L$, respectively, and write $\charp_L(w)$ for the characteristic polynomial of $w$ on $V_L$. Then $V = V_L \oplus V_L'$ where $V_L'$ is the intersection of reflection hyperplanes of reflections in $W_L$. The element $w$ preserves this decomposition and acts as identity on $V_L'$. So $\charp(w)(x) = \charp_L(w)(x) \cdot (x-1)^{\dim V_L'} = \charp_L(w)(x) \cdot \Phi_1^{\dim V_L'}$, and $\Phi_d$ ($d >1$) divides $\charp(w)$ if and only if it divides $\charp_L(w)$. So (1) follows.
		
		\item By (1), $\cl_W(w)$ can be computed from $\charp_L(w)$. So (2) follows from the decomposition $\charp_L(w) = \charp_{L_1}(w) \cdots \charp_{L_s}(w) = \charp_{L_1}(w_1) \cdots \charp_{L_s}(w_s)$. \qedhere
	\end{enumerate}	
\end{proof}

For the next fact, recall that a positive integer $m$ is said to be a \textit{regular number} for $W$ if there exists an element $w \in W$ of order $m$ that admits a regular eigenvector in the reflection representation. Such elements have been extensively studied in \cite{Springer:reg}. In particular, there is a unique conjugacy class $W[m]$ in $W$ consisting of regular elements of order $m$. Regular numbers are classified in each type in \textit{loc. cit.}:
\begin{itemize}
	\item Type $A_n$: divisors of $n$ and $n+1$
	\item Type $B_n$: divisors of $2n$
	\item Type $C_n$: divisors of $2n$
	\item Type $D_n$: divisors of $2n-2$ and $n$
	\item Type $E_6$: $2, 3, 4, 6, 8, 9, 12$
	\item Type $E_7$: $2, 3, 6, 7, 9, 14, 18$
	\item Type $E_8$: $2, 3, 4, 5, 6, 8, 10, 12, 15, 20, 24, 30$
	\item Type $F_4$: $2, 3, 4, 6, 8, 12$
	\item Type $G_2$: $2, 3, 6$.
\end{itemize}

%

\begin{lemma}\label{lem:clW(reg)}
	Let $m$ be a regular number of $W$. Then $\cl_W(W[m]) = m$.
\end{lemma}

\begin{proof}
	Let $d_1 \le \cdots \le d_n$ be the degrees of $W$ (i.e. the degrees of homogenous algebraically independent generator of $\BC[\fh]^W$). Let $w \in W$ be a regular element with regular eigenvector $v$ with eigenvalue $\zeta_m$. Recall the following facts from \cite[Theorem 4.2]{Springer:reg}:
	\begin{enumerate}
		\item $m$ divides at least one of the degrees $d_i$ (in fact, the dimension of the $\zeta_m$-eigenspace of $w$ is equal to the number of $d_i$'s divisible by $m$); 
		
		\item The order of $w$ is equal to $m$.
		
		\item The eigenvalues of $w$ are precisely given by 
		\begin{equation*}
			\zeta_m{}^{d_1-1}, \ldots, \zeta_m{}^{d_n-1}.
		\end{equation*}
	\end{enumerate}
	
	Now property (2) implies that the minimal polynomial $\min(w)$ of $w$ divides $x^m -1$. The factors of $x^m-1$ are precisely the cyclotomic polynomials $\Phi_a$ for $a | m$. We claim that $\Phi_m$ appears in $\min(w)$. Indeed, since $m$ divides some $d_i$ by property (1), we have $\zeta_m{}^{d_i} = 1$. So $\zeta_m{}^{d_i-1}$ is a primitive $m$-th root of unity which is an eigenvalue of $w$. Since $\Phi_m$ is the minimal polynomial of $\zeta_m{}^{d_i-1}$ over $\BQ$, it must divide $\min(w)$. This proves the claim. 
	
	Finally, since all irreducible factors of $\charp(w)$ appear in $\min(w)$, they are all of the form $\Phi_a$, $a|m$, and clearly $\Phi_m$ has the largest subscript among them.
\end{proof}

\medskip

Next, we provide an alternative description of $\cl_W([w])$ in terms of root valuations. This fact is implicit in the proof of \cite[Lemma 2.2]{Yun:RTmin-arXiv}. We include a proof here for completeness. Recall from \ref{cls:rt-val-str} the notion of root valuations.

\begin{lemma}\label{lem:cl=root-val}
	Let $[w] \in \ubar W$, and let $\gamma\in \fh_w(\cO)_{r_{min}}$. Then
	\begin{enumerate}
		\item For any $\alpha$, there is a cyclotomic factor $\Phi_{d_\alpha}$ of $\charp(w)$ so that $\val(\alpha(\gamma)) =\frac1{d_\alpha}$.
		\item We have $\min \{ r_{min}(\alpha) \mid \alpha \in R\} = \frac1{\cl_W([w])}$.
	\end{enumerate}
\end{lemma}

In the case where $m$ is a regular number, it was shown in \cite[4.9]{GKM:rval} that the minimal valuation function $r_{min}$ attached to the conjugacy class $W[m]$ is the constant function $r_{min} = 1/m$. 

\begin{proof}
	Let $n$ be the order of $w$. Let $\zeta=\zeta_n$ be a fixed primitive $n$-th root of unity. For a polynomial $f(x)\in\BC[x]$, we write $\fh^{f(w)}$ for the kernel of $f(w): \fh \to \fh$. In particular, $\fh^{w - \zeta^j}$ is the $w$-eigenspace with eigenvalue $\zeta^j$, and $\fh^{\Phi_d(w)}$ is the sum of $w$-eigenspaces whose eigenvalues are primitive $d$-th roots of unity. 
	
	Recall that $\fh_w(\cO)$ is the set of fixed points in $\fh \series{t^{1/n}}$ under the automorphism $h(t^{1/n}) \mapsto wh(\zeta\inv t^{1/n})$. We have
	\begin{equation*}
		\fh_w(\cO) = \sum_{j \ge 0} \fh^{w-\zeta^j} t^{j/n} .
	\end{equation*}
	Indeed, if $h \in \fh^{w- \zeta^j}$ is nonzero, then $h t^{k/n}$ is fixed under the above automorphism if and only if $h t^{k/n} = (wh) (\zeta\inv t^{1/n})^k = \zeta^{j-k} h t^{k/n}$ which forces $j \in k + n\BZ$. Hence $\gamma$ can be written as
	\begin{equation*}
		\gamma = \sum_{j \ge 0} \gamma_j t^{j/n} \text{ with } \gamma_j \in \fh^{w-\zeta^j}.
	\end{equation*}
	For any root $\alpha$ of $\fg$, we have $\alpha(\gamma) = \sum_{j \ge 0} \alpha(\gamma_j) t^{j/n}$ where $\alpha(\gamma_j) \in \BC$. So $\val(\alpha(\gamma)) = j/n$ where $j$ is the smallest number so that $\alpha(\gamma_j) \neq 0$. 
	
	Let $j_\alpha$ be the smallest number so that $\alpha$ is nonzero on $\fh^{w- \zeta^{j_\alpha}}$. In particular $0\leqslant j_\alpha\leqslant n-1$, and $\val(\alpha(\gamma))\geqslant j_\alpha/n$. We claim that $\val(\alpha(\gamma)) = j_\alpha/n$. Indeed, $\gamma$ belongs to $\fh_w(\cO)_{r_{min}}$ implies that for each root $\alpha$, it has the smallest $\alpha$-valuation among all the regular elements in $\fh_w(\cO)_{tn}$ by Proposition \ref{prop:open-stratum}(2). On the other hand, there exists an element $\gamma' \in \fh_w(\cO)_{tn}$ so that the equality $\val(\alpha(\gamma')) = j_\alpha/n$ is achieved. To see this, we pick a regular element $h' \in \fh$ and write it as $h' = \sum_{j=0}^{n-1} h_j'$ according to the decomposition $\fh = \bigoplus_{j=0}^{n-1} \fh^{w- \zeta^j}$. Let $h_{j_\alpha} \in \fh^{w-\zeta^{j_\alpha}}$ be any element such that $\alpha(h_{j_\alpha}) \neq 0$, and set $\gamma'= h_{j_\alpha} t^{j_\alpha/n} + \sum_{j=0}^{n-1} h_j' t^{(j+n)/n} \in \fh_w(\cO)_{tn}$. Then $\gamma'$ is regular semisimple. Indeed, if $\beta$ is a root so that $\beta(\gamma') = 0$, then $\beta(h_{j_\alpha}) t^{j_\alpha/n} + \sum_{j=0}^{n-1} \beta(h_j') t^{(j+n)/n} = 0$. This forces $\beta(h_j') =0$ for all $j$, and hence $\beta(h') =0$, contradicting the regularity of $h'$. By construction, we have $\val(\alpha(\gamma')) = j_\alpha/n$. Hence by the minimality of the $\alpha$-valuation of $\gamma$, we have $\val(\alpha(\gamma)) \leqslant j_\alpha/n$. So the claim is proved.
	
	\smallskip
	
	We now show that for any root $\alpha$, the $\alpha$-root valuation of $\gamma$ is of the form $1/d_\alpha$, where $d_\alpha$ is an integer depending on $\alpha$. Indeed, let $d_\alpha$ be the order of $\zeta^{j_\alpha}$. Then $d_\alpha | n$. Since both $w$ and $\alpha$ are defined on a rational form $\fh_\BQ$ of $\fh$, and $\zeta^{j_\alpha}$ is an eigenvalue of $w$, we deduce that the cyclotomic polynomial $\Phi_{d_\alpha}$ divides $\charp(w)$, and $\alpha$ is nonzero on any $\fh^{w-\zeta'}$ for any $\zeta'$ a root of $\Phi_{d_\alpha}$. Such $\zeta'$ are of the form $\zeta^{j'}$ with the smallest $j'$ being $n/d_{\alpha}$. Thus, $n/d_\alpha=j_\alpha$. We deduce $j_\alpha/n=1/d_\alpha$. This proves (1).

\smallskip

	Part (2) easily follows from (1) and the definition of $\cl_W([w])$.
\end{proof}

The remainder of this subsection is devoted to proving the following theorem.

\begin{theorem}\label{thm:cl-triangle}
	The following diagram commutes
	\begin{equation*}
		\begin{tikzcd}
			\ubar W \ar[dr, "\RTmin"', two heads] \ar[rr, "\cl_W"] 
			&& \BZ_{1 \le \bullet \le \Bh}\\
			& \ubar \cN \ar[ur, "\cl_n"', end anchor=south west]
		\end{tikzcd}.
	\end{equation*}
\end{theorem}

In particular, $\cl_W$ is constant on fibers of $\RTmin$.

\smallskip

In the case where $\BO \in \ubar \cN$ is distinguished, there is a unique class $[w]$ in $\RTmin\inv(\BO)$, namely $[w] = \KL(\BO)$, and $[w]$ is elliptic \cite[Proposition 9.2]{KL:Fl}. In this case the theorem says $\cl_W([w]) = \cl_n(\BO)$. Interpreting $1/\cl_W([w])$ as a root valuation as in Lemma \ref{lem:cl=root-val}, this particular equality was proven in \cite[Proposition 9.11]{KL:Fl}. Our proof in the general case is a modification of their argument.

We will need the following lemma.

\begin{lemma}\label{lem:lattice}
Let $\fl \subset \fg$ be a Levi, let $\{h,e,f\} \subset \fl$ be an $\fsl_2$-triple where $h \in \fh$ and $e$ is distinguished in $\fl$. Let $2a$ be the highest $h$-weight on $\fl$. 
\begin{enumerate}
\item For $\gamma \in e + t \fg(\cO)$, any nonzero eigenvalue of $\ad(\gamma)$ acting on $\fg(\bar F)$ has valuation $\ge \frac 2{2a+3}$.
\item  There is an element $\gamma^0 \in e + t \fg(\cO)$ such that the operator $\ad \gamma^0 \acts \fg(\cO)$ has a nonzero eigenvalue, and all its nonzero eigenvalues have valuation equal to $\frac1{a+1}$.
\end{enumerate}
\end{lemma}

\begin{proof}
To prove (1), we show that for any $\gamma \in e + t \fg(\cO)$, there is an $\ad(\gamma)$-stable $\cO$-lattice $\cL\subset \fg(F)$ such that $\ad(\gamma)^{2a+3}(\cL)\subset t^2\cL$. Then 
under any $\cO$-basis of $\cL$, the matrix of $(\ad \gamma)^{2a+3} \acts \cL$, viewed as an element in $\Mat_{\dim \fg}(\cO)$, lies in the ideal $t^2 \Mat_{\dim \fg}(\cO)$. Since this matrix is also the matrix of $(\ad \gamma)^{2a+3} \acts \fg(\bar F)$, we see that any nonzero eigenvalue of $(\ad \gamma)^{2a+3} \acts \fg(\bar F)$ has valuation $\ge 2$. So any nonzero eigenvalue of $\ad \gamma \acts \fg(\bar F)$ has valuation $\ge \frac 2{2a+3}$.

\smallskip

To construct the lattice $\cL$, consider the element
	\begin{equation*}
		\xi = h + 2(a+1)d \in \fh \oplus \BC d 
	\end{equation*}
	where $\ad d$ acts on $\fg(F)$ by loop rotation (i.e. $\ad d = t \partial_t$). Then we have a decomposition
	\begin{equation*}
		\fg(F) = \mathop{\widehat{\bigoplus}}\limits_{i \in \BZ} \fg(F)_i
	\end{equation*}
	according to $\xi$-weights. More explicitly, writing $\fg_i$ for the $h$-weight space in $\fg$ for the weight $i$, we have
	\begin{align*}
		\fg(F)_0 &= \fg_0,\\
		\fg(F)_1 &= \fg_1 \oplus t \fg_{-2a-1},\\
		\fg(F)_2 &= \fg_2 \oplus t \fg_{-2a},\\
		&\vdots\\
		\fg(F)_{2a+1} &= \fg_{2a+1} \oplus t \fg_{-1},
	\end{align*}
	and 
	\begin{equation*}
		\fg(F)_{i+2(a+1)k} = t^k \fg(F)_i.
	\end{equation*}
	We take 
	\begin{equation*}
		\fg' 
		= \bigoplus_{0 \le i \le 2a+1} \fg(F)_i
		=
		\begin{array}{l}
			\fg_0 \oplus \fg_1 \oplus \cdots \oplus \fg_{2a} \oplus \fg_{2a+1} \\
			\oplus t \fg_{-2a-1} \oplus t \fg_{-2a} \oplus \cdots \oplus t \fg_{-2} \oplus t \fg_{-1},
		\end{array}
	\end{equation*}
	\begin{equation*}
		\cL
		= \cO \otimes_{\BC} \fg' = \fg(F)_{\ge 0} = \mathop{\widehat{\bigoplus}}\limits_{i \ge 0} \fg(F)_i.
	\end{equation*}
	It is clear that $\cL$ is a lattice in $\fg(F)$. It satisfies $t^3 \fg(\cO) \subset t^2 \cL$.
	
	To show that $\cL$ has the desired property $(\ad \gamma)^{2a+3} \cL \subseteq t^2 \cL$, write $\gamma = e + t \sum_i X_i + t^2 \sum_i Y_i + \cdots$ with $X_i, Y_i \in \fg_i$. Then each term in the sum is $\xi$-homogeneous with weight
	\begin{align*}
		\deg_\xi(e) &= 2,\\
		\deg_\xi(t X_{-2a-1}) &= -2a-1+ 2(a+1)  = 1,\\
		\deg_\xi(t X_i) &= i + 2(a+1) \ge -2a + 2(a+1) = 2 \quad (\text{for } i \ge -2a),\\
		\deg_\xi(t^2 Y_i) &= i + 4(a+1) \ge -2a + 4(a+1) = 2a+4,
	\end{align*}
	and so on. In particular, if we set $\gamma' := e + t \sum_{i \ge -2a} X_i + t^2 \sum_i Y_i + \cdots$, then $\gamma = tX_{-2a-1} + \gamma'$ and $\gamma' \in \fg(F)_{\ge 2}$. Expand $(\ad \gamma)^{2a+3} = (\ad tX_{-2a-1} + \ad \gamma')^{2a+3}$ into a sum of products of $\ad t X_{-2a-1}$ and $\ad \gamma'$, and let $P$ be a summand. If $\ad t X_{-2a-1}$ appears at most twice in $P$, then $P \in \fg(F)_{\ge 4(a+1)}$ (here $4(a+1) = 2(2a+1) + 2$, where $2(2a+1)$ comes from the $2a+1$ occurrences of $\ad \gamma'$ and the $2$ at the end comes from the at most two occurrences of $\ad t X_{-2a-1}$). Hence
	\begin{align*}
		P \fg' &= \sum_{0 \le i \le 2a+1} P \fg(F)_i \\
		&\subset \sum_{0 \le i \le 2a+1} \fg(F)_{\ge i+4(a+1)}\\
		&= \sum_{0 \le i \le 2a+1} t^2 \fg(F)_{\ge i}\\
		&= t^2 \cL.
	\end{align*}
	If $\ad t X_{-2a-1}$ appears at least three times in $P$, then $P \in t^3 \fg(\cO)$, and hence 
	\begin{equation*}
		P \fg' \subset P(\fg \oplus t\fg) \subseteq t^3 \fg(\cO) \subset t^2 \cL.
	\end{equation*}
	Therefore each summand $P$ of $(\ad \gamma)^{2a+3}$, and thus $(\ad \gamma)^{2a+3}$ itself sends $\fg'$ into $t^2 \cL$, as required.

	Next, we construct the desired element $\gamma^0$ in (2). Let $\gamma^0 = e + t X_{-2a}$ where $X_{-2a} \in \fl_{-2a}$ is any nonzero element. Then $\gamma^0|_{t=1}$ is not nilpotent by \cite[9.3(i)]{Springer:reg}. From the commutative diagram 
	\begin{equation*}
		\begin{tikzcd}
			\fg(F)_i \ar[r, "t \mapsto 1"] \ar[d, "\ad \gamma^0"']
			& \fg \ar[d, " \ad(\gamma^0|_{t=1})"]\\
			\fg[t] \ar[r, "t \mapsto 1"]
			& \fg
		\end{tikzcd}
	\end{equation*}
	($i \ge 0$), we see that $\ad \gamma^0$ cannot be nilpotent either, and it has a nonzero eigenvalue. Note that $\gamma^0 \in \fg(F)_2$, so
	\begin{align*}
		(\ad \gamma^0)^{a+1} \fg'
		&= \sum_{0 \le i \le 2a+1} (\ad \gamma^0)^{a+1} \fg(F)_i\\
		&\subseteq \sum_{0 \le i \le 2a+1} \fg(F)_{i+2(a+1)}\\
		&= \sum_{0 \le i \le 2a+1} t \fg(F)_i
		= t \fg'.
	\end{align*}
	It follows that as a map from $\fg'$ to $t \fg'$, $(\ad \gamma^0)^{a+1} = t \phi$ for some $\BC$-linear map $\phi: \fg' \to \fg'$. Hence eigenvalues of $(\ad \gamma^0)^{a+1} \acts \fg(F)$ are of the form $t \lambda$ where $\lambda$ are eigenvalues of $\phi$. In particular, all nonzero eigenvalues of $\ad \gamma^0$ have valuation $= \frac1{a+1}$, as desired.

\end{proof}

We are now ready to prove the next proposition, which is a key step to proving Theorem \ref{thm:cl-triangle}.

\begin{proposition}\label{prop:rt-val-of-lift}
	For any $\BO \in \ubar \cN$, we have
	\begin{equation*}
		\cl_W(\KL(\BO)) = \cl_n(\BO).
	\end{equation*}
\end{proposition}

\begin{proof}
Recall that $\cl_n(\BO)$ is defined by first choosing an $\fsl_2$-triple $\{h,e,f\} \subset \fl$ attached to $\BO$ in the Bala-Carter Levi $\fl$ of $\BO$, then setting $\cl_n(\BO) = a+1$ where $2a$ is the highest $h$-weight on $\fl$.
	
First, we show $\cl_W(\KL(\BO))\le a+1$.
Let $r_{min}: R \to \BZ$ be the minimal root valuation map attached to the class $\KL(\BO)$. Recall from Proposition \ref{prop:yun6.1} that  there is a dense fp open subset $V\subset \BO+t\fg(\cO)$ consisting of elements in $ \fg(\cO)_{(\KL(\BO),r_{min})}$. Each $\gamma$ in $V$ is regular semisimple, the smallest valuation of eigenvalues of $\ad\gamma$ is equal to $\frac1{\cl_W(\KL(\BO))}$ by Lemma \ref{lem:cl=root-val}. On the other hand, from 
Lemma \ref{lem:lattice}, for $\gamma\in \BO + t \fg(\cO)$, any nonzero eigenvalue of $\ad \gamma \acts \fg(\bar F)$ has valuation $\ge \frac 2{2a+3}$. So
\begin{equation*}
		\frac 1{\cl_W(\KL(\BO))} \ge \frac{2}{2a+3}.
	\end{equation*}
	Since $\frac1{a+1} = \frac2{2a+2}$ is the smallest number of the form $\frac1d$ that is $\ge \frac2{2a+3}$, we in fact have
	\begin{equation}\label{eqn:1/cl-1/a+1}
		\frac 1{\cl_W(\KL(\BO))} \ge \frac1{a+1} > \frac{2}{2a+3}
	\end{equation}
	and thus $\cl_W(\KL(\BO))\le a+1$.

\smallskip

To show the other inequality, for any integer $m$, consider the map
	\[\chi_m: \fg(\cO)\to \BA^1(\cO),\quad a \mapsto \mathrm{tr}\left(\Exterior^{m}(\ad (a): \fg(\cO) \to \fg(\cO))\right). \]
	For any $j>0$, the subsets $\BA^1(\cO)_{\val \ge j}$, resp. $\BA^1(\cO)_{\val > j}$, consisting of power series with valuation $\ge j$, resp $>j$, are fp closed subsets in $\BA^1(\cO)=\cO$.
Thus, their preimages $\fg(\cO)_{\chi_m \val \ge j}$, resp. $\fg(\cO)_{\chi_m \val > j}$,  under the morphism $\chi_m$ are  fp closed subsets of $\fg(\cO)$.

Now, let $k$ be the number of roots $\alpha \in R$ so that $r_{min}(\alpha)$ achieves the minimum value $\frac1{\cl_W(\KL(\BO))}$. Then any element $\gamma$ in $V$ has exactly $k$ eigenvalues with valuation $\frac1{\cl_W(\KL(\BO))}$, and all the others have valuation $> \frac1{\cl_W(\KL(\BO))}$. Thus, for any $m\le k$, the product of any $m$ eigenvalues of $\gamma$ has valuation $\geqslant \frac m{\cl_W(\KL(\BO))}$. Thus 
$$V\subset \fg(\cO)_{\chi_m \val \ge \frac m{\cl_W(\KL(\BO))}}.$$
Since $\fg(\cO)_{\chi_m \val \ge \frac m{\cl_W(\KL(\BO))}}$ is fp closed, and $\BO + t\fg(\cO)$ is the closure of $V$, we deduce
\begin{equation}\label{eq:m<=k}
\BO + t\fg(\cO)\subset \fg(\cO)_{\chi_m \val \ge \frac m{\cl_W(\KL(\BO))}},\quad \text{ if }m\le k.
\end{equation}
If $m>k$, then the product of any $m$ eigenvalues of $\gamma$ always has valuation $>\frac m{\cl_W(\KL(\BO))}$. A similar argument implies
$$\BO + t\fg(\cO)\subset \fg(\cO)_{\chi_m \val > \frac m{\cl_W(\KL(\BO))}},\quad \text{ if }m> k.$$

Consider the element $\gamma^0$ constructed in Lemma \ref{lem:lattice}	(2).
Let $k'$ be the number of nonzero eigenvalues of $\ad \gamma^0 \acts \fg(\bar F)$. It is nonzero.
Moreover, we have $k'\leqslant k$. Indeed, if $k' > k$, then	\begin{equation*}
		\BO + t \fg(\cO) \subset \fg(\cO)_{\chi_{k'}\val > \frac{k'}{\cl_W(\KL(\BO))}} \subset \fg(\cO)_{\chi_{k'}\val > \frac{k'}{a+1}},
	\end{equation*}
	where the second inclusion is deduced from \eqref{eqn:1/cl-1/a+1}.
	But $\gamma^0$ is an element that is both in $\BO + t \fg(\cO)$ and in $\fg(\cO)_{\chi_{k'}\val = \frac{k'}{a+1}}$, a contradiction. So $k'\leqslant k$. It follows from \eqref{eq:m<=k} that
	$$\BO + t\fg(\cO)\subset \fg(\cO)_{\chi_{k'} \val \ge \frac {k'}{\cl_W(\KL(\BO))}}.$$
In particular, $\gamma^0\in \fg(\cO)_{\chi_{k'} \val \ge \frac {k'}{\cl_W(\KL(\BO))}}$. On the other hand, $\gamma^0$ has exactly $k'$ nonzero eigenvalue whose valuation are all equal to $\frac{1}{a+1}$, thus $\val(\chi_{k'}(\gamma^0))=\frac{k'}{a+1}$. We deduce
$$\frac{k'}{a+1}\ge \frac {k'}{\cl_W(\KL(\BO))}.$$
	As $k'>0$, this implies $\cl_W(\KL(\BO))\ge a+1$. We are done.
\end{proof}

\begin{proof}[Proof of Theorem \ref{thm:cl-triangle}]
	For any Levi $M \subset G$, let $\RTmin_M: \ubar W_M \to \ubar \cN_M$, $\KL_M: \ubar \cN_M \to \ubar W_M$ be respectively the minimal reduction type and Kazhdan-Lusztig map for the group $M$ (so that $\RTmin_G= \RTmin$ and $\KL_G= \KL$). Recall from \cite[Corollary 3.4.(2)]{Yun:RTmin-arXiv} that the following diagram commutes
	\begin{equation*}
		\begin{tikzcd}
			\ubar W_M \ar[r] \ar[d, "\RTmin_M"']
			& \ubar W \ar[d, "\RTmin"]\\
			\ubar \cN_M \ar[r, "\Sat_M^G"]
			& \ubar \cN		
		\end{tikzcd}
	\end{equation*}	
	where the top horizontal map sends the conjugacy class of $w \in W_M$ to its conjugacy class in $W$.
	
	For $\BO\in\cN$, take any $[w] \in \RTmin\inv(\BO)$, and take any $w \in [w]$. Let $M$ be the centralizer of $\fh^w$ in $G$. It is a Levi subgroup in $G$ and $w$ is elliptic in the Weyl group $W_M$ of $M$. Write $[w]_M \in \ubar W_M$ for the corresponding conjugacy class, and write $\BO_M= \RTmin_M([w]_M)$. Then $\BO = \Sat_M^G \BO_M$ by the above diagram, and $\cl_W([w]) = \cl_W([w]_M)$ by Lemma \ref{lem:clW-in-Levi}. Let $L \subset M$ be a Bala-Carter Levi for $\BO_M$, and write $\BO_M = \Sat_L^M \BO_L$. Then $\BO= \Sat_L^G \BO_L$ and $L$ is a Bala-Carter Levi for $\BO$, and so by definition $\cl_n(\BO) = \cl_{n,M}(\BO_M) = \cl_{n,L}(\BO_L)$.
	
	We claim $\KL_M(\BO_M) = [w]_M$. Indeed, by Lusztig's definition of his map $\Psi$ \cite[Theorem 0.2]{Lusztig:Phi2} (which equals $\KL$ by \cite[1.14]{Yun:RTmin-arXiv}), $\KL_M(\BO_M)$ is the unique most elliptic class in $\RTmin_M\inv(\BO_M)$. Since $[w]_M \in \RTmin_M\inv(\BO_M)$ and $[w]_M$ is elliptic in $W_M$, the claim follows. By applying Proposition \ref{prop:rt-val-of-lift} to $\BO_M$, we see that 
	\begin{equation*}
		\cl_W([w]) = \cl_W([w]_M) = \cl_W(\KL_M(\BO)) \xeq{\mathrm{Prop} \ref{prop:rt-val-of-lift}} \cl_{n,M}(\BO_M) = \cl_n(\BO)=\cl_n(\RTmin([w])).
	\end{equation*}
	Thus the diagram in Theorem \ref{thm:cl-triangle} commutes.
	
	\smallskip
	
	Finally, we show that $\cl_n$ lands in $\BZ_{1 \le \bullet \le \Bh}$. For any orbit $\BO$, let $\{h_d,e_d,f_d\}$ be the corresponding $\fsl_2$-triple in $\fg$ with $h \in \fh$ dominant, and let $\fl$ be a Bala-Carter Levi of $\BO$ containing $\fh$. Then
	\begin{equation*}
		\cl_n(\BO) = 1+ \frac12 \max_{\alpha \in R(\fl,\fh)} \langle \alpha,h \rangle
		\le 1+ \frac12 \max_{\alpha \in R(\fg,\fh)} \langle \alpha,h \rangle
		= 1 + \frac12 \langle \theta, h_d \rangle
	\end{equation*}
	where $\langle \theta, h_d \rangle$ is a positive linear combination of the numbers $\langle \alpha_i, h_d \rangle$, where $\alpha_i$ are simple roots. As $\BO$ varies, the value $1 + \frac12 \langle \theta, h_d \rangle$ achieves maximum precisely when the $\langle \alpha_i, h_d \rangle$ are maximum. This happens precisely when $\BO$ is the regular orbit, where $\langle \alpha_i, h_d \rangle = 2$ for all $i$, in which case $1 + \frac12 \langle \theta, h_d \rangle = \Bh$.
\end{proof}

\begin{notation}\label{not:cl=cln}
	In the rest of the paper, we write $\cl$ for both $\cl_W$ and $\cl_n$ if there is no confusion.
\end{notation}

Finally, we record an immediate consequence from the calculations of $\BO(m)$ in \S \ref{subsec:cln}.

\begin{proposition}\label{prop:RTmin(reg-class)}
	Let $m$ be a regular number of $W$. Then $m \in \im \cl$, and $\BO(m) = \RTmin(W[m])$.
\end{proposition}

\begin{proof}
	In classical types, the values of $\RTmin(W[m])$ have been computed in \cite[Table 3]{Jakob-Yun} (see the cases where $\nu = d/m$ is equal to $1/m$; the notation $\lambda^{a,b}$ denotes the unique partition of $a$ with $b$ parts each part is equal to either $\lfloor \frac ab \rfloor$ or $\lceil \frac ab \rceil$). They agree with the description of $\BO(m)$ computed in \S \ref{subsec:cln}. 
	
	In exceptional cases, this can be checked using the values of $\RTmin$ described in \cite[\textsection 2]{Lusztig:Phi2}.
\end{proof}

\section{Relation to two-sided cells}\label{sec:cells}

In this section, we establish a connection between the cyclotomic level map and two-sided cells in the affine Weyl group, see Theorem \ref{thm:O(m)-n-cells}. Our result is closely related to two bijections --- one classical and the other affine --- that link two-sided cells to nilpotent orbits. To provide context, we review these bijections in \S \ref{subsec:KL-cells}.


\subsection{Kazhdan-Lusztig cells and Barbasch-Vogan-Lusztig-Spaltenstein's construction}\label{subsec:KL-cells}

\begin{clause}[Cells in $W$ and $W_{aff}$]\label{cls:KL-cells}
	Let $W'$ be a Coxeter group, which for us equals $W$ or $W_{aff}$. Let $\cH'$ denote the Hecke algebra of $W'$ over $\BZ[q^{1/2},q^{-1/2}]$, that is a free $\BZ[q^{1/2},q^{-1/2}]$-module with basis $T_w$, $w \in W'$ satisfying
	\begin{equation*}
		T_x T_y = T_{xy} \text{ if } \ell(x) + \ell(y) = \ell(xy),
	\end{equation*}
	\begin{equation*}
		(T_s + 1)(T_s - q) = 0 \text{ for any simple reflection } s.
	\end{equation*}
	There is an involution $h \mapsto \bar h$ of $\cH'$ defined by
	\begin{equation*}
		\overline{q^{1/2}} = q^{-1/2},\quad 
		\overline{T_w} = T_{w\inv}\inv.
	\end{equation*}
	Let $C_w$ denote the Kazhdan-Lusztig basis element of $\cH$ given by
	\begin{equation*}
		C_w = q^{-\ell(w)/2} \sum_{y \le w} (-q)^{\ell(w) - \ell(y)} \overline{P_{y,w}} T_y,
	\end{equation*}
	for some $P_{y,w} \in \BZ[q]$ with degree $\le \frac12 (\ell(w) - \ell(y) - 1)$ satisfying $\overline{C_w} = C_w$, as in \cite[Theorem 1.1]{KL:Hecke}. 
	
	Whenever $s \in W'$ is a simple reflection, $w \in W'$, and $sw > w$, the product $C_s C_w$ is a $\BZ$-linear combination of $C_y$'s for some $y \le sw$. If $C_y$ appears, we say $w \ge_L y$. We use $\le_L$ to generate a preorder on $W'$. If $w \ge_L y$ and $w \le_L y$, we write $w \sim_L y$. The equivalence classes of $\sim_L$ are called \textit{left cells}
	\begin{equation*}
		\bc^L(w) = \{ y \in W' \mid y \sim_L w\}.
	\end{equation*}
	The same definition can be made using $C_w C_s$ ($ws > w$) instead, and the resulting relations are denoted by $\le_R$, $\sim_R$, respectively, and the equivalence classes are called \textit{right cells}, denoted by $\bc^R(w)$. The preorder generated by using both $\le_L$ and $\le_R$ is denoted by $\le_{LR}$, and we have the corresponding equivalence relation $\sim_{LR}$ and the equivalence classes $\ubar \bc(w)$, called the \textit{two-sided cells}. 
\end{clause}

Next, we would like to review the Barbasch-Vogan-Lusztig-Spaltenstein duality map and a bijection between two-sided cells and special orbits. Both constructions go through associated varieties of primitive ideals which we recall now.

\begin{clause}[Primitive ideals]
	For the moment, let $\fg$ be a semisimple Lie algebra. An \textit{infinitesimal character} is a ring map $\chi: \cZ(\fg) \to \BC$ from the center of the enveloping algebra $\cU(\fg)$ to $\BC$. Via the Harish-Chandra isomorphism $\cZ(\fg) \bij \Sym(\fh)^W$, infinitesimal characters are parameterized by $W$-orbits in $\fh^*$. 
	
	A \textit{primitive ideal} $J$ in $\cU(\fg)$ is the annihilator of a simple $\fg$-module. Each primitive ideal gives rise to a closed subvariety in $\fg^*$, the \textit{associated variety of $J$}, by
	\begin{equation*}
		\AV(J) := \AV(\cU(\fg)/J) \subset \fg^*
	\end{equation*}
	where $\AV(\cU(\fg)/J)$ is the associated variety of $\cU(\fg)/J$ defined at the beginning of \S \ref{subsec:AV}. By means of an invariant bilinear form on $\fg$, we may identify $\fg^*$ with $\fg$ and $\AV(J)$ with a subvariety of $\fg$. It is known that $\AV(J)$ is contained in the nilpotent cone $\cN$. In fact, by the work of Borho-Brylinski \cite{Borho-Brylsinki:diff-op-1}, Kashiwara-Tanisaki \cite{KT:ch} and Joseph \cite{Joseph:prim}, $\AV(J)$ is irreducible, and is hence is the closure of a single nilpotent orbit $\BO \in \ubar \cN$.
	
	We say $J$ has infinitesimal character $\chi$ if it contains the ideal $\cU(\fg) \cdot \ker \chi$. For any fixed $\chi$, there is a unique maximal primitive ideal with infinitesimal character $\chi$. 
\end{clause}

\begin{clause}[Barbasch-Vogan-Lusztig-Spaltenstein duality]\label{cls:BV}
	In \cite{Barbasch-Vogan:unipotent} Barbasch-Vogan constructed a map
	\begin{equation*}
		\bd: \ubar{\check \cN} \aro \ubar \cN
	\end{equation*}
	as follows. Let $\check \BO \in \ubar{\check \cN}$ be any orbit. Take an $\fsl_2$-triple $\{h,e,f\}$ with $e \in \check \BO$ and $h \in \check \fh = \fh^*$ dominant. Then the $W$-orbit of $\lambda := \frac h2 \in \fh^*$ determines an infinitesimal character $\chi_\lambda: \cZ(\fg) \to \BC$ via the Harish-Chandra homomorphism. Let $J_\lambda$ be the unique maximal primitive ideal with this infinitesimal character, and let $\BO$ be the nilpotent orbit whose closure equals $\AV(J_\lambda)$. Then the map $\bd$ sends
	\begin{equation*}
		\bd(\check \BO) = \BO.
	\end{equation*}
	The image of $\bd$ equals $\ubar \cN_{sp}$, the set of special nilpotent orbits in the sense of Lusztig \cite[\S 9]{Lusztig:sp-1}, and $\bd$ restricts to an order reversing bijection between special orbits of $\check \fg$ and those of $\fg$.
	
	The duality map $\bd$ was discovered earlier by Lusztig \cite{Lusztig:sp-1} and Spaltenstein \cite{Spaltenstein:unip}. 
\end{clause}

\begin{clause}[Two-sided cells in $W$ and special orbits]
	By the accumulated work of many people (Joseph, Vogan, Kazhdan-Lusztig, Brylinski-Kashiwara, Beilinson-Bernstein), there is an order-reversing bijection 
	\begin{equation}
		\{\text{two-sided cells in }W\} \bijects \ubar \cN_{sp}
	\end{equation}
	sending $\ubar \bc(w)$ to the unique orbit $\BO$ whose closure equals the associated variety of the annihilator $\Ann_{\cU(\fg)} L(w \rho - \rho)$. Here the left hand side is equipped with the order $\le_{LR}$ and the right side is equipped with the closure order $\preccurlyeq$.
	
	Together with the Barbasch-Vogan-Lusztig-Spaltenstein duality, we obtain a triangle of bijections
	\begin{equation}\label{diag:cells-bij-sp}
		\begin{tikzcd}
			& \{\text{two-sided cells in } W\} \ar[dr, leftrightarrow] \\
			\ubar{\check \cN}_{sp} \ar[rr, leftrightarrow, "\bd"] \ar[ur, dashed, "(*)"] && \ubar \cN_{sp}
		\end{tikzcd}
	\end{equation}
	where the dashed arrow $(*)$ is the composition of the other two arrows. In the special case where $\check \BO \in \ubar{\check \cN}_{sp}$ is \textit{even}, i.e. the infinitesimal character $\chi_\lambda$ constructed from $\check \BO$ is integral, the arrow $(*)$ sends
	\begin{equation*}
		(*): \check \BO \mapsto \ubar \bc(w_\lambda)
	\end{equation*}
	where $w_\lambda$ is the unique longest element in $W$ stabilizing $\lambda = \frac h2$; see \cite{Barbasch-Vogan:unipotent}.
\end{clause}

\begin{clause}[Lusztig's bijection, {\cite[Theorem 4.8]{Lusztig:aff-cells-4}}]\label{cls:L-bij}
	Lusztig constructed a bijection 
	\begin{equation*}
		\{ \text{two-sided cells in } W_{aff} \} \bijects \ubar{\check \cN}
	\end{equation*}
	which can be thought of as an affine analog to the bijection $(*)$ in (\ref{diag:cells-bij-sp}). This bijection is order preserving by Bezrukavnikov \cite[Theorem 4]{Bez:order-preserving}. Despite the similarity to the bijection in the classical case, the construction of Lusztig's bijection is very different. Roughly, it goes as follows. First, Lusztig defined the \textit{asymptotic Hecke algebra} $J$ which admits an algebra decomposition $J = \bigoplus_{\ubar \bc} J_{\ubar \bc}$ over all two-sided cells $\ubar \bc$ in $W_{aff}$. Hence for each simple $J$-module $E$ there is a unique two-sided cell $\ubar \bc$ so that $J_{\ubar \bc}$ does not act as zero on $E$. Second, there is a bijection
	\begin{equation*}
		\{\text{simple $J$-modules}\} \bijects \{\text{triples } (u,s,\rho)\}/\check G
	\end{equation*}
	where $u \in \check G$ is a unipotent element, $s \in Z_{\check G}(u)$ is a semisimple element, and $\rho$ is an irreducible representation of $Z_{\check G}(su)/Z_{\check G}(su)^\circ$ that appears in $H^*(\cB_u^s)$ where $\cB_u^s$ is the $s$-fixed locus in the Springer fiber $\cB_u$ of $u$. Here $\cB_u$ is the group version of the Springer fiber, that is
	\begin{equation*}
		\cB_u = \big\{ g \check B \in \check G/ \check B \mid \Ad(g)\inv u \in \check B \big\}.
	\end{equation*}
	Finally, it was shown that remembering only the two-sided cell $\ubar \bc$ attached to a simple $J$-module corresponds to remembering only the $\check G$-conjugacy class of $u$ of the triple $(u,s,\rho)$. This results in the desired bijection.
	
\end{clause}

\subsection{Cyclotomic levels and two-sided cells}\label{subsec:cl-n-cells}

Let $\check \cl_n: \ubar{\check \cN} \to \BZ$ be the map $\cl_n$ defined on the dual side, and write $\check \BO(m)$ for the unique maximal orbit in $\check \cl_n\inv([1,m])$.

\begin{theorem}\label{thm:O(m)-n-cells}
	Let $\fg$ be of simply-laced type. For each integer $m \in [1, \Bh]$, define the following objects:
	\begin{itemize}
		\item $\xi_m$ is the unique dominant $W_{aff}$-translate of $m \Lambda_0 + \rho \in \fh_{aff}^*$ (see \ref{cls:aff-alg-notations} for notations), that is, the unique dominant element in the $W_{aff}$-orbit of $m \Lambda_0 + \rho$ under the usual (i.e. non-dot) action,
		\item $W_m := \Stab_{W_{aff}}(\xi_m)$, a proper standard parabolic subgroup of $W_{aff}$,
		\item $w_m \in W_m$ is the longest element, and
		\item $\ubar \bc(w_m)$ is the two-sided cell in $W_{aff}$ containing $w_m$.
	\end{itemize}
	Then under Lusztig's bijection, 
	\begin{equation*}
		\check \BO(m) \text{ corresponds to } \ubar \bc(w_m).	
	\end{equation*}
	In particular,
	\begin{itemize}
		\item if $m \in \im \check \cl_n$, then the following diagram commutes.
		\begin{equation*}
			\begin{tikzcd}
				m \ar[d, leftrightarrow, "\check \cl_n"'] \ar[r, leftrightarrow] & \ubar \bc(w_m) \ar[dl, leftrightarrow]\\
				\check \BO(m)
			\end{tikzcd}
		\end{equation*}
		
		\item Let $m_0 \in \im \check \cl_n$ be the largest number so that $m_0 \le m$. Then we have $\ubar \bc(w_m) = \ubar \bc(w_{m_0})$, that is, $w_m \sim_{LR} w_{m_0}$.
	\end{itemize}
\end{theorem}


\begin{clause}[Affine analog of the Barbasch-Vogan-Lusztig-Spaltenstein construction]\label{cls:aff-duality}
	Recall that Lusztig's bijection can be viewed as an affine analog to the map $(*)$ in (\ref{diag:cells-bij-sp}), though their constructions differ. However, Theorem \ref{thm:O(m)-n-cells} shows that if restricted to the subset of orbits of the form $\check \BO(m)$ for $m \in \im \check \cl_n$ in simply-laced types, Lusztig's bijection can be expressed as $\check \BO(m) \mapsto \xi_m \mapsto \ubar \bc(w_m)$, mirroring the restriction of the map $(*)$ to the set of even orbits, which sends $\check \BO \mapsto \lambda = \frac h2 \mapsto \ubar \bc(w_\lambda)$. 
\end{clause}

The proof of the theorem makes use of the following characterization of Lusztig's map taken from \cite[Lemma 4.6, Proposition 4.8]{Yun:Epi}. Recall that a simple reflection in $W_{aff}$ is said to be \textit{special} if the corresponding node in the extended Dynkin diagram of $\fg$ is conjugate to the $0$-th node under an automorphism of the diagram.

\begin{lemma}\label{lem:L-bij-as-jind}
	Let $W_m \subset W_{aff}$ be a proper parabolic subgroup, let $w_m \in W_m$ be the longest element, and let $\check \BO_m$ be the nilpotent orbit that corresponds to $\ubar \bc(w_m)$ under Lusztig's bijection. By means of the projection $W_{aff} \surj W$, we may identify $W_m$ as a subgroup of $W$. 
	\begin{enumerate}
		\item Under the Springer correspondence, $(\check \BO_m, \ubar \BC)$ corresponds to $j_{W_m}^W (\sgn)$, the truncated induction (or $j$-induction) of the sign representation.
		
		\item If $W_m$ is special, i.e. if $W_m$ does not contain a special simple reflection, then a conjugate of $W_m$ is identified with a standard parabolic subgroup of $W$ under the projection $W_{aff} \surj W$, which gives rise to a standard parabolic subgroup $\check P_m \subset \check G$. Then $\check \BO_m = \Ind_{\check P_m}^{\check G} \{0\}$ is the Richardson orbit associated to $\check P_m$.
		
		\item Let $e_m \in \check \BO_m$. Then $\dim \cB_{e_m} = \ell(w_m)$, where $\cB_{e_m}$ is the Springer fiber of $e_m$.
	\end{enumerate}
\end{lemma}

In view of the lemma, the proof of \ref{thm:O(m)-n-cells} reduces to the following two steps:
\begin{enumerate}[label=(\alph*)]
	\item Compute the dominant translate $\xi_m$ and determine the parabolic subgroup $W_m$ from $\xi_m$,
	
	\item Find the orbit $\check \BO_m$ by tracing through the Springer correspondence $j_{W_m}^W(\sgn) \mapsto (\check \BO_m, \underline{\BC})$ and verify that $\check \BO_m = \check \BO(m)$.
\end{enumerate}
Both steps are done by an explicit type-by-type calculation. Details in type $A$ and $D$ are contained in \S \ref{subapp:Kac}, and the calculations in type $E$ are done by utilizing the \textsf{atlas} software. 

We have recorded the values of $\xi_m$ in simply-laced types for any $m \in [1,\check \Bh]$ in Tables \ref{tbl:Kac-A}-\ref{tbl:Kac-E8}, as they seem to be of independent interest. These values are recorded in the form of Kac diagrams:

\begin{definition}
	Let $\fg$ be a simple Lie algebra and let $\xi \in \fh_{aff}^* = \fh^* \oplus \BC \Lambda_0$ be an affine weight. The \textbf{Kac diagram} of $\xi$ is the weighted extended Dynkin diagram of $\fg$ where the $i$-th node is labeled by the number $\langle \check \alpha_i, \xi \rangle$. Here $\check \alpha_i$ denotes the $i$-th affine simple coroot.
\end{definition}

We remark that the notion of Kac diagrams originated from the work of Kac \cite{Kac:aut}. The problem of computing the Kac diagrams of weights of the form $\xi_m$ was considered in \cite{stable-grading, Reeder-Yu} in their studies of periodic gradings on semisimple Lie algebras and their applications to representations of $p$-adic groups. In fact the authors of \cite{stable-grading} have computed the Kac diagrams of $\xi_m$ for elliptic regular numbers $m$, i.e. regular numbers $m$ so that the conjugacy class $W[m]$ is elliptic, see \S 7 of \textit{loc. cit.} The Kac diagrams in type $E$ was computed earlier in \cite{deGraaf:nilp-theta}.

\begin{table}[p]
	\caption{Kac diagrams of $\xi_m$ in $A_n$ ($1 \le m \le \Bh$)}\label{tbl:Kac-A}
	Let $b$ be the largest integer with $bm \le n+1$; $v = n+1-bm$.
	\begin{tabular}{c|c|c}
		condition on $m$ & Kac diagram of $\xi_m$ & type of $W_m$ \\ \hline
		$m \mid n+1$ (i.e. $v = 0$)
		& $\underbrace{\underbrace{0 \cdots 01}_{b \text{ terms}} \cdots \underbrace{0 \cdots 01}_{b \text{ terms}}}_{m \text{ copies}}$
		& $m A_{b-1}$\\ \hline
		$m \not\mid n+1$ (i.e. $v >0$)
		& $\underbrace{\underbrace{0 \cdots 01}_{b \text{ terms}} \cdots \underbrace{0 \cdots 01}_{b  \text{ terms}}}_{m-v \text{ copies}} \underbrace{\underbrace{0 \cdots 01}_{b+1 \text{ terms}} \cdots \underbrace{0 \cdots 01}_{b+1 \text{ terms}}}_{v \text{ copies}}$
		& $v A_b + (m-v) A_{b-1}$
	\end{tabular}
	\vspace{2\baselineskip}
	\caption{Kac diagrams of $\xi_m$ in $D_n$ ($n \ge 4$, $1 \le m \le \Bh$)}\label{tbl:Kac-D}
	Let $c$ be the largest integer so that $cm \le n-1$.
	\begin{tabular}{c|c|c}
		condition on $m$ & Kac diagram of $\xi_m$ & type of $W_m$ \\ \hline
		$\substack{m=1}$
		& \begin{dynkinDiagram}[%
			labels={1,0,0,0,0,0,0,0},
			extended,
			edge length= .5cm]
			D{***...****}
		\end{dynkinDiagram}
		& $\substack{D_n}$\\ \hline
		$\substack{%
			m>n-1 \;(\text{i.e. } c=0)\\%
			m \text{ odd}}$
		& 
		\begin{dynkinDiagram}[%
			labels={1,0,1,0,1,0,1,0,1,1,1,1}, 
			extended, 
			edge length = .5cm] 
			D{***...*****...***}
			\dynkinBrace*[D_{2n-m-1}]07
			\dynkinBrace*[D_{m-n+2}]8{10}
		\end{dynkinDiagram}
		& $\substack{(n-\frac{m+1}2) A_1}$\\ \hline
		$\substack{%
			m>n-1 \;(\text{i.e. } c=0)\\%
			m \text{ even}}$
		&
		\begin{dynkinDiagram}[%
			labels={1,1,0,1,0,1,0,1,0,1,1,1,1}, 
			extended, 
			edge length = .5cm] 
			D{****...*****...***}
			\dynkinBrace*[D_{2n-m-1}]08
			\dynkinBrace*[D_{m-n+2}]9{11}
		\end{dynkinDiagram}
		& $\substack{(n-\frac {m+2}2) A_1}$\\ \hline
		$\substack{%
			m \le n-1\;(\text{i.e. } c\ge1)\\%
			m \text{ odd}\\%
			(2c+1)m \le 2n-2}$
		& 
		$\substack{%
			\begin{dynkinDiagram}[%
				labels={1,0,0,0,1,0,0,1,0,0,1,0,0,1,0,0,1,0,0,0,0}, 
				extended, 
				] 
				D{**...***...**...*...***...**...*...***...***}
				\dynkinBrace*[D_{2c+3}]04
				\dynkinBrace*[A_{2c+2}]57
				\dynkinBrace*[A_{2c+2}]8{10}
				\dynkinBrace*[A_{2c+1}]{11}{13}
				\dynkinBrace*[A_{2c+1}]{14}{16}
				\dynkinBrace*[D_{c+1}]{17}{19}
			\end{dynkinDiagram}\\%
			n-(2c+1)\frac{m-1}2 - c-2 \text{ copies of } A_{2c+2},\quad
			(c+1)m-n \text{ copies of } A_{2c+1}}$
		& $\substack{%
			D_{c+1} + \big((c+1)m-n\big) A_{2c}\\%
			+ \big(n-(2c+1)\frac{m-1}2 -c-1 \big) A_{2c+1}
		}$\\ \hline
		$\substack{%
			m \le n-1\;(\text{i.e. } c\ge1)\\%
			m \text{ even}\\%
			(2c+1)m \le 2n-2}$
		& 
		$\substack{%
			\begin{dynkinDiagram}[%
				labels={0,0,0,0,1,0,0,1,0,0,1,0,0,1,0,0,1,0,0,0,0}, 
				extended, 
				] 
				D{**...***...**...*...***...**...*...***...***}
				\dynkinBrace*[D_{c+2}]04
				\dynkinBrace*[A_{2c+2}]57
				\dynkinBrace*[A_{2c+2}]8{10}
				\dynkinBrace*[A_{2c+1}]{11}{13}
				\dynkinBrace*[A_{2c+1}]{14}{16}
				\dynkinBrace*[D_{c+1}]{17}{19}
			\end{dynkinDiagram}\\%
			n-(2c+1)\frac{m}2 - 1 \text{ copies of } A_{2c+2},\quad
			(c+1)m-n \text{ copies of } A_{2c+1}}$
		& $\substack{%
			2 D_{c+1} + \big((c+1)m-n\big) A_{2c}\\%
			+ \big(n-(2c+1)\frac{m}2 - 1 \big) A_{2c+1}
		}$\\ \hline
		$\substack{%
			m \le n-1 \;(\text{i.e. } c\ge1)\\%
			m>1 \text{ odd},\; m \not\mid 2n-1\\%
			(2c+1)m > 2n-2}$
		& 
		$\substack{%
			\begin{dynkinDiagram}[%
				labels={1,0,0,0,1,0,0,1,0,0,1,0,0,1,0,0,1,0,0,0,0}, 
				extended, 
				] 
				D{**...***...**...*...***...**...*...***...***}
				\dynkinBrace*[D_{2c+1}]04
				\dynkinBrace*[A_{2c}]57
				\dynkinBrace*[A_{2c}]8{10}
				\dynkinBrace*[A_{2c+1}]{11}{13}
				\dynkinBrace*[A_{2c+1}]{14}{16}
				\dynkinBrace*[D_{c+1}]{17}{19}
			\end{dynkinDiagram}\\%
			(2c+1)\frac{m+1}2-n-c-1 \text{ copies of } A_{2c},\quad
			n-cm-1 \text{ copies of } A_{2c+1}}$
		& $\substack{%
			D_{c+1} + (n-cm-1) A_{2c} \\%
			+ \big((2c+1)\frac{m+1}2-n-c\big) A_{2c-1} 
		}$\\ \hline
		$\substack{%
			m \le n-1 \;(\text{i.e. } c\ge1)\\%
			m>1 \text{ odd},\; m \mid 2n-1\\%
			(2c+1)m > 2n-2\\%
			(\Rightarrow (2c+1)m=2n-1)}$ 
		& 
		$\substack{%
			\begin{dynkinDiagram}[%
				labels={1,0,0,0,1,0,0,1,0,0,1,0,0,0,0}, 
				extended, 
				edge length = .5cm] 
				D{**...***...**...*...***...***}
				\dynkinBrace*[D_{2c+2}]04
				\dynkinBrace*[A_{2c+1}]57
				\dynkinBrace*[A_{2c+1}]8{10}
				\dynkinBrace*[D_{c+1}]{11}{13}
			\end{dynkinDiagram}\\%
			n-cm-2 \text{ copies of }A_{2c+1}}$
		& $\substack{(n-cm-1) A_{2c} + D_{c+1}}$\\ \hline
		$\substack{%
			2m \le n-1 \;(\text{i.e. } c>1)\\%
			m \text{ even}\\%
			(2c+1)m > 2n-2}$
		& 
		$\substack{%
			\begin{dynkinDiagram}[%
				labels={0,0,0,0,1,0,0,1,0,0,1,0,0,1,0,0,1,0,0,0,0}, 
				extended, 
				] 
				D{**...***...**...*...***...**...*...***...***}
				\dynkinBrace*[D_{c+1}]04
				\dynkinBrace*[A_{2c}]57
				\dynkinBrace*[A_{2c}]8{10}
				\dynkinBrace*[A_{2c+1}]{11}{13}
				\dynkinBrace*[A_{2c+1}]{14}{16}
				\dynkinBrace*[D_{c+1}]{17}{19}
			\end{dynkinDiagram}\\%
			(2c+1)\frac{m}2-n \text{ copies of } A_{2c},\quad
			n-cm-1 \text{ copies of } A_{2c+1}}$
		& $\substack{%
			D_{c+1} + (n-cm-1)A_{2c}\\%
			+ \big((2c+1)\frac{m}2-n\big) A_{2c-1} + D_c
		}$\\ \hline
		$\substack{%
			m \le n-1 < 2m \;(\text{i.e. } c=1)\\%
			m \text{ even}\\%
			3m = (2c+1)m > 2n-2}$
		& 
		$\substack{%
			\begin{dynkinDiagram}[%
				labels={1,1,0,1,0,1,0,0,1,0,0,1,0,0}, 
				extended, 
				edge length = .5cm] 
				D{***...*****...*****}
				\dynkinBrace*[A_2]23
				\dynkinBrace*[A_2]45
				\dynkinBrace*[A_3]68
				\dynkinBrace*[A_3]9{11}
			\end{dynkinDiagram}\\%
			\frac{3m}2-n \text{ copies of } A_2,\quad
			n-m-1 \text{ copies of } A_3}$
		& $\substack{\big( \frac{3m}2 -n +2 \big) A_1 + (n-m-1) A_2}$
	\end{tabular}		
\end{table}

\begin{table}[p]
	\caption{Kac diagrams of $\xi_m$ in $E_6$ ($1 \le m \le \Bh$)}\label{tbl:Kac-E6}
	The parentheses $(m)$ around $m$ means $m \not\in \im \cl$.
	\begin{tabular}{c|c|c||c|c|c}
		\hline
		$m$ & Kac diagram of $\xi_m$ & $W_m$
		&$m$ & Kac diagram of $\xi_m$ & $W_m$ \\ \hline
		$12$ &	\dynkin[labels={1,1,1,1,1,1,1}, extended] E6 &  $1$
		&$6$ &	\dynkin[labels={1,1,0,0,1,0,1}, extended] E6 &	$3A_1$\\
		$(11)$ &	\dynkin[labels={0,1,1,1,1,1,1}, extended] E6 &  $A_1$
		&$5$ &   \dynkin[labels={0,1,0,0,1,0,1}, extended] E6 &  $A_2+2A_1$\\
		$(10)$ &	\dynkin[labels={1,1,0,1,1,1,1}, extended] E6 &  $A_1$
		&$4$ &   \dynkin[labels={1,0,0,0,1,0,0}, extended] E6 &  $2A_2+A_1$\\
		$9$ &	\dynkin[labels={1,1,1,1,0,1,1}, extended] E6 &	$A_1$
		&$3$ &	\dynkin[labels={0,0,0,0,1,0,0}, extended] E6 &	$3A_2$\\
		$8$ &	\dynkin[labels={1,1,1,0,1,0,1}, extended] E6 &  $2A_1$
		&$2$ &   \dynkin[labels={0,0,1,0,0,0,0}, extended] E6 &  $A_5+A_1$\\
		$(7)$ &   \dynkin[labels={1,0,1,1,0,1,0}, extended] E6 &  $3A_1$
		&$1$ &   \dynkin[labels={1,0,0,0,0,0,0}, extended] E6 &  $E_6$
	\end{tabular}
	\vspace{2\baselineskip}
%
	\caption{Kac diagrams of $\xi_m$ in $E_7$ ($1 \le m \le \Bh$)}\label{tbl:Kac-E7}
	The parentheses $(m)$ around $m$ means $m \not\in \im \cl$.
	\begin{tabular}{c|c|c||c|c|c}
		\hline
		$m$ & Kac diagram of $\xi_m$ & $W_m$ 
		&$m$ & Kac diagram of $\xi_m$ & $W_m$ \\ \hline
		$18$ & \dynkin[labels={1,1,1,1,1,1,1,1}, extended] E7 & $1$
		&$9$ & \dynkin[labels={0,1,0,0,1,0,1,1}, extended] E7 & $4A_1$\\
		$(17)$ & \dynkin[labels={0, 1, 1, 1, 1, 1, 1, 1}, extended] E7 & $A_1$
		&$8$ & \dynkin[labels={1,0,0,0,1,0,1,1}, extended] E7 & $A_2+2A_1$\\
		$(16)$ & \dynkin[labels={1, 0, 1, 1, 1, 1, 1, 1}, extended] E7 & $A_1$
		&$7$ & \dynkin[labels={0,1,0,0,1,0,0,1}, extended] E7 & $A_2+3A_1$\\
		$(15)$ & \dynkin[labels={1, 1, 1, 0, 1, 1, 1, 1}, extended] E7 & $A_1$
		&$6$ & \dynkin[labels={1,0,0,0,1,0,0,1}, extended] E7 & $2A_2+A_1$\\
		$14$ & \dynkin[labels={1,1,1,1,0,1,1,1}, extended] E7 & $A_1$
		&$5$ & \dynkin[labels={0,0,0,0,1,0,0,1}, extended] E7 & $A_3+A_2+A_1$\\
		$(13)$ & \dynkin[labels={1, 1, 0, 1, 1, 0, 1, 1 }, extended] E7 & $2A_1$
		&$4$ & \dynkin[labels={1,0,0,0,0,1,0,0}, extended] E7 & $A_4+A_2$\\
		$12$ & \dynkin[labels={1,1,1,1,0,1,0,1}, extended] E7 & $2A_1$
		&$3$ & \dynkin[labels={0,0,0,0,0,1,0,0}, extended] E7 & $A_5+A_2$\\
		$(11)$ & \dynkin[labels={1, 1, 1, 0, 1, 0, 1, 0}, extended] E7 & $3A_1$
		&$2$ & \dynkin[labels={0,0,1,0,0,0,0,0}, extended] E7 & $A_7$\\
		$10$ & \dynkin[labels={1,0,1,1,0,1,0,1}, extended] E7 & $3A_1$
		&$1$ & \dynkin[labels={1,0,0,0,0,0,0,0}, extended] E7 & $E_7$
	\end{tabular}
\end{table}

\begin{table}[p]
	\caption{Kac diagrams of $\xi_m$ in $E_8$ ($1 \le m \le \Bh$)}\label{tbl:Kac-E8}
	The parentheses $(m)$ around $m$ means $m \not\in \im \cl$.
	\begin{tabular}{c|c|c||c|c|c}
		\hline
		$m$ & Kac diagram of $\xi_m$ & $W_m$
		&$m$ & Kac diagram of $\xi_m$ & $W_m$ \\ \hline
		$30$ &	\dynkin[labels={1,1,1,1,1,1,1,1,1}, extended] E8 & $1$
		&$15$ &	\dynkin[labels={1,1,0,0,1,0,1,0,1}, extended] E8 & $4A_1$\\
		$(29)$ &  \dynkin[labels={0, 1, 1, 1, 1, 1, 1, 1, 1}, extended] E8 & $A_1$
		&$14$ &	\dynkin[labels={1,1,0,0,1,0,0,1,1}, extended] E8 & $A_2+2A_1$\\
		$(28)$ &	\dynkin[labels={1, 1, 1, 1, 1, 1, 1, 1, 0}, extended] E8 & $A_1$
		&$(13)$ &	\dynkin[labels={1, 0, 0, 0, 1, 0, 1, 0, 1}, extended] E8 & $A_2+3A_1$\\
		$(27)$ &	\dynkin[labels={1, 1, 1, 1, 1, 1, 1, 0, 1 }, extended] E8 & $A_1$
		&$12$ &	\dynkin[labels={1,1,0,0,1,0,0,1,0}, extended] E8 & $A_2+3A_1$\\
		$(26)$ &	\dynkin[labels={1, 1, 1, 1, 1, 1, 0, 1, 1}, extended] E8 & $A_1$
		&$(11)$ &	\dynkin[labels={0, 1, 0, 0, 1, 0, 0, 1, 0}, extended] E8 & $2A_2+2A_1$\\
		$(25)$ &	\dynkin[labels={1, 1, 1, 1, 1, 0, 1, 1, 1}, extended] E8 & $A_1$
		&$10$ &	\dynkin[labels={1,0,0,0,1,0,0,1,0}, extended] E8 & $2A_2+2A_1$\\
		$24$ &  \dynkin[labels={1,1,1,1,0,1,1,1,1}, extended] E8 & $A_1$
		&$9$ &	\dynkin[labels={1,0,0,0,1,0,0,0,1}, extended] E8 & $A_3+A_2+A_1$\\
		$(23)$ &	\dynkin[labels={1, 1, 0, 0, 1, 1, 1, 1, 1}, extended] E8 & $2A_1$
		&$8$ & 	\dynkin[labels={0,0,0,0,1,0,0,0,1}, extended] E8 & $A_3+A_2+2A_1$\\
		$(22)$ &	\dynkin[labels={1, 0, 1, 1, 0, 1, 1, 1, 1}, extended] E8 & $2A_1$
		&$7$ &	\dynkin[labels={0,0,0,0,0,1,0,0,1}, extended] E8 & $A_4+A_2+A_1$\\
		$(21)$ &  \dynkin[labels={ 1, 1, 1, 0, 1, 0, 1, 1, 1}, extended] E8 & $2A_1$
		&$6$ &	\dynkin[labels={1,0,0,0,0,1,0,0,0}, extended] E8 & $A_4+A_3$\\
		$20$ &	\dynkin[labels={1,1,1,1,0,1,0,1,1}, extended] E8 & $2A_1$
		&$5$ &	\dynkin[labels={0,0,0,0,0,1,0,0,0}, extended] E8 & $2A_4$\\
		$(19)$ &	\dynkin[labels={1, 1, 0, 1, 1, 0, 1, 0, 1}, extended] E8 & $3A_1$
		&$4$ &	\dynkin[labels={0,0,0,0,0,0,1,0,0}, extended] E8 & $D_5+A_3$\\
		$18$ &	\dynkin[labels={1,1,1,1,0,1,0,1,0}, extended] E8 & $3A_1$
		&$3$ &	\dynkin[labels={0,0,1,0,0,0,0,0,0}, extended] E8 & $A_8$\\
		$(17)$ &	\dynkin[labels={0, 1, 1, 0, 1, 0, 1, 0, 1}, extended] E8 & $4A_1$
		&$2$ & 	\dynkin[labels={0,1,0,0,0,0,0,0,0}, extended] E8 & $D_8$\\
		$(16)$ &	\dynkin[labels={1, 0, 1, 1, 0, 1, 0, 1, 0}, extended] E8 & $4A_1$
		&$1$ &	\dynkin[labels={1,0,0,0,0,0,0,0,0}, extended] E8 & $E_8$\\
	\end{tabular}
\end{table}

\clearpage

\section{Combinatorics in classical types}\label{sec:co}

\subsection{Highest $h$-weights and values of $\cl$}\label{subapp:partition}

In this subsection, we provide proofs of Lemma \ref{lem:max-h-hts} and \ref{values:A}-\ref{values:B}. Our notations for partitions follows \ref{notn:partition}.

\begin{lemma}[{\ref{lem:max-h-hts}}]\label{lem:max-h-hts:pf}
	Let $\fl \subset \fg$ be a Levi, let $\{h,e,f\} \subset \fl$ be an $\fsl_2$-triple where $h \in \fh$ and $e$ is distinguished in $\fl$. Then
	\begin{equation*}
		\max_{\alpha \in R(\fl,\fh)} \langle \alpha, h \rangle \le 
		\max_{\beta \in R(\fg,\fh)} \langle \beta, h \rangle \le 
		1+ \max_{\alpha \in R(\fl,\fh)} \langle \alpha, h \rangle .
	\end{equation*}
	In other words, if the highest $h$-weight on $\fl$ is $2a$, then the highest $h$-weight on $\fg$ is either $2a$ or $2a+1$.
\end{lemma}
	
\begin{proof}
	The first inequality is clear. The second inequality is verified type by type. The exceptional cases are verified using the \textsf{atlas} software. It remains to consider classical types. 
	
	Let $\BO_q \in \ubar \cN$ be any orbit and let $q = (q_1,q_2,\ldots)$ be its partition. Let $\{h_d,e_d,f_d\} \subset \fg$ be an $\fsl_2$-triple with $e \in \BO_q$ and $h_d \in \fh$ dominant. Then 
	\begin{equation*}
		\max_{\beta \in R(\fg,\fh)} \langle \beta,h \rangle = \langle \theta, h_d \rangle.
	\end{equation*}
	Let $L$ be a Bala-Carter Levi of $\BO$, let $L_1,\ldots, L_s$ be simple factors of $L$, let $\BO_i$ be the distinguished orbit in $\fl_i$ so that $\BO = \Sat_L^G (\BO_1 \times \cdots \times \BO_s)$, and let $p^i$ be the partition of $\BO_i$. Then
	\begin{equation*}
		\max_{\alpha \in R(\fl,\fh)} \langle \alpha,h \rangle
		= \max\{ \langle \theta_1, h_{1,d} \rangle, \ldots, \langle \theta_s, h_{s,d}\rangle  \}
	\end{equation*}
	where $\theta_i$ is the highest root of $L_i$ with respect to some choice of Borel containing $\fh$ and $\{h_{i,d},e_{i,d}, f_{i,d}\}$ is an $\fsl_2$-triple for $\BO_i$ with $h_{i,d} \in \fh$ dominant with respect to this Borel. Therefore, it is enough to compare $\langle \theta, h_d \rangle$ with the largest $\langle \theta_i, h_{i,d} \rangle$. Embed $\fg$ into some $\fgl(N)$, choose $\fh$ to be the diagonal Cartan, and choose a Borel in $\fg$ as in \cite[5.2]{Collingwood-McGovern}, then the description of $h_d$ in terms of the partition $q$ is written down in \cite[3.6, 5.3]{Collingwood-McGovern}, and the formula for the highest root $\theta$ can be found in \cite[Plate I-IV]{Bourbaki:Lie:4-6}. The Bala-Carter Levi and the partitions $p^i$ of $\BO_i$ can be found following \cite[2.3, 2.6.5]{LMBM:unip}. We make use of these formulas in the calculations below without further citing these works.
	
	\textbf{Type $A_n$}. We have
	\begin{equation}\label{eqn:h-in-BCD}
		h = D :=
		\diag(\underbrace{q_1-1, q_1-3,\ldots, 1-q_1}_{q_1 \text{ parts}}, \underbrace{q_2-1, q_2-3, \ldots, 1-q_2}_{q_2 \text{ parts}},\ldots, \underbrace{q_s-1, q_s-3, \ldots , 1-q_s}_{q_s \text{ parts}})
	\end{equation}
	and $h_d$ is the dominant Weyl group translate of $h$, where dominance means the entries are weakly descending. The largest (resp. smallest) entry in $h_d$ is $q_1-1$ (resp. $1-q_1$). Hence $h_d = \diag(q_1-1,  \ldots,  1-q_1)$. The highest root is $\theta = \eps_1 - \eps_n$, where $\eps_i$ takes the $i$-th entry. Therefore $\langle \theta, h_d \rangle = (q_1-1)-(1-q_1) = 2q_1-2$. On the other hand,	the Bala-Carter Levi for $\BO_q$ is $\fl = \fgl(q_1) \times \cdots \times \fgl(q_s)$, where $\BO_i$ is the regular orbit in $\fgl(q_i)$ with partition $p^i = (q_i)$. Applying the formula for $\langle \theta, h_d \rangle$ to $\langle \theta_i, h_{i,d} \rangle$, we see $\langle \theta_i, h_{i,d} \rangle = 2q_i-2$. Therefore
	\begin{equation*}
		\langle \theta, h_d \rangle = 2q_1-2 = \max\{ 2q_1-2,\ldots, 2q_s-2 \} = \max\{ \langle \theta_i, h_{i,d} \rangle\}
	\end{equation*}
	as required.
	
	\textbf{Type $C_n$}. In this case $h = \diag(D,-D)$, where $D$ is equal to (\ref{eqn:h-in-BCD}). So $h_d = \diag(q_1-1, \ldots)$, i.e. the first entry is $q_1-1$. The highest root is $\theta= 2 \eps_1$. So $\langle \theta, h_d \rangle = 2(q_1-1) = 2q_1-2$. On the other hand, $q$ can be decomposed uniquely as $q = (a_1{}^2,\ldots, a_t{}^2) \cup \gamma$, where the sequence $a_i$ is weakly descending, and $\gamma$ has only even parts, each occurring once. Then the Bala-Carter Levi of $\BO_q$ is $\fl_1 \times \cdots \times \fl_t \times \fl_0 = \fgl(a_1) \times \cdots \times \fgl(a_t) \times \fsp(|\gamma|)$, the partition for $\BO_i$ in $\fl_i = \fgl(a_i)$ is $p^i = (a_i)$, and the partition for $\BO_0$ in $\fl_0 = \fsp(|\gamma|)$ is $\gamma$. Hence $\langle \theta_i, h_{i,d} \rangle = 2a_i-2$ for $1 \le i \le t$ (applying the formula for $\langle \theta,h_d \rangle$ in type $A$), $\langle \theta_0, h_{0,d} \rangle = 2 \gamma_1-2$ (applying the formula for $\langle \theta, h_d \rangle$ in type $C$), and $\max\{ \langle \theta_i, h_{i,d} \rangle\} = \max\{ 2a_1-2, 2\gamma_1-2\}$.
	
	We have two cases. If $q_1 = q_2$, then $a_1 = q_1 \ge \gamma_1$ since for any repeated part, at least two of its occurrence must come from the $a_i$'s. Then
	\begin{equation*}
		\langle \theta, h_d \rangle = 2q_1 -2 = \max\{ 2q_1-2 = 2a_1-2,\; 2\gamma_1-2 \} = \max\{ \langle \theta_i, h_{i,d} \rangle\}.
	\end{equation*}
	If $q_1 > q_2$, then $q_1$ is non-repeated, and $\gamma_1 = q_1 > a_1$. So 
	\begin{equation*}
		\langle \theta, h_d \rangle = 2q_1 -2 = \max\{ 2a_1-2,\; 2q_1-2 = 2\gamma_1-2 \} = \max\{ \langle \theta_i, h_{i,d} \rangle\}.
	\end{equation*}
	In either case, the desired inequality holds.
	
	\textbf{Type $D_n$}. If $e$ is not very even, then $h = \diag(D,-D)$, where $D$ is again equal to (\ref{eqn:h-in-BCD}). If $e$ is very even, then we may need to replace $h$ by a translate under an outer automorphism of $\fg$, but that does not affect the numbers we are computing. So $\eps_1(h_d) = q_1-1$, and $\eps_2(h_d) = \max\{q_1-3, q_2-1\}$. The highest root is $\theta = \eps_1+\eps_2$. Hence $\langle \theta, h_d \rangle = q_1-1 + \max\{q_1-3, q_2-1\}$. Decompose $q$ uniquely as $q = (a_1{}^2,\ldots, a_t{}^2) \cup \gamma$ where $\gamma$ has only odd parts, each occurring once. Then the Bala-Carter Levi of $\BO_q$ is $\fl_1 \times \cdots \times \fl_t \times \fl_0 = \fgl(a_1) \times \cdots \times \fgl(a_t) \times \fso(|\gamma|)$, the partition for $\BO_i$ in $\fl_i = \fgl(a_i)$ is $p^i = (a_i)$, and the partition for $\BO_0$ in $\fl_0 = \fso(|\gamma|)$ is $\gamma$. Hence $\langle \theta_i, h_{i,d} \rangle = 2a_i-2$ for $1 \le i \le t$ (applying the formula for $\langle \theta,h_d \rangle$ in type $A$), and $\langle \theta_0, h_{0,d} \rangle = \gamma_1-1 + \max\{\gamma_1-3, \gamma_2-1\}$ (applying the formula for $\langle \theta,h_d \rangle$ in type $D$). Note that since parts in $\gamma$ are odd and non-repeating, we have $\gamma_2 \le \gamma_1 -2$. Hence $\gamma_2-1 \le \gamma_1-3$, and so $\langle \theta_0, h_{0,d} \rangle = \gamma_1-1+\gamma_1-3 = 2\gamma_1-4$. Therefore $\max\{ \langle \theta_i, h_{i,d} \rangle\} = \max\{ 2a_1-2, 2\gamma_1-4\}$.
	
	We again have several cases. If $q_2 = q_1$, then $a_1 = q_1 \ge \gamma_1$, and $2\gamma_1-4 < 2a_1-2$. Hence
	\begin{multline*}
		\langle \theta, h_d \rangle = q_1-1+ \max\{q_1-3, q_2-1\} = q_1-1+ q_1-1\\
		= 2q_1 -2 = \max\{ 2q_1-2 = 2a_1-2,\; 2\gamma_1-4 \} = \max\{ \langle \theta_i, h_{i,d} \rangle\}.
	\end{multline*}
	If $q_2 = q_1-1$, then $q_1$ is non-repeating, and so $\gamma_1 = q_1> a_1$. We also see that $\max\{q_1-3,q_2-1\} = \max\{q_1-3, q_1-2\} = q_1-2$. So
	\begin{equation*}
		\langle \theta, h_d \rangle = q_1-1 + \max\{q_1-3,q_2-1\}
		= q_1-1 + q_1-2 = 2q_1-3.
	\end{equation*}
	Also, $q_1 > a_1$ implies $a_1 \le q_1-1$ and hence $2a_1-2 \le 2(q_1-1)-2 = 2q_1-4 = 2 \gamma_1-4$, and so $\max\{ \langle \theta_i, h_{i,d} \rangle\} = \max\{ 2a_1-2, 2\gamma_1-4\} = 2 q_1-4$. Therefore
	\begin{equation*}
		\langle \theta, h_d \rangle = 2q_1-3 = 1+ (2q_1-4) = 1+ \max\{ \langle \theta_i, h_{i,d} \rangle\}.
	\end{equation*}
	If $q_2 < q_1-1$, then again $\gamma_1 = q_1 > a_1$. However in this case $\max\{q_1-3,q_2-1\} = q_1-3$, and so $\langle \theta, h_d \rangle = q_1-1 + q_1-3 = 2q_1-4$. So
	\begin{equation*}
		\langle \theta, h_d \rangle = 2q_1-4 = \max\{ \langle \theta_i, h_{i,d} \rangle\}.
	\end{equation*}
	In all cases, the desired inequality holds.
	
	\textbf{Type $B_n$}. The argument is identical to type $D_n$ (the only change is now $h_d = \diag(0,D, -D)$, where $D$ is equal to (\ref{eqn:h-in-BCD}) with a zero entry removed. But this does not affect rest of the argument).
\end{proof}


Note that the above proof in particular contains explicit calculations of the quantity $\max_{\alpha \in R(\fl,\fh)} \langle \alpha, h \rangle = \max\{ \langle \theta_i, h_{i,d}\rangle\}$, and as an immediate byproduct, the values of $\cl_n(\BO) = 1 + \frac12 \max_{\alpha \in R(\fl,\fh)} \langle \alpha, h \rangle$ for any orbit $\BO$ in classical types. 

\subsection{Kac diagrams and Lusztig's bijection}\label{subapp:Kac}

This subsection contains the calculation of Kac diagrams of $\xi_m$ and the proof of Theorem \ref{thm:O(m)-n-cells} for type $A$ and $D$.

\begin{clause}[Proof of \ref{thm:O(m)-n-cells} in type $A_n$]\label{cls:Kac-A}
	In type $A_n$, every node in the extended Dynkin diagram is a special node. Moreover, for every $m \in [1,\check \Bh] = [1,n+1]$, there is at least one simple root not orthogonal to $\xi_m$. So there is at least one node in the Kac diagram of $\xi_m$ that is labeled by $1$. Treating that as a special node, we may invoke Lemma \ref{lem:L-bij-as-jind}(2) which says that $\check \BO_m$ is the Richardson orbit induced from the parabolic corresponding to $W_m$. The induction of orbits can be computed easily in terms of partitions. Hence it remains to compute $W_m$.
	
	Recall that in the finite case, we have an identification
	\begin{equation*}
		\{\text{integral weights in } \fh^*\} \bijects
		\{ \ba = (a_0, a_1, \ldots, a_{n-1}, a_n) \in \BZ^{n+1} \} /\BZ 
	\end{equation*}
	where $a \in \BZ$ acts by adding $a$ to each entry. The Weyl group can be identified with the set of all permutations of the $a_i$'s. We have $\langle \check \alpha_i, \ba \rangle = a_{i-1}-a_i$, and $\ba$ is dominant if and only if $a_i \ge a_{i+1}$ for any $i$. In the affine case, we have instead
	\begin{multline}\label{eqn:aff-wts-An}
		\{\text{integral weights in } \fh_{aff}^*\} \bijects\\
		\sP := \{ (\ldots, a_{i-1},a_i, a_{i+1},\ldots) \in \BZ^\BZ \mid  a_{i-1} - a_i = a_{n+1+i-1} - a_{n+1+i} \text{ for any i}\} / \BZ
	\end{multline}
	where $a \in \BZ$ acts by adding $a$ to each entry.	The extended Weyl group can be identified with the set of permutations of the indexing set $\BZ$ commuting with the $(n+1)$-shifting operator \cite[1.2]{Shi:some}. In particular, two tuples $\ba=(a_i)_{i \in \BZ}$, $\bc=(c_i)_{i \in \BZ}$ in $\sP$ are in the same $W_{ext}$-orbit if and only if for each integer $x$, the multiplicity of $x$ (i.e. the number of occurrences of $x$) in $\ba$ and in $\bc$ are the same. Again we have $\langle \check \alpha_i, \ba \rangle = a_{i-1}-a_i$, and a tuple $\ba$ is dominant if and only if $a_{i-1} \ge a_i$ for any $i$. 
	
	Consider the integral weight $m \Lambda_0 + \rho$, and let $\ba=(a_i)_{i \in \BZ} \in \sP$ be the corresponding tuple. This weight satisfies $\langle \check \alpha_0, m \Lambda_0 + \rho \rangle = m- \check \Bh+1 = m-n$ and $\langle \alpha_i, m \Lambda_0 + \rho \rangle = 1$ for any $1 \le i \le n$. So we have $a_n = 0$, $a_{k(n+1)-1} - a_{k(n+1)} = m-n$ for any $k \in \BZ$, and $a_{i-1} - a_i = 1$ for $i \not \in (n+1)\BZ$, i.e. 
	\begin{align*}
		\ba &= 
		\big( \ldots,\; 2m,\;
		\underbrace{m + n,\; m +n-1,\; \ldots,\; m +1 ,\;m}_{n+1 \text{ parts}},\\
		&\qquad \underbrace{n,\; n-1,\;\ldots,\; 1,\;0}_{n+1 \text{ parts}},\; 
		\underbrace{-m+n,\; -m+n -1,\;\ldots,\; -m +1,\; -m }_{n+1 \text{ parts}},\;
		-2m+n ,\; \ldots \big)\\
		&= \cdots \cup [n+2m, 2m] \cup [n+m,m] \cup [n,0] \cup [n-m, -m] \cup [n-2m,-2m] \cup \cdots
	\end{align*}
	where $[i,j]$ denotes the sequence $(i,i-1,\ldots,j+1,j)$ and $\cup$ denotes concatenation of sequences.	From the above discussion, the dominant translate $\xi_m$ of $m \Lambda_0 + \rho$ corresponds to the following tuple which we again denote by $\xi_m$ by an abuse of notation
	\begin{equation*}
		\xi_m = 
		\big( \ldots, \underbrace{2,\ldots,2}_{m_2 \text{ copies}}, \underbrace{1, \ldots,1}_{m_1 \text{ copies}}, \underbrace{0,\ldots,0}_{m_0 \text{ copies}}, \underbrace{-1,\ldots,-1}_{m_{-1} \text{ copies}},\ldots \big).
	\end{equation*}
	Here we put the last $0$ at the $n$-th position, and $m_x$ denotes the multiplicity $x$ in $\ba = (a_i)_{i \in \BZ}$. It remains to compute the numbers $m_x$. Let $b$ be the largest integer so that $bm \le n+1$, and let $v = n+1-bm$ (so that $(m^b,v)$ is a partition of $n+1$). Among the sequences $[n+km,-km]$ ($k \in \BZ$), it is easy to see that the numbers $0,1,\ldots, v-1$ and $m$ appear in and only in $[n,0], [n-m,-m], \ldots, [n-bm,-bm]$, and the numbers $v, v +1, \ldots, m-1$ appear in and only in $[n,0], [n-m,-m], \ldots, [n-(b-1)m,-(b-1)m]$. Hence
	\begin{equation*}
		m_0 = m_1 = \cdots = m_{v -1} = m_m = b+1,
	\end{equation*}
	\begin{equation*}
		m_v = m_{v+1} = \cdots = m_{m-1} = b.
	\end{equation*}
	Therefore
	\begin{equation*}
		\xi_m = 
		\big( \ldots, \underbrace{m,\ldots,m}_{b+1 \text{ copies}},
		\underbrace{%
			\underbrace{m-1,\ldots,m-1}_{b \text{ copies}},\ldots, \underbrace{v, \ldots,v}_{b \text{ copies}}, %
			\underbrace{v-1, \ldots,v-1}_{b+1 \text{ copies}}, \ldots, \underbrace{0,\ldots,0}_{b+1 \text{ copies}},%
		}_{n+1 \text{ parts in total}} 
		\underbrace{-1,\ldots,-1}_{b \text{ copies}},\ldots \big).
	\end{equation*}
	By the periodic nature of $\xi_m$, the parts described above completely determines $\xi_m$. Thus
	\begin{equation}\label{eqn:Kac-diag-A}
		\text{Kac diagram of } \xi_m =
		\begin{cases}
			\underbrace{\underbrace{0 \cdots 01}_{b \text{ parts}} \cdots \underbrace{0 \cdots 01}_{b  \text{ parts}}}_{m-v \text{ copies}} \underbrace{\underbrace{0 \cdots 01}_{b+1 \text{ parts}} \cdots \underbrace{0 \cdots 01}_{b+1 \text{ parts}}}_{v \text{ copies}}, & \text{if } v \neq 0;\\
			\underbrace{\underbrace{0 \cdots 01}_{b \text{ parts}} \cdots \underbrace{0 \cdots 01}_{b \text{ parts}}}_{m \text{ copies}}, & \text{if } v = 0
		\end{cases}
	\end{equation}
	where we have written the Kac diagram as a string for convenience (instead of as a circle). The Levi $\check \fl_m \subset \check \fg$ corresponding to $W_m = \Stab_{W_{aff}}(\xi_m)$ is
	\begin{equation*}
		\check L_m = 
		\begin{cases}
			\fgl(b+1)^v \times \fgl(b)^{m-v}, & v \neq 0;\\
			\fgl(b)^m, & v = 0.
		\end{cases}
	\end{equation*}
	The zero orbit of $\fgl(b+1)$ (resp. $\fgl(b)$) has partition $(1^{b+1})$ (resp. $(1^b)$), and the induction of the product of these orbits to $\check G$ (i.e. the Richardson orbit corresponding to $\check \fl_m$) is given by the partition which is the sum of the $(1^{b+1})$ and the $(1^b)$'s (see Notation \ref{notn:partition} for the meaning of a sum of partitions $p+q$), which equals $q(m) = (m^b,v)$. Hence $\check \BO_m = \check \BO(m)$ by Lemma \ref{lem:L-bij-as-jind}(2). This concludes the proof in type $A_n$.
\end{clause}

\begin{clause}[Kac diagram of $\xi_m$ in type $D_n$]
	In this discussion we sketch our approach for finding the Kac diagram of $\xi_m$ (Table \ref{tbl:Kac-D}).
	
	As in type $A$, we would like to represent affine weights of type $D_n$ as a tuple of integers. In the finite case, we have
	\begin{equation*}
		\fh^* \bijects \{ \ba = ( a_1,a_2,\ldots, a_n,-a_n, \ldots, -a_2,-a_1) \mid a_i \in \BC\}
	\end{equation*}
	and the Weyl group is generated by transpositions $s_i$ ($1 \le i \le n-1$) which swaps $a_i \leftrightarrow a_{i+1}$ and $-a_i \leftrightarrow -a_{i+1}$, and $s_n$ which swaps $a_{n-1} \leftrightarrow -a_n$ and $a_n \leftrightarrow -a_{n-1}$. We have $\langle \check \alpha_i, \ba \rangle = a_i - a_{i+1}$ for $1 \le i \le n-1$, and $\langle \check \alpha_n, \ba \rangle = a_{n-1} + a_n = a_{n-1} - (-a_n)$. In the affine case, we have
	\begin{multline*}
		\fh_{aff}^* \bijects
		\Big\{ \ba = ( \ldots, b_n^{(2)}, a_1^{(1)}, \ldots, a_n^{(1)}, b_n^{(1)}, \ldots, b_1^{(1)}, a_1^{(0)}, \ldots, a_n^{(0)}, b_n^{(0)}, \ldots, b_1^{(0)}, a_1^{(-1)}, \ldots) \in \BC^\infty \mid \\
		a_i^{(j)} - a_{i+1}^{(j)} = b_{i+1}^{(j)} - b_i^{(j)} \text{ for any } 1 \le i, i+1 \le n \text{ and any } j,\\
		a_n^{(j)} - b_{n-1}^{(j)} = a_{n-1}^{(j)} - b_n^{(j)} \text{ for any } j,\;
		b_1^{(j)} - a_2^{(j-1)} = b_2^{(j)} - a_1^{(j-1)} \text{ for any } j,\\
		a_i^{(0)} = -b_i^{(0)} \text{ for any } i \qquad\Big\}.
	\end{multline*}
	The simple coroots pair by
	\begin{equation}\label{eqn:<crt,Zinf>-Dn-a}
		\langle \check \alpha_i, \ba \rangle = a_i^{(j)} - a_{i+1}^{(j)} = b_{i+1}^{(j)} - b_i^{(j)} \text{ for any } 1 \le i \le n-1 \text{ for any } j,
	\end{equation}
	\begin{equation}\label{eqn:<crt,Zinf>-Dn-b}
		\langle \check \alpha_n, \ba \rangle  = a_n^{(j)} - b_{n-1}^{(j)} = a_{n-1}^{(j)} - b_n^{(j)} \text{ for any } j,
	\end{equation}
	\begin{equation}\label{eqn:<crt,Zinf>-Dn-c}
		\langle \check \alpha_0, \ba \rangle  = b_1^{(j)} - a_2^{(j-1)} = b_2^{(j)} - a_1^{(j-1)} \text{ for any } j.
	\end{equation}
	The affine Weyl group is generated by transpositions $s_i$ ($1 \le i \le n-1$) which swaps $a_i^{(j)} \leftrightarrow a_{i+1}^{(j)}$ and $b_i^{(j)} \leftrightarrow b_{i+1}^{(j)}$, $s_n$ which swaps $a_n^{(j)} \leftrightarrow b_{n-1}^{(j)}$ and $a_{n-1}^{(j)} \leftrightarrow b_n^{(j)}$, and $s_0$ which swaps $b_1^{(j)} \leftrightarrow a_2^{(j-1)}$ and $b_2^{(j)} \leftrightarrow a_1^{(j-1)}$. 
	
	Note that the tuples $\ba$ described above satisfy the periodic conditions for affine weights of type $A_{2n-1}$ (see (\ref{eqn:aff-wts-An})), hence it makes sense to say whether $\ba$ is dominant as an affine weight for $A_{2n-1}$. The following claim states the relationship between two notions of dominance. Let $\Omega(D_n)$ be the automorphism group of the extended Dynkin diagram of type $D_n$, so that $W_{ext}(D_n) = W_{aff}(D_n) \rtimes \Omega(D_n)$. The group $\Omega(D_n)$ contains the elements $\sigma_{02}$ which swaps the $0$th and the $2$nd nodes, and $\sigma_{n-1,n}$ which swaps the $(n-1)$-th and the $n$-th nodes. Then $\sigma_{02}$ acts on $\ba$ by swapping $b_1^{(j)} \leftrightarrow a_1^{(j-1)}$ for any $j$, and $\sigma_{n-1,n}$ acts by swapping $a_n^{(j)} \leftrightarrow b_n^{(j)}$ for any $j$. 
	
	\begin{claim}\label{clm:dom-wt-Dn-vs-An}
		Let $\ba \in \fh_{aff}^*$ and let $\ubar\ba = (\ubar a_i^{(j)}, \ubar b_i^{(j)})_{i,j \in \BZ} \in \fh_{aff}^*$ be its dominant translate under $W_{aff}(D_n)$. Up to the action of an element of $\Omega(D_n)$, $\ubar \ba$ is dominant as an affine weight of type $A$ and the entries of $\ubar \ba$ satisfy $\ubar a_i^{(j)} = -\ubar b_i^{(-j)}$ for any $i$ and $j$.
	\end{claim}
	
	Therefore, to find the dominant translate of a weight, it is enough to find its $W_{ext}(A_{2n-1})$-translate which is dominant for affine $A_{2n-1}$. The latter is easy to do, see \ref{cls:Kac-A}.
	
	\begin{proof}
		According to the formulas (\ref{eqn:<crt,Zinf>-Dn-a})-(\ref{eqn:<crt,Zinf>-Dn-c}), each entry in $\ubar \ba$ is greater than or equal to the entry on its right, except it is possible to have $\ubar a_n^{(j)} < \ubar b_n^{(j)}$ for all $j$, and/or $\ubar b_1^{(j)} < \ubar a_1^{(j)}$ for all $j$. If we are in the first situation, the conditions $\langle \check \alpha_{n-1}, \ubar \ba \rangle \ge 0$ and $\langle \check \alpha_n, \ubar \ba \rangle \ge 0$ force $\ubar a_{n-1}^{(j)} \ge \ubar b_n^{(j)} > \ubar a_n^{(j)} \ge \ubar b_{n-1}^{(j)}$. Apply $\sigma_{n-1,n} \in \Omega(D_n)$ to $\ubar \ba$. Then $\ubar a_n^{(j)}$ and $\ubar b_n^{(j)}$ are swapped, and the entries of the resulting tuple satisfy $\ubar a_{n-1}^{(j)} \ge \ubar a_n^{(j)} > \ubar b_n^{(j)} \ge \ubar b_{n-1}^{(j)}$. Similarly, if we are in the second situation, then we may apply $\sigma_{02} \in \Omega(D_n)$ to $\ubar \ba$. Then each entry in the resulting tuple is greater than or equal to entries on its right, which is dominant as an affine weight of type $A$.
		
		The equality $a_i^{(j)} = - b_i^{(-j)}$ is satisfied by $\ba$, and this property is preserved by the simple reflections $s_i$ of $W_{aff}(D_n)$ and the elements $\sigma_{02}$, $\sigma_{n-1,n}$ of $\Omega(D_n)$. Hence the entries of $\ubar \ba$ satisfy $\ubar a_i^{(j)} = -\ubar b_i^{(-j)}$ as well. 
	\end{proof}
	
	Consider the weight $m \Lambda_0 + \rho$ ($1 \le m \le \check \Bh = 2n-2$), and let $\ba \in \BC^\infty$ be the corresponding tuple. Then $\langle \check \alpha_i, \ba \rangle = 1$ for $1 \le i \le n$, and $\langle \check \alpha_0, \ba \rangle = m - \check \Bh +1 = m-2n+3$. Hence
	\begin{multline*}
		\ba =
		\cdots \cup [2m+n-1, 2m] \cup [2m, 2m-n+1] 
		\cup [m+n-1,m] \cup [m, m-n+1]\\
		\cup [n-1,0] \cup [0,-n+1] 
		\cup [-m+n-1, -m] \cup [-m, -m-n+1] \\
		\cup [-2m+n-1,-2m] \cup [-2m,-2m-n+1] \cup \cdots
	\end{multline*}
	where $[c,d]$ denotes the sequence $(c,c-1,\ldots, d+1,d)$, and $\cup$ denotes concatenation of sequences. By Claim \ref{clm:dom-wt-Dn-vs-An} and the argument in \ref{cls:Kac-A}, its dominant translate $\xi_m$ can be easily constructed once the multiplicity of each integer in $\ba$ is found. We do this next.
	
	Observe that we can instead find the multiplicity of each integer in the following sequence
	\begin{multline*}
		\ba' = \cdots \cup [2m+n-1, 2m-n+1] \cup [m+n-1,m-n+1] \cup [n-1,-n+1] \\
		\cup [-m+n-1,-m-n+1] \cup [-2m+n-1, -2m-n+1] \cup \cdots.
	\end{multline*}
	The multiplicity of an integer $x$ in $\ba'$ is equal to the multiplicity of $x$ in $\ba$, except if $x \in m \BZ$ then $x$ appears one more time in $\ba$ than in $\ba'$. Note the intervals in $\ba'$ are simply $[n-1,-n+1]$ shifted by multiples of $m$. Let
	\begin{align*}
		c_x^+ &= \text{the largest integer with } c_x^+ m \le n-1 + x,\\
		c_x^- &= \text{the largest integer with } c_x^- m \le n-1 - x,\\
		c_x &= c_x^+ + c_x^-
	\end{align*}
	(for $x =0$, $c_x^+ = c_x^- = c$, where $c$ is the integer that appears in Table \ref{tbl:Kac-D}). Then $x$ appears $c_x^+ +1$ times in $\cdots \cup [m+n-1,m-n+1] \cup [n-1,-n+1]$, and appears $c_x^-$ times in $[-m+n-1,-m-n+1] \cup [-2m+n-1, -2m-n+1] \cup \cdots$. So $x$ appears $c_x^+ + c_x^- + 1 = c_x + 1$ times in $\ba'$. Let
	\begin{equation*}
		m_x := c_x + 1 + \delta_{x,m\BZ}, \text{ where }
		\delta_{x,m\BZ} = 
		\begin{cases}
			1 & x \in m \BZ\\
			0 & x \not\in m \BZ.
		\end{cases}
	\end{equation*}
	Then $x$ appears $m_x$ times in $\ba$. 
	
	Our next task is to compute $c_x$ inductively, and ultimately write down a formula for $c_x$ for any non-negative integer $x$. Summing up the inequalities
	\begin{align*}
		c_x^+ m \le n-1 + x < (c_x^+ +1)m\\
		c_x^- m \le n-1-x < (c_x^- +1)m
	\end{align*}
	we obtain
	\begin{equation*}
		c_x m \le 2n-2 < (c_x +2)m.
	\end{equation*}
	Hence $c_x \in \{\ell, \ell-1\}$ for all $x$, where 
	\begin{center}
		$\ell :=$ the largest integer so that $\ell m \le 2n-2$.
	\end{center}
	The following lemma describes precisely when the value of $c_x$ changes as $x$ increases. The proof is straightforward combinatorics.
	
	\begin{lemma}~
		\begin{itemize}
			\item If $c_x = \ell-1$, then $c_x = c_{x+1} = \cdots = c_{(c_x^+ +1)m - n} = \ell-1$, and $c_{(c_x^+ +1)m - n+1} = \ell$.
			
			\item If $c_x = \ell$, then $c_x = c_{x+1} = \cdots = c_{n- c_x^- m-1} = \ell$, and
			\begin{equation*}
				c_{n - c_x^- m} = 
				\begin{cases}
					\ell & \text{if } m \mid 2n-1 ;\\
					\ell-1 & \text{otherwise}.
				\end{cases}
			\end{equation*}
		\end{itemize}
	\end{lemma}
	
	The following corollary records the values of $c_x$ for any $x$. It is obtained by keeping track of the changes of $c_x$ as $x$ increases starting from $x = 0$. The conditions $2c = \ell$ and $2c = \ell-1$ below correspond to whether $c_0 = \ell$ or $c_0 = \ell-1$. Recall that $c = c_0^+ = c_0^-$ (so $2c = c_0$).
	
	\begin{corollary}\label{cor:cx}~
		\begin{itemize}
			\item Suppose $2c = \ell$ (equivalently $(2c+1)m > 2n-2$) and $m \not\mid 2n-1$. Let $x_0 = 0$ and let
			\begin{equation*}
				x_i = 
				\begin{cases}
					(c + \tfrac i2) m - n+1 & \text{if } i \text{ even},\\
					n - (c - \tfrac{i-1}2)m & \text{if } i \text{ odd}.
				\end{cases}
			\end{equation*}
			Then
			\begin{equation*}
				c_x = 
				\begin{cases}
					\ell & \text{if } x \in [x_0, x_1-1] \cup [x_2,x_3-1] \cup \cdots\\
					\ell-1 & \text{if } x \in [x_1,x_2-1] \cup [x_3,x_4-1] \cup \cdots
				\end{cases}
			\end{equation*}
			
			\item Suppose $2c = \ell$ (equivalently $(2c+1)m > 2n-2$) and $m \mid 2n-1$. Then
			\begin{equation*}
				c_x = \ell \text{ for all } x \ge 0.
			\end{equation*}
			
			\item Suppose $2c = \ell-1$ (equivalently $(2c+1)m \le 2n-2$). Let $x_0 = 0$ and let
			\begin{equation*}
				x_i = 
				\begin{cases*}
					n - (c- \tfrac{i-2}2)m & \text{if } i \text{ even},\\
					(c+ \tfrac{i+1}2)m -n + 1 & \text{if } i \text{ odd}.
				\end{cases*}
			\end{equation*}
			Then 
			\begin{equation*}
				c_x = 
				\begin{cases}
					\ell-1 & \text{if } x \in [x_0, x_1-1] \cup [x_2,x_3-1] \cup \cdots\\
					\ell & \text{if } x \in [x_1,x_2-1] \cup [x_3,x_4-1] \cup \cdots.
				\end{cases}
			\end{equation*}
		\end{itemize}
	\end{corollary}
	
	To obtain a formula for $m_x = c_x + 1 + \delta_{x,m\BZ}$ (the multiplicity of $x$ in $\ba$), we need to determine the values of $\delta_{x,m\BZ}$, which amounts to finding the positions of each integer multiple $ym$ of $m$ in the intervals $[x_i, x_{i+1}-1]$. The following lemma does not describe the precise location of the $ym$'s, but it describes the intervals in which the $ym$'s lie.
	
	\begin{lemma}\label{lem:ym}
		For any $y \in \BZ_{\ge 0}$, we have $ym \in [x_{2y}, x_{2y+1}-1]$, where the $x_i$'s are defined as in the preceding corollary. If $2c = \ell$ and $m \mid 2n-1$, we define the $x_i$'s as in the case $2c = \ell$, $m \not\mid 2n-1$.
	\end{lemma}
	
	Using these facts, one can write down $\xi_m$ as an element $\ubar \ba \in \BC^\infty$. Namely, 
	\begin{equation*}
		\ubar \ba = 
		\Big( \ldots,
		\underbrace{2,\ldots,2}_{m_2 \text{ copies}}
		\underbrace{1,\ldots,1}_{m_1 \text{ copies}},
		\underbrace{0,\ldots,\boxed{0}}_{c_0^+ +1 \text{ copies}}, \underbrace{0,\ldots,0}_{c_0^+ +1 \text{ copies}},\ldots \Big)
	\end{equation*}
	where the boxed $0$ is the entry $\ubar a_n^{(0)}$.	Our ultimate goal is to describe the Kac diagram of $\xi_m$. To this end, we only need the values of the entries $\ubar b_1^{(1)}, \ubar a_1^{(0)}, \ldots, \ubar a_n^{(0)}$ ($n+1$ parts in total). It turns out that all these entries are $< m$, as the following lemma shows. Hence the for these entries $x \neq 0$, we have $\delta_{x,m\BZ} = 0$ by Lemma \ref{lem:ym}, and hence $m_x = c_x+1$.
	
	\begin{lemma}
		In the sequence 
		\begin{equation*}
			\underbrace{x_2-1, \ldots, x_2-1}_{c_{x_2-1}+1 \text{ copies}}, \ldots, 
			\underbrace{1,\ldots,1}_{c_1+1 \text{ copies}},
			\underbrace{0,\ldots,\boxed{0}}_{c_0^+ +1 \text{ copies}}
		\end{equation*}
		the number of parts is at least $n+1$. By the preceding lemma, all these entries are less than $m$.
	\end{lemma}
	
	Therefore
	\begin{align}
		\Big( \ubar b_1^{(1)}, \ubar a_1^{(0)}, \ldots, \boxed{\ubar a_n^{(0)}} \Big)
		&= \text{the last } n+1 \text{ parts in } 
		\Big( \ldots,
		\underbrace{2,\ldots,2}_{c_2+1 \text{ copies}}
		\underbrace{1,\ldots,1}_{c_1+1 \text{ copies}},
		\underbrace{0,\ldots,\boxed{0}}_{c_0^+ +1 \text{ copies}} \Big) \nonumber \\ 
		&= \text{the last } n+1 \text{ parts in } 
		\Big( \ldots,
		\underbrace{2,\ldots,2}_{c_2+1 \text{ copies}}
		\underbrace{1,\ldots,1}_{c_1+1 \text{ copies}},
		\underbrace{0,\ldots,\boxed{0}}_{c +1 \text{ copies}} \Big) \label{eqn:ubar-a}
	\end{align}
	(the last equality is because $c = c_0^+$).
	
	The remaining task is to determine the precise values of the entries $\ubar b_1^{(1)}, \ubar a_1^{(0)}, \ldots, \ubar a_n^{(0)}$, which amounts to finding the position of $b_1^{(n)}$ in the right side of (\ref{eqn:ubar-a}). The answer depends on the parity of $m$, and on whether $c$ is small (equal to $0$ or $1$). We present the details for the case $c \ge 1$ and $2c = \ell-1$. The other cases are similar and omitted.
	
	\begin{example}
		Suppose $c \ge 1$ (equivalently $m \le n-1$) and $2c = \ell-1$ (equivalently $(2c+1) m \le 2n-2$). By the formulas for $c_x$ in Corollary \ref{cor:cx}, the right hand side of (\ref{eqn:ubar-a}) contains the sequence
		\begin{equation*}
			\be := 
			\Big( 
			\underbrace{x_2-1,\ldots,x_2-1}_{2c+2 \text{ copies}}, \ldots,
			\underbrace{x_1,\ldots,x_1}_{2c+2 \text{ copies}}
			\underbrace{x_1-1,\ldots,x_1-1}_{2c+1 \text{ copies}}, \ldots
			\underbrace{1,\ldots,1}_{2c+1 \text{ copies}},
			\underbrace{0,\ldots,\boxed{0}}_{c +1 \text{ copies}} \Big).
		\end{equation*}
		The number of parts in this sequence is
		\begin{equation*}
			N := (x_2-x_1)(2c+2) + (x_1-1)(2c+1) + (c+1).
		\end{equation*}
		The previous lemma guarantees $N \ge n+1$, so the sequence
		\begin{equation*}
			\bd := 
			\Big( \ubar b_1^{(1)}, \ubar a_1^{(0)}, \ldots, \boxed{\ubar a_n^{(0)}} \Big)
		\end{equation*}
		is contained in $\be$. Moreover, it can be shown that the number of parts in the subsequence
		\begin{equation*}
			\mathbf{f} := 
			\Big(
			\underbrace{x_1-1,\ldots,x_1-1}_{2c+1 \text{ copies}}, \ldots
			\underbrace{1,\ldots,1}_{2c+1 \text{ copies}},
			\underbrace{0,\ldots,\boxed{0}}_{c +1 \text{ copies}} \Big)
		\end{equation*}
		is $\le n+1$. Hence the sequence $\bd$ contains $\mathbf{f}$. 
		
		Let us discuss various possibilities for the labels on the $0$-th and the $1$st nodes on the Kac diagram. Write
		\begin{equation*}
			r := N - (n+1)
			= (c+1) (2n-(2c+1)m) - c - 2 \ge 0.
		\end{equation*}
		This is the number of parts left if we remove $\bd$ from $\be$.
		\begin{itemize}
			\item If $r \in -1 + (2c+2)\BZ_{>0}$, then 
			\begin{equation*}
				\bd = \Big( y, y-1, y-1, \ldots, 0, \boxed{0} \Big)
			\end{equation*}
			for some integer $y$, and so $\langle \check \alpha_0, \ubar \ba \rangle = \ubar b_1^{(1)} - \ubar a_1^{(0)} = 1$ and $\langle \check \alpha_0, \ubar \ba \rangle = \ubar a_1^{(0)} - \ubar a_2^{(0)} = 0$. Hence the Kac diagram is of the form
			\begin{equation*}
				\begin{dynkinDiagram}[%
					labels={1,0,0,0,1,0,0,1,0,0,1,0,0,1,0,0,1,0,0,0,0}, 
					extended, 
					edge length = .5cm
					] 
					D{**...***...**...*...***...**...*...***...***}
					\dynkinBrace*[D_{?}]04
					\dynkinBrace*[A_{2c+2}]57
					\dynkinBrace*[A_{2c+2}]8{10}
					\dynkinBrace*[A_{2c+1}]{11}{13}
					\dynkinBrace*[A_{2c+1}]{14}{16}
					\dynkinBrace*[D_{c+1}]{17}{19}
				\end{dynkinDiagram}.
			\end{equation*}
			Here the subdiagram
			\begin{equation*}
				\begin{dynkinDiagram}[%
					labels={0,0,1,0,0,1,0,0,0,0}, 
					edge length = .5cm
					] 
					D{*...**...*...***...***}
					\dynkinBrace*[A_{2c+1}]{1}{3}
					\dynkinBrace*[A_{2c+1}]{4}{6}
					\dynkinBrace*[D_{c+1}]{7}{9}
				\end{dynkinDiagram}
			\end{equation*}
			is calculated from the subsequence $\mathbf{f}$ of $\bd$, and the number of $A_{2c+1}$'s is the same as the length of the interval $[x_1-1,1]$ which equals $x_1-1 = (c+1)m-n$. It is straightforward to find that there are $n-(2c+1) \frac {m-1}2 -c-2$ copies of $A_{2c+2}$, and the leftover part forms a $D_{2c+3}$.			
			
			The condition $r \in -1 + (2c+2)\BZ_{>0}$ in particular implies $r + 1 \equiv 0 \modulo 2(c+1)$, which can be simplified to
			\begin{equation*}
				(c+1)((2c+1)m-1) \equiv 0 \modulo 2(c+1).
			\end{equation*}
			This forces $(2c+1)m-1$ to be even, which forces $m$ to be odd. Conversely, if $m$ is odd, then
			\begin{equation*}
				r+1 = (c+1) (2n-(2c+1)m) - c - 1
				= (c+1) (2n-(2c+1)m-1)
			\end{equation*}
			where the second factor $(2n-(2c+1)m-1)$ is even and positive (positivity follows from $(2c+1)m \le 2n-2$ which follows from $2c = \ell-1$). Hence $r+1 \in (2c+2)\BZ_{>0}$. Therefore the condition $r \in -1 + (2c+2)\BZ_{>0}$ is equivalent to $m$ being odd.
			
			\item If $r \in -2 + (2c+2)\BZ_{>0}$, then 
			\begin{equation*}
				\bd = \Big( y, y, y-1, \ldots, 0, \boxed{0} \Big)
			\end{equation*}
			for some $y$, and so $\langle \check \alpha_0, \ubar \ba \rangle = \langle \check \alpha_1, \ubar \ba \rangle = 1$. However, the condition $r \in -2 + (2c+2)\BZ_{>0}$ implies $r+2 \equiv 0 \modulo 2(c+1)$, which simplifies to
			\begin{equation*}
				(c+1) (m+1) \equiv 1 \modulo 2(c+1).
			\end{equation*}
			Further reducing modulo $c+1$, we get $0 \equiv 1 \modulo (c+1)$. Since we have assumed $c \ge 1$, this is impossible. Therefore this case cannot happen.
			
			\item If $r \not\in -1 + (2c+2)\BZ_{>0}$ or $-2 + (2c+2)\BZ_{>0}$, then 
			\begin{equation*}
				\bd = \Big( y, y, y, \ldots, 0, \boxed{0} \Big)
			\end{equation*}
			for $y$ (i.e. the first three parts of $\bd$ are the same), and hence $\langle \check \alpha_0, \ubar \ba \rangle = \langle \check \alpha_1, \ubar \ba \rangle =0$. So the Kac diagram has the form
			\begin{equation*}
				\begin{dynkinDiagram}[%
					labels={0,0,0,0,1,0,0,1,0,0,1,0,0,1,0,0,1,0,0,0,0}, 
					extended, 
					edge length = .5cm
					] 
					D{**...***...**...*...***...**...*...***...***}
					\dynkinBrace*[D_{?}]04
					\dynkinBrace*[A_{2c+2}]57
					\dynkinBrace*[A_{2c+2}]8{10}
					\dynkinBrace*[A_{2c+1}]{11}{13}
					\dynkinBrace*[A_{2c+1}]{14}{16}
					\dynkinBrace*[D_{c+1}]{17}{19}
				\end{dynkinDiagram}
			\end{equation*}
			Similar to the first case, there are $(c+1)m-n$ copies of $A_{2c+1}$. It is straightforward to find that there are $n-(2c+1) \frac m2 -1$ copies of $A_{2c+2}$, and the leftover part forms a $D_{c+2}$.
			
			From the discussions of the previous two cases, we see that the condition $r \not\in -1 + (2c+2)\BZ_{>0}$ or $-2 + (2c+2)\BZ_{>0}$ is equivalent to $m$ being even.
		\end{itemize}
		
		Thus, we obtain the following two lines of Table \ref{tbl:Kac-D}:
		\begin{center}
			\begin{tabular}{c|c|c}
				condition on $m$ & Kac diagram of $\xi_m$ & type of $W_m$ \\ \hline
				$\substack{%
					m \le n-1\;(\text{i.e. } c\ge1)\\%
					m \text{ odd}\\%
					(2c+1)m \le 2n-2}$
				& 
				$\substack{%
					\begin{dynkinDiagram}[%
						labels={1,0,0,0,1,0,0,1,0,0,1,0,0,1,0,0,1,0,0,0,0}, 
						extended, 
						] 
						D{**...***...**...*...***...**...*...***...***}
						\dynkinBrace*[D_{2c+3}]04
						\dynkinBrace*[A_{2c+2}]57
						\dynkinBrace*[A_{2c+2}]8{10}
						\dynkinBrace*[A_{2c+1}]{11}{13}
						\dynkinBrace*[A_{2c+1}]{14}{16}
						\dynkinBrace*[D_{c+1}]{17}{19}
					\end{dynkinDiagram}\\%
					n-(2c+1)\frac{m-1}2 - c-2 \text{ copies of } A_{2c+2},\quad
					(c+1)m-n \text{ copies of } A_{2c+1}}$
				& $\substack{%
					D_{c+1} + \big((c+1)m-n\big) A_{2c}\\%
					+ \big(n-(2c+1)\frac{m-1}2 -c-1 \big) A_{2c+1}
				}$\\ \hline
				$\substack{%
					m \le n-1\;(\text{i.e. } c\ge1)\\%
					m \text{ even}\\%
					(2c+1)m \le 2n-2}$
				& 
				$\substack{%
					\begin{dynkinDiagram}[%
						labels={0,0,0,0,1,0,0,1,0,0,1,0,0,1,0,0,1,0,0,0,0}, 
						extended, 
						] 
						D{**...***...**...*...***...**...*...***...***}
						\dynkinBrace*[D_{c+2}]04
						\dynkinBrace*[A_{2c+2}]57
						\dynkinBrace*[A_{2c+2}]8{10}
						\dynkinBrace*[A_{2c+1}]{11}{13}
						\dynkinBrace*[A_{2c+1}]{14}{16}
						\dynkinBrace*[D_{c+1}]{17}{19}
					\end{dynkinDiagram}\\%
					n-(2c+1)\frac{m}2 - 1 \text{ copies of } A_{2c+2},\quad
					(c+1)m-n \text{ copies of } A_{2c+1}}$
				& $\substack{%
					2 D_{c+1} + \big((c+1)m-n\big) A_{2c}\\%
					+ \big(n-(2c+1)\frac{m}2 - 1\big) A_{2c+1}
				}$
			\end{tabular}
		\end{center}
	\end{example}
	
	We end by summarizing the how various conditions in Table \ref{tbl:Kac-D} come up in the above calculations.
	\begin{itemize}
		\item The conditions $(2c+1) m > 2n-2$ and $(2c+1) m \le 2n-2$ correspond to whether $c = c_0^+$ is equal to $\ell$ or $\ell-1$, and the condition $m \mid 2n-1$ and $m \not\mid 2n-1$ affects the multiplicities of various numbers in $\xi_m$ (see Corollary \ref{cor:cx}).
		
		\item The parity of $m$ affect the labels on the $0$-th and the $1$st nodes of the Kac diagram.
		
		\item Special care may be needed for small values of $c$. This results in the conditions $c=0$, $c=1$, and $c\ge 1$. 
	\end{itemize}
\end{clause}

\begin{clause}[Proof of Theorem \ref{thm:O(m)-n-cells} in type $D_n$]
	The Kac diagrams of $\xi_m$ and the types of $W_m$ have been collected in Table \ref{tbl:Kac-D}. It remains to verify that $j_{W_m}^W (\sgn)$ corresponds to $\check \BO(m)$ under the Springer correspondence. There is only one case where $m \not\in \im \check \cl$, i.e. when $m>n-1$ (i.e. $c=0$) and $m$ is odd. The largest number $m_0 \in \im \check \cl$ that is $\le m$ is $m_0 = m-1$. In all other cases, $m \in \im \check \cl$.
	
	Recall 
	\begin{itemize}
		\item $c$ is the largest integer so that $cm \le n-1$, that is $cm \le n-1 < (c+1)m$.
		\item $b$ is the largest integer so that $bm \le 2n$, that is $bm \le 2n < (b+1)m$.
		\item $k$ is the largest integer so that $(m+1) + 2km \le 2n$, that is $(m+1) + 2km \le 2n < (m+1) + (2k+2)m$.
	\end{itemize}
	
	We first deal with the cases where a special node of the Kac diagram is labeled by $1$. In this case $W_m$ is special, it corresponds to a parabolic $\check P_m = \check L_m \check U_m \subset \check G_m$, and $\check \BO_m = \Ind_{\check L_m}^{\check G} \{0\}$ by Lemma \ref{lem:L-bij-as-jind}(2). Recall from \cite[\textsection 2.3]{LMBM:unip} that a Levi $L$ in $\SO(2n)$ is of the form $L = \GL(a_1) \times \cdots \times \GL(a_t) \times \SO(2a_0)$ where the $a_i$'s form a partition of $n$. Given orbits $\BO_i$ of $\GL(a_i)$ and $\BO_0$ of $\SO(2a_0)$, write $p^i$ and $p^0$ for the corresponding partitions. Then the induction of $\prod_i \BO_i$ from $L$ to $\SO(2n)$ is given by the partition $B(p^0 + 2 \sum_{j=1}^t p^j)$ where the sum of partitions is defined in Notation \ref{notn:partition}, and $B(p)$ takes the unique maximal orthogonal partition dominated by $p$. Here a partition is orthogonal if each even part appears with even multiplicity.
	\begin{itemize}
		\item $m=1$. In this case $W_m = D_n = W$ and $\check L_m = \check G$, and $\check \BO_m = \Ind_{\check G}^{\check G} \{0\} = \{0\} = \check \BO(1)$.
		
		\item $m>n-1$ (i.e. $c=0$), $m$ odd. From the Kac diagram, we have $W_m = (n- \frac{m+1}2) A_1$ and $\check L_m = \GL(2)^{n-\frac{m+1}2} \times \GL(1)^{m-n+1}$. So
		\begin{align*}
			\Ind_{\check L_m}^{\check G} \{0\}
			&= B \Big( 2(n- \tfrac{m+1}2) \cdot (1,1) + 2(m-n+1) \cdot (1) \Big)\\
			&= B \Big( (m+1, 2n-m-1) \Big) = (m, 2n-m).
		\end{align*}
		
		We have two subcases. Suppose $m \in \im \cl$. Since $\im \cl = [1,n] \cup (2\BZ \cap [n+1,2n-2])$ by \ref{values:D}, we must have $m=n$, which implies $b=2$, i.e. $q(m) = (m,m) = (n,n)$. This agrees with $\Ind_{\check L_m}^{\check G} \{0\}$. Now suppose $m \not\in \im \cl$. Then $m_0 = m-1$, and $\Ind_{\check L_m}^{\check G}\{0\} = (m,2n-m) = q(m-1) = q(m_0)$, as required.		
		
		\item $m>n-1$, (i.e. $c=0$), $m$ even. In this case $W_m = (n- \frac m2 -1) A_1$ and $\check L_m = \GL(2)^{n-\frac m2-1} \times \GL(1)^{m-n+2}$. Hence
		\begin{align*}
			\Ind_{\check L_m}^{\check G} \{0\}
			&= B \Big( 2(n-\tfrac m2-1) \cdot (1,1) + 2(m-n+2) \cdot (1) \Big)\\
			&= B \Big( (m+2, 2n-m-2) \Big)
			= (m+1, 2n-m-1).
		\end{align*}
		This equals $q(m) = (m+1, v)$ (we have $k=0$).
		
		\item $m \le n-1$ (i.e. $c \ge 1$), $m$ odd, $(2c+1) m \le 2n-2$. In this case $W_m = D_{c+1} + ((c+1)m-n) A_{2c} + (n - (2c+1) \frac{m-1}2 - c-1) A_{2c+1}$ and $\check L_m = \GL(2c+1)^{(c+1)m-n} \times \GL(2c+2)^{n-(2c+1) \frac{m-1}2 -c-1} \times \SO(2c+2)$. So
		\begin{align*}
			\Ind_{\check L_m}^{\check G} \{0\}
			&= B \Big( (1^{2c+2}) + 2((c+1)m-n) \cdot (1^{2c+1}) + 2(n-(2c+1) \tfrac{m-1}2 -c-1) \cdot (1^{2c+2}) \Big)\\
			&= B \Big( (m^{2c+1}, 2n-(2c+1)m) \Big)
			= (m^{2c+1}, 2n-(2c+1)m).
		\end{align*}
		We next find the value of $b$. From the condition $(2c+1)m \le 2n-2 < 2n$ and the condition $bm \le 2n$, we see that $2c+1 \le b$. If $2c+1 <b$, then $2c+2 \le b$, which implies
		\begin{equation*}
			2n-2 < (2c+2)m \le bm \le 2n
		\end{equation*}
		where the first and the last inequality are from the definition of $c$ and $b$, respectively. Since $(2c+2)m$ is even, this forces $2c+2=b$ and $(m^{2c+1}, 2n-(2c+1)m) = (m^{2c+2}) = (m^b) = q(m)$ as required. If $2c+1=b$, then $(m^{2c+1}, 2n-(2c+1)m) = (m^b, v) = q(m)$.
		
		\item $m \le n-1$ (i.e. $c \ge 1$), $m > 1$ odd, $m \not\mid 2n-1$, $(2c+1)m> 2n-2$. In this case $W_m = D_{c+1} + (n-cm-1) A_{2c} + ((2c+1) \frac{m+1}2-n-c) A_{2c-1}$ and $\check L_m = \GL(2c+1)^{n-cm-1} \times \GL(2c)^{(2c+1) \frac{m+1}2-n-c} \times \SO(2c+2)$. So
		\begin{align*}
			\Ind_{\check L_m}^{\check G} \{0\}
			&= B \Big( (1^{2c+2}) + 2(n-cm-1) \cdot (1^{2c+1}) + 2((2c+1) \tfrac{m+1}2-n-c) \cdot (1^{2c}) \Big)\\
			&= B \Big( (m^{2c}, 2n-2cm-1,1) \Big)
			= (m^{2c}, 2n-2cm-1,1).
		\end{align*}
		We next find the value of $b$. The condition $(2c+1)m > 2n-2$ implies $2n-1 \le (2c+1)m$. If $(2c+1)m \le 2n$, then since $(2c+1)m$ is odd, we must have $(2c+1)m = 2n-1$, which implies $m \mid 2n-1$, contradicting our condition on $m$. So we must have $2n < (2c+1)m$ and hence $b <2c+1$. On the other hand the definition of $c$ reads $cm \le n-1$ which implies $2c \le 2n-2 < 2n$, and hence $2c \le b$. Therefore $b = 2c$ an $q(m) = (m^b, v, 1) = (m^{2c},v,1)$. This partition equals that of $\Ind_{\check L_m}^{\check G} \{0\}$, as required.
		
		\item $m \le n-1$ (i.e. $c \ge 1$), $m > 1$ odd, $m \mid 2n-1$, $(2c+1)m> 2n-2$. These conditions imply $2n-1 = (2c+1)n$, and hence $b = 2c+1$. We have $W_m = D_{c+1} + (n-cm-1) A_{2c}$ and $\check L_m = \GL(2c +1)^{n-cm-1} \times \SO(2c+2)$. Hence
		\begin{align*}
			\Ind_{\check L_m}^{\check G} \{0\}
			&= B \Big( (1^{2c+2}) + 2(n-cm-1) \cdot (1^{2c+1}) \Big)\\
			&= B \Big( (m^{2c+1},1) \Big)
			= (m^{2c+1},1) 
			= (m^b,1) = q(m).
		\end{align*}
		
		\item $m \le n-1 < 2m$ (i.e. $c=1$), $m$ even, $3m = (2c+1)m > 2n-2$. In this case $W_m = (\frac{3m}2-n+2) A_1 + (n-m-1) A_2$ and $\check L_m = \GL(2)^{\frac{3m}2-n} \times \GL(3)^{n-m-1} \times \GL(1) \times \SO(4)$. So
		\begin{align*}
			\Ind_{\check L_m}^{\check G} \{0\}
			&= B \Big( (1^4) + 2(\tfrac{3m}2-n) \cdot (1^2) + 2(n-m-1) \cdot (1^3) + 2 \cdot (1) \Big)\\
			&= B \Big( (m+1,m-1,2n-2m-1,1) \Big)
			= (m+1,m-1,2n-2m-1,1) 
		\end{align*}
		The condition $n-1<2m$ on $m$ implies $2n < 3m+2$. Since both sides of this inequality are even, $2n \le 3m = (m+1) + 2m-1$. So by the definition of $k$, we must have $k=0$. So $(m+1,m-1,2n-2m-1,1) = (m+1,m^{2k},m-1,v,1) = q(m)$.
	\end{itemize}
	
	In remaining two cases, all special nodes in the Kac diagram of $\xi_m$ are labeled by $0$, and we need to show that the representation $j_{W_m}^W \sgn$ corresponds to the pair $(\check \BO(m), \ubar \BC)$ under the Springer correspondence. Instead of computing this directly, we use the following commutative diagram \cite[Theorem 14]{Chua}
	\begin{equation*}
		\begin{tikzcd}
			& {\ubar W}\\
			{\ubar{\check \cN}_m} \ar[ur, "\KL_{\check G}^{\check \bP_m}"] 
			&& {\ubar{\check \cN}} \ar[ul, "\KL"']\\
			\Irr\Rep_{sp}(W_m) \ar[u, "\operatorname{Spr}"] \ar[rr, "j_{W_m}^W"]
			&& \Irr\Rep(W) \ar[u, "\operatorname{Spr}"']
		\end{tikzcd}.
	\end{equation*}
	Here we are writing $\check \bP_m \subset \check G(F)$ for the standard parahoric subgroup corresponding to $W_m \subset W_{aff}$, $\check \bL_m$ for its Levi factor containing the Cartan $\check H$, $\check \bU_m$ for the unipotent part of $\check \bP_m$, and $\check \cN_m$ for the nilpotent cone of $\check \fl_m$. The set $\Irr\Rep_{sp}(W_m)$ is the set of special representations of $W_m$. The maps $\operatorname{Spr}$ send a representation $E$ to $\BO$ if $E$ corresponds to a pair $(\BO, \tau)$ under Springer correspondence. The map $\KL_{\check G}^{\check \bP_m}$ sends a nilpotent orbit $\BO_m \in \ubar{\check \cN}_m$ to the unique conjugacy class $[w] \in \ubar W$ such that for any $e \in \BO_m$, there is a dense open set $U \subset e + \check \fu_m$ ($\check \fu_m$ is the Lie algebra of $\check \bU_m$) so that any $\gamma \in U$ is regular semisimple topologically nilpotent of type $[w]$ (see \cite[Lemma 10]{Chua}). Recall that we aim to check whether $\operatorname{Spr} j_{W_m}^W \sgn = (\check \BO(m), \ubar \BC)$. In fact, it is enough to check that the underlying orbit of $\operatorname{Spr} j_{W_m}^W \sgn$ is equal to $\check \BO(m)$ since Lemma \ref{lem:L-bij-as-jind}(1) guarantees that the local system on the orbit is trivial. Because of the commutative diagram and the fact that both $\KL_{\check G}^{\check \bP_m}$ and $\KL$ are injective, it is enough to check whether $\KL_{\check G}^{\check \bP_m} \operatorname{Spr} \sgn = \KL_{\check G}^{\check \bP_m} \{0\}$ is equal to $\KL(\check \BO(m))$.
	
	To compute $\KL_{\check G}^{\check \bP_m}$ and $\KL$, we use the formula given in \cite[Theorem 8.11]{Spaltenstein:order}. First, for any maximal standard parahoric subgroup $\check \bP_{max} \subset \check G(F)$ containing $\check \bP_m$, we have $\KL_{\check G}^{\check \bP_m} = \KL_{\check G}^{\check \bP_{max}} \Ind_{\check \bL_m}^{\check \bL_{max}}$ where $\Ind$ denotes the Lusztig-Spaltenstein induction of nilpotent orbits. Also, $\KL = \KL_{\check G}^{\check G(\cO)}$. The Lusztig-Spaltenstein induction can easily be computed in terms of partitions (see the beginning of this proof for the rule). So it remains to explain the procedure for computing $\KL_{\check G}^{\check \bP_{max}}$. Note that the Levi $\check \bL_{max}$ has at most two simple factors $\check \bL_1$ and $\check \bL_2$. Hence $\KL_{\check G}^{\check \bP_{max}}$ is a map
	\begin{equation*}
		\KL_{\check G}^{\check \bP_{max}}: \ubar \cN_{\check \bL_1} \times \ubar \cN_{\check \bL_2} \aro \ubar W.
	\end{equation*}
	Write $\cP$ for the set of partitions, and $\cP^+$ for the set of orthogonal partitions. In the cases considered below, both $\check \bL_1$ and $\check \bL_2$ are of type $D$, and orbits of $\bL_1$ and $\bL_2$ are not both very even. So we may view $\ubar \cN_{\check \bL_1}$ and $\ubar \cN_{\check \bL_2}$ are subsets of $\cP^+$. For a partition $p$, write $C(p)$ for the unique maximal symplectic partition (i.e. a partition where each odd part appears with even multiplicity) dominated by $p$. Then the map $\KL_{\check G}^{\check \bP_{max}}$ factors as
	\begin{multline*}
		\KL_{\check G}^{\check \bP_{max}}:
		\big\{ \BO_1 \times \BO_2 \in \ubar \cN_{\check \bL_1} \times \ubar \cN_{\check \bL_2} \mid \BO_1 \text{ and } \BO_2 \text{ are not both very even} \big\}
		\subset \cP^+ \times \cP^+
		\xrightarrow{\eta} \cP
		\xrightarrow{\varpi} \ubar W
	\end{multline*}
	\cite[4.5]{Spaltenstein:order}. Here 
	\begin{equation*}
		\eta(p_1,p_2) = C(p_1+p_2)
	\end{equation*}
	and $\varpi$ is some map we don't need.
	
	We are ready to verify the remaining cases.
	\begin{itemize}
		\item $m \le n-1$, (i.e. $c \ge 1$), $m$ even, $(2c+1)m \le 2n-2$. The Kac diagram of $\xi_m$ is 
		\begin{equation*}
			\begin{dynkinDiagram}[%
				labels={0,0,0,0,1,0,0,1,0,0,1,0,0,1,0,0,1,0,0,0,0}, 
				extended, 
				edge length = .5cm
				] 
				D{**...***...**...*...*x*...**...*...***...***}
				\dynkinBrace*[D_{c+2}]04
				\dynkinBrace*[A_{2c+2}]57
				\dynkinBrace*[A_{2c+2}]8{10}
				\dynkinBrace*[A_{2c+1}]{11}{13}
				\dynkinBrace*[A_{2c+1}]{14}{16}
				\dynkinBrace*[D_{c+1}]{17}{19}
			\end{dynkinDiagram}
		\end{equation*}
		with $n-(2c+1)\frac{m}2 - 1$ copies of $A_{2c+2}$ and $(c+1)m-n$ copies of $A_{2c+1}$. 
		
		We first compute $\KL_{\check G}^{\check \bP_m} \{0\}$. We let $\check \bP_{max}$ be the standard parahoric subgroup obtained by removing the crossed node, and write its Levi factor as $\check \bL_1 \times \check \bL_2$ where $\check \bL_1$ (resp. $\check \bL_2$) corresponds to the nodes to the left (resp. right) of the crossed node. We are being sloppy on the center of the Levis since they do not matter when calculating nilpotent orbits. Let $\check \bL_i' \subset \check \bL_i$ denote the Levi corresponding to those nodes marked by $0$. Then
		\begin{equation*}
			\check \bL_1' 
			= \GL(2c+2)^{n-(2c+1)\frac m2-1} \times \SO(2c+2)
		\end{equation*}
		\begin{equation*}
			\check \bL_2'
			= \GL(2c+1)^{(c+1)m-n} \times \SO(2c+2).
		\end{equation*}
		Hence
		\begin{align*}
			\BO_1 &:= \Ind_{\check \bL_1'}^{\check \bL_1} \{0\}\\
			&= B \Big( (1^{2c+2}) + 2(n-(2c+1)\tfrac m2 -1) \cdot (1^{2c+2}) \Big)\\
			&= B \Big( ((2n-(2c+1)m-1)^{2c+2}) \Big)
			= ((2n-(2c+1)m-1)^{2c+2}),
		\end{align*}
		\begin{align*}
			\BO_2 &:= \Ind_{\check \bL_2'}^{\check \bL_2} \{0\}\\
			&= B \Big( (1^{2c+2}) + 2((c+1)m-n) \cdot (1^{2c+1}) \Big)\\
			&= B \Big( ( (2(c+2)m-2n+1)^{2c+1},1 ) \Big)
			= ( (2(c+2)m-2n+1)^{2c+1},1 ).
		\end{align*}
		Hence
		\begin{align*}
			\KL_{\check G}^{\check \bP_m} \{0\}
			&= \KL_{\check G}^{\check \bP_{max}}(\BO_1 \times \BO_2)\\
			&= \varpi \eta \Big( ((2n-(2c+1)m-1)^{2c+2}), ( (2(c+2)m-2n+1)^{2c+1},1 ) \Big)\\
			&= \varpi C\Big( ((2n-(2c+1)m-1)^{2c+2}) + ( (2(c+2)m-2n+1)^{2c+1},1 ) \Big) \\
			&= \varpi C \Big( (m^{2c+1}, 2n-(2c+1)m) \Big)
			= \varpi (m^{2c+1}, 2n-(2c+1)m) .
		\end{align*}
		
		Next we compute $\KL(\check \BO(m))$. We first find $k$. First, the condition $(2c+1)m \le 2n-2$ can be rewritten as $(m+1) + 2cm + 1 \le 2n$ which, by the definition of $k$, implies $k \ge c$. Suppose $k \ge c+1$. By the definition of $c$, $n-1 < (c+1)m$, which implies $2n-1 < (2c+2)m +1$. Adding $m$ to both sides we get
		\begin{equation*}
			2n-1+m < (2c+2)m +1+m = (2c+3)m +1 = (m+1) + 2(c+1)m \le (m+1) +2km \le 2n.
		\end{equation*}
		This forces $1 < m$, which is absurd. Hence $k < c+1$ and $k = c$. Moreover we have
		\begin{equation*}
			(2k+2)m = (2c+2)m = 2(c+1)m > 2(n-1) = 2n-2.
		\end{equation*}
		Since both sides are even, it implies $(2k+2)m \ge 2n$. Therefore 
		\begin{equation*}
			q(m) = (m+1, m^{2k}, v) = (m+1,m^{2c}, v)
		\end{equation*}
		and hence
		\begin{equation*}
			\KL(\check \BO(m))
			= \varpi \eta (q(m))
			= \varpi C \Big( (m+1,m^{2c}, v) \Big)
			= \varpi (m^{2c+1}, v-1).
		\end{equation*}
		This equals $\KL_{\check G}^{\check \bP_m} \{0\}$, as required.
		
		\item $2m \le n-1$ (i.e. $c > 1$), $m$ even, $(2c+1) m > 2n-2$. The Kac diagram of $\xi_m$ is 
		
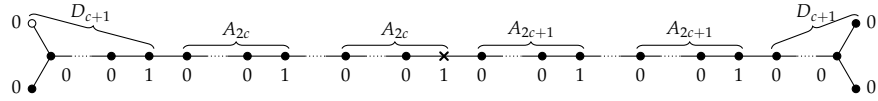
\begin{equation*}
			\begin{dynkinDiagram}[%
				labels={0,0,0,0,1,0,0,1,0,0,1,0,0,1,0,0,1,0,0,0,0}, 
				extended, 
				edge length = .5cm
				] 
				D{**...***...**...*...*x*...**...*...***...***}
				\dynkinBrace*[D_{c+1}]04
				\dynkinBrace*[A_{2c}]57
				\dynkinBrace*[A_{2c}]8{10}
				\dynkinBrace*[A_{2c+1}]{11}{13}
				\dynkinBrace*[A_{2c+1}]{14}{16}
				\dynkinBrace*[D_{c+1}]{17}{19}
			\end{dynkinDiagram}
		\end{equation*}
		with $(2c+1)\frac{m}2-n$ copies of $A_{2c}$ and
		$n-cm-1$ copies of $A_{2c+1}$.
		
		We first compute $\KL_{\check G}^{\check \bP_m} \{0\}$. Let $\check \bP_{max}$ be the standard parahoric subgroup obtained by removing the crossed node, and write the semisimple part of its Levi factor $\check \bL_{max} = \check \bL_1 \times \check \bL_2$ where $\check \bL_1$ (resp. $\check \bL_2$) corresponds to the nodes to the left (resp. right) of the crossed node. Let $\check \bL_i' \subset \check \bL_i$ denote the Levi corresponding to those nodes marked by $0$. Then
		\begin{equation*}
			\check \bL_1' 
			= \GL(2c)^{(2c+1)\frac m2 -n} \times \SO(2c),
		\end{equation*}
		\begin{equation*}
			\check \bL_2'
			= \GL(2c+1)^{n-cm-1} \times \SO(2c+2).
		\end{equation*}
		Hence 
		\begin{align*}
			\BO_1 &:= \Ind_{\check \bL_1'}^{\check \bL_1} \{0\}\\
			&= B \Big( (1^{2c}) + 2((2c+1) \tfrac m2 -n) \cdot (1^{2c}) \Big)\\
			&= B \Big( ( ((2c+1)m - 2n+1)^{2c} ) \Big)
			=  ( ((2c+1)m - 2n+1)^{2c} ),
		\end{align*}
		\begin{align*}
			\BO_2 &:= \Ind_{\check \bL_2'}^{\check \bL_2} \{0\}\\
			&= B \Big( (1^{2c+2}) + 2(n-cm-1) \cdot (1^{2c+1}) \Big)\\
			&= B \Big( ( (2n-2cm-1)^{2c+1}, 1 ) \Big)
			= ( (2n-2cm-1)^{2c+1}, 1 ).
		\end{align*}
		Therefore
		\begin{align*}
			\KL_{\check G}^{\check \bP_m} \{0\}
			&= \KL_{\check G}^{\check \bP_{max}}(\BO_1 \times \BO_2)\\
			&= \varpi \eta \Big( ( ((2c+1)m - 2n+1)^{2c} ), ( (2n-2cm-1)^{2c+1}, 1 ) \Big)\\
			&= \varpi C \Big( ((2c+1)m - 2n+1)^{2c} )+ ( (2n-2cm-1)^{2c+1}, 1 ) \Big)\\
			&= \varpi C \Big( (m^{2c}, 2n-2cm-1,1) \Big)\\
			&= 
			\begin{cases}
				\varpi (m^{2c}, 2n-2cm-2, 2) & \text{if } cm < n-1;\\
				\varpi (m^{2c}, 1^2) & \text{if } cm = n-1.
			\end{cases}
		\end{align*}
		
		Next we compute $\KL(\check \BO(m))$. We first find $k$. By our conditions on $c$ and $m$, 
		\begin{equation*}
			2cm \le 2n-2 < (2c+1) m
		\end{equation*}
		which implies
		\begin{multline*}
			(m+1) + 2(c-1)m < (m+1) + 2(c-1)m + m+1 = 2cm+2 \le 2n\\
			< (2c+1)m + 2 = (m+1) + 2cm +1.
		\end{multline*}
		The left half reads $(m+1) + 2(c-1)m < 2n$. By definition of $k$, this implies $c-1 \le k$. The right half reads $2n < (m+1) + 2cm +1$. Since both sides are even, we have $2n \le (m+1) + 2cm -1$ which implies $k < c$. So $k = c-1$, and 
		\begin{equation*}
			(2k+2) m = 2cm \le 2n-2 < 2n.
		\end{equation*} 
		Hence
		\begin{equation*}
			q(m) = (m+1, m^{2k}, m-1, v, 1) = (m+1, m^{2c-2}, m-1, v, 1)
		\end{equation*}
		and
		\begin{equation*}
			\KL(\check \BO(m))
			= \varpi\eta(q(m))
			= \varpi C \Big( (m+1, m^{2c-2}, m-1, v, 1) \Big)
			=
			\begin{cases}
				\varpi (m^{2c}, v-1,2) & \text{if } v \neq 1;\\
				\varpi (m^{2c}, 1^2) & \text{if } v = 1.
			\end{cases}
		\end{equation*}
		Note that the condition $v \neq 1$ (resp. $v =1$) is equivalent to $cm < n-1$ (resp. $cm = n-1$). Hence $\KL(\check \BO(m))$ equals $\KL_{\check G}^{\check \bP_m} \{0\}$ as required. 
	\end{itemize}
	This completes the verification in all cases.
\end{clause}

\printbibliography
\listoftables
\listoffigures

\end{document}